\documentclass[12pt, reqno]{amsart} 
\edef\restoreparindent{\parindent=\the\parindent\relax}
\usepackage{parskip}
\restoreparindent
\usepackage{floatrow}
\usepackage{amssymb, amsfonts, amsthm, amsmath, amscd}
\usepackage{graphicx}   
\usepackage[latin2]{inputenc}
\usepackage{t1enc}
\usepackage[mathscr]{eucal}
\usepackage{indentfirst}
\usepackage{pict2e}
\usepackage{epic}
\numberwithin{equation}{section}
\usepackage{stmaryrd}
\usepackage{scalerel,stackengine}
\usepackage[margin=2.5cm]{geometry}
\usepackage{mathtools}
\usepackage[colorlinks = true,
            linkcolor = blue,
            urlcolor  = blue,
            citecolor = blue,
            anchorcolor = blue]{hyperref}
\usepackage[titletoc]{appendix}
\usepackage[T1]{fontenc}
\usepackage[numbers,square]{natbib}
\usepackage{adjustbox}
\usepackage{dsfont}
\usepackage{times}
\usepackage{enumerate}
\usepackage{setspace}
\usepackage{leftidx}
\usepackage{multirow}

\usepackage[most]{tcolorbox}
\usepackage{pagecolor}
\definecolor{bananamania}{rgb}{0.98, 0.91, 0.71}
\usepackage{tikz-cd}
\usepackage{pgf,tikz}
\usetikzlibrary{math}
\usetikzlibrary{arrows,decorations.markings}
\usetikzlibrary{patterns,intersections}
\usetikzlibrary{calc}
\usetikzlibrary{decorations.pathreplacing}

\newenvironment{myproof}[2]{\paragraph{\textit{Proof of {#1} }{#2}.}}{\hfill$\square$}

\theoremstyle{plain}
\newtheorem{theorem}{Theorem}[section]
\newtheorem{lemma}[theorem]{Lemma}
\newtheorem{corollary}[theorem]{Corollary}
\newtheorem{proposition}[theorem]{Proposition}

\theoremstyle{definition}
\newtheorem{definition}[theorem]{Definition}

\newtheorem{remark}[theorem]{Remark}

\newtheorem{example}[theorem]{Example}

\makeatletter
\newcommand*\bulletsmall{\mathpalette\bulletsmall@{.5}}
\newcommand*\bulletsmall@[2]{\mathbin{\vcenter{\hbox{\scalebox{#2}{$\m@th#1\bullet$}}}}}
\makeatother

\def\CC{\mathbb{C}}
\def\AA{\mathbb{A}}

\def\RR{\mathbb{R}}
\def\ZZ{\mathbb{Z}}
\def\PP{\mathbb{P}}

\def\a{\alpha}
\def\b{\beta}

\def\s{\sigma}

\def\G{\Gamma}
\def\w{\omega}
\def\ra{\rightarrow}
\def\tt{\theta}

\def\O{\Omega}
\def\gg{\mathfrak{g}}

\def\hh{\mathfrak{t}}
\def\ol{\overline}
\def\bl{\bullet}

\def\ecc{q_i^{-1}|~i\in I\setminus I_P}
\def\eccc{q_i^{\pm 1}|~i\in I\setminus I_P}

\DeclareMathOperator{\fof}{Frac}
\DeclareMathOperator{\spec}{Spec}
\DeclareMathOperator{\sym}{Sym}
\DeclareMathOperator{\rep}{Rep}

\DeclareMathOperator{\pt}{pt}

\DeclareMathOperator{\lie}{Lie}

\DeclareMathOperator{\pr}{pr} 
\DeclareMathOperator{\ev}{ev}
\DeclareMathOperator{\id}{id}
\DeclareMathOperator{\im}{Im}
\DeclareMathOperator{\pd}{PD}

\DeclareMathOperator{\coker}{coker}
\DeclareMathOperator{\adjoint}{Ad}

\DeclareMathOperator{\jac}{Jac}

\def\rhod{\rho^{\vee}}
\def\ad{\alpha^{\vee}}
\def\Q{Q^{\vee}}

\def\Bd{B^{\vee}}
\def\Gd{G^{\vee}}
\def\Td{T^{\vee}}
\def\Bmd{B_-^{\vee}}
\def\Pmd{P_-^{\vee}}
\def\UPd{\mathcal{U}'}
\def\Bed{B_{e^T}^{\vee}}

\def\nd{\mathfrak{n}^{\vee}_-}

\def\RP{\mathcal{R}_P}
\def\UP{\mathcal{U}_P}
\def\UUU{\mathcal{U}}
\def\OO{\mathcal{O}}
\def\KK{\mathcal{K}}
\def\BB{\mathcal{B}}

\def\PP{\mathbb{P}}

\def\FF{\mathcal{F}}
\def\LL{\mathcal{L}}
\def\ag{\mathcal{G}r}
\def\agl{\mathcal{G}r^{
\lambda}}
\def\agll{\mathcal{G}r^{\leqslant 
\lambda}}

\def\gm{\mathbb{G}_m}
\def\ga{\mathbb{G}_a}
\def\wl{wt_{\lambda}}

\def\wp{w_Pw_0}
\def\wpd{\dot{w}_P\dot{w}_0}
\def\wdd{\omega^{\vee}}

\def\GB{G^{\vee}/B^{\vee}_-}

\def\per{\mathcal{P}_{G(\mathcal{O})}(\mathcal{G}r)}
\def\co{Q^{\vee}}
\def\cop{Q^{\vee}_+}

\def\vp{\varphi}

\def\JJ{\mathcal{J}}

\def\MM{\overline{\mathcal{M}}}

\newcommand{\bk}[1]{\left[#1\right]}
\def\al{\ad}

\newcommand{\fib}[1]{\mathcal{E}_{#1}}

\newcommand{\ffib}[1]{\mathcal{E}_{#1}(G/P)}

\def\Ld{L^{\vee}}
\def\ZL{Z(L^{\vee})}
\def\Pd{P^{\vee}}
\def\Xd{X^{\vee}_P}
\def\Xdt{X^{\vee}_t}
\def\cchh{\chi}
\def\hp{\hslash}
\def\zgd{Z_{\hslash}(\mathfrak{g}^{\vee})}
\def\utd{U_{\hslash}(\mathfrak{t}^{\vee})}
\def\ubd{U_{\hslash}(\mathfrak{b}^{\vee})}
\def\und{U_{\hslash}(\mathfrak{n}^{\vee})}
\def\unmd{U_{\hslash}(\mathfrak{n}^{\vee}_-)}
\def\ugd{U_{\hslash}(\mathfrak{g}^{\vee})}
\def\ggd{\mathfrak{g}^{\vee}}
\def\nndm{\mathfrak{n}^{\vee}_-}
\def\nnd{\mathfrak{n}^{\vee}}
\def\ttd{\mathfrak{t}^{\vee}}
\def\bbd{\mathfrak{b}^{\vee}}
\def\Ud{U^{\vee}}
\def\Umd{U^{\vee}_-}
\newcommand{\dm}[1]{D_{\hslash,#1}}
\DeclareMathOperator{\mccccc}{mc}
\def\mctd{\mccccc_{T^{\vee}}}
\def\bries{G_0(\Xd,W,\pi,p)}
\def\briesp{G_0(\Xd,W,\pi,p)}

\def\shift{\mathbb{S}}
\def\That{\widehat{T}}
\def\Ghat{\widehat{G}}
\def\Hhat{\widehat{H}}
\newcommand{\qc}[1]{\nabla^{#1}}
\def\gmpt{\bulletsmall_{GMP}^T}
\def\gmpg{\bulletsmall_{GMP}^G}
\def\gmph{\bulletsmall_{GMP}^H}
\def\pgmpt{\Phi_{GMP}^T}
\def\psgmpt{\Psi_{GMP}^T}
\def\psgmpg{\Psi_{GMP}^G}

\def\pgmph{\Phi_{GMP}^H}
\def\pgmpg{\Phi_{GMP}^G}
\def\pbfd{\Phi_{BF}^{\delta}}
\def\psbfd{\Psi_{BF}^{\delta}}
\def\pbfk{\Phi_{BF}^{\kappa}}
\def\psbfk{\Psi_{BF}^{\kappa}}
\def\ascl{\xi}
\def\aggh{ \bulletsmall_{\mathcal{G}r}^H }
\def\aggg{ \bulletsmall_{\mathcal{G}r}^G }
\def\aggt{ \bulletsmall_{\mathcal{G}r}^T }
\DeclareMathOperator{\gs}{GS}
\DeclareMathOperator{\eendo}{End}
\def\hcdot{\cdot}
\DeclareMathOperator{\fr}{Fr}
\def\hc{\mathcal{H}\mathcal{C}_{\hslash}}
\def\hcf{\mathcal{H}\mathcal{C}^{fr}_{\hslash}}
\def\hct{\widetilde{\mathcal{H}\mathcal{C}}_{\hslash}}
\def\hci{\Theta_{HC}}
\def\hcprod{*}
\DeclareMathOperator{\twist}{Tw}
\def\kf{\kappa_{\hslash}}
\def\dkf{\delta_{\hslash}}
\def\ot{\otimes}
\def\ic{\mathcal{I}\mathcal{C}}
\def\GS{\mathbb{G}\mathbb{S}}
\def\daggg{\mathbb{D}_{\mathcal{G}r}}
\def\dagl{\mathbb{D}_{\mathcal{G}r^{\leqslant\lambda }}}
\DeclareMathOperator{\Hom}{Hom}
\DeclareMathOperator{\module}{mod}

\def\bfisod{\Theta^{\delta}}
\def\bfisok{\Theta^{\kappa}}
\def\UU{M}
\def\VV{\mathcal{V}}
\def\WW{\mathcal{W}}
\def\twdot{\cdot_{-\rho^{\vee}}}
\def\bd{\beta^{\vee}}
\def\ZZZ{\mathcal{Z}}
\DeclareMathOperator{\mirror}{mir}

\def\hey{z}
\def\ffff{f}

\def\CCC{\mathcal{C}}
\def\CCCfr{\mathcal{C}^{fr}}

\def\CCCp{\mathcal{C}_+}
\def\CCChtf{\mathcal{C}_{\hslash TF}}

\def\vol{\omega_{X^{\vee}_P}}
\def\Mir{\Phi_{mir}}
\def\zzeta{\eta}

\def\ip{I_P}
\def\iip{I\setminus I_P}

\def\qi{q_i|~i\in I\setminus I_P}
\def\qim{q_i^{-1}|~i\in I\setminus I_P}

\def\dh{D_{\hslash}}
\begin{document}
\title[The $D_{\hslash}$-module mirror conjecture for flag varieties]{The $D_{\hslash}$-module mirror conjecture for flag varieties}

\author{Chi Hong Chow}
\address{Max Planck Institute for Mathematics, 53111 Bonn, Germany}
\email{chow@mpim-bonn.mpg.de}

\begin{abstract} 
Rietsch constructed a candidate $T$-equivariant mirror LG model for any flag variety $G/P$. In this paper, we prove the following mirror symmetry prediction: the small $T\times\mathbb{G}_m$-equivariant quantum cohomology of $G/P$ equipped with quantum $\hp$-connection is isomorphic as $D_{\hslash}$-modules to the Brieskorn lattice associated to the LG model equipped with Gauss-Manin $\hp$-connection.
\end{abstract}

\makeatletter
\@namedef{subjclassname@2020}{\textup{2020} Mathematics Subject Classification}
\makeatother
 
%\keywords{} 

\maketitle
%%%%%%%%%%%%%%%%%%%%%%%%%%%%%%%%%%%
%%%%%%%%%%%%%%%%%%%%%%%%%%%%%%%%%%%
%%%%%%%%%%%%%%%%%%%%%%%%%%%%%%%%%%%
%%%%%%%%%%%%%%%%%%%%%%%%%%%%%%%%%%%
%%%%%%%%%%%%%%%%%%%%%%%%%%%%%%%%%%%
%%%%%%%%%%%%%%%%%%%%%%%%%%%%%%%%%%%
\section{introduction}\label{1}
%%%%%%%%%%%%%%%%%%%%%%%%%%%%%%%%%%%
%%%%%%%%%%%%%%%%%%%%%%%%%%%%%%%%%%%
\subsection{Main result}\label{mainresult}
Let $G$ be a simple simply-connected complex algebraic group and $P\subseteq G$ a parabolic subgroup. Motivated by the work of Batyrev, Ciocan-Fontanine, Kim and van Straten \cite{BCFKS}, Givental \cite{Givental}, Joe and Kim \cite{JK} and Peterson (unpublished, but see \cite{Peter}), Rietsch \cite{Rietsch} constructed a candidate $T$-equivariant mirror of $G/P$ in terms of the Langlands dual group $\Gd$. It is a quintuple $(\Xd,W,\pi,p,\vol)$ where
\begin{itemize}
\item $\Xd$ is a smooth affine variety;

\item $W\in\OO(\Xd)$ is a regular function;

\item $\pi:\Xd\ra\ZL$ is a smooth morphism onto a subtorus $\ZL$ of $\Td$;

\item $p:\Xd\ra\Td$ is a morphism; and

\item $\vol\in \O^{top}(\Xd/\ZL)$ is a fiberwise volume form.
\end{itemize}

In this paper, we prove
\begin{theorem}\label{main0}
The $\dh$-module mirror conjecture holds for the above mirror pair.
\end{theorem}

\noindent That is, we prove that some $\dh$-modules associated to $G/P$ and $(\Xd,W,\pi,p,\vol)$ are isomorphic.

On the A-side, we have an $H^{\bl}_{T\times\gm}(\pt)\ot_{\CC[\hp]}\dm{(\gm)^{|\iip|}}$-module $QH^{\bl}_{T\times\gm}(G/P)[\qim]$, the \textit{small $T\times\gm$-equivariant quantum cohomology of $G/P$}, where the $\dh$-module structure is given by the \textit{quantum $\hp$-connection}
\[\nabla_{\partial_{q_i}}:=\hp\partial_{q_i}+q_i^{-1}c_1^{T\times\gm}(L_{\w_i})\star -.\]

On the B-side, we have a $\sym^{\bl}(\ttd)[\hp]\ot_{\CC[\hp]}\dm{\ZL}$-module 
\begin{align*}
 &\bries \\
 :=~& \coker\left( \sym^{\bl}(\ttd)[\hp]\ot \O^{top-1}(\Xd/\ZL)\xrightarrow{\partial}\sym^{\bl}(\ttd)[\hp]\ot  \O^{top}(\Xd/\ZL) \right)
\end{align*}
where 
\[  \partial(z\ot \w):= z\ot (\hp d\w + dW\wedge\w) - \sum_i zh_i\ot(p^*\langle h^i,\mccccc_{\Td} \rangle)\wedge\w .\]
It is called the \textit{Brieskorn lattice} and its $\dm{\ZL}$-module structure is given by the \textit{Gauss-Manin $\hp$-connection}
\[ \eta\cdot [z\ot \w]:= \left[z\ot(\hp\LL_{\widetilde{\eta}}\w + (\LL_{\widetilde{\eta}}W)\w) - \sum_i zh_i\ot (\iota_{\widetilde{\eta}} p^*\langle h^i,\mccccc_{\Td} \rangle)\w \right] .\]

In addition to the $\dh$-module structures, both $QH_{T\times\gm}^{\bl}(G/P)[\qim]$ and $\bries$ have two structures, namely the \textit{grading} and the \textit{shift operators}. At the semi-classical limit $\hp\to 0$, these modules become the \textit{$T$-equivariant quantum cohomology ring} $QH_T^{\bl}(G/P)[\qim]$ and the \textit{Jacobi ring} $\jac(\Xd,W,\pi,p)$ respectively.

\begin{theorem}\label{main}
(Precise version of Theorem \ref{main0}) After identifying the base algebras via some canonical isomorphisms
\[ \sym^{\bl}(\ttd)[\hp]\xrightarrow{\sim} H^{\bl}_{T\times\gm}(\pt)\quad\text{ and }\quad \OO(\ZL)\xrightarrow{\sim} \CC[\eccc], \]
there exists a $\sym^{\bl}(\ttd)[\hp]\ot\OO(\ZL)$-linear map
\[ \Mir:\bries\ra QH^{\bl}_{T\times\gm}(G/P)[\ecc] \]
satisfying 
\begin{enumerate}
\item it is bijective;

\item it is $\dm{\ZL}$-linear;

\item $\Mir([\vol])=1$;

\item at the semi-classical limit, it becomes a ring isomorphism
\[ \Mir^{\hp=0}: \jac(\Xd,W,\pi,p)\xrightarrow{\sim} QH_T^{\bl}(G/P)[\ecc];\]

\item it intertwines the shift operators; and

\item it is graded.
\end{enumerate}
Any $\sym^{\bl}(\ttd)[\hp]\ot\OO(\ZL)$-linear map satisfying (3) and (5) must be equal to $\Mir$.
\end{theorem}

\begin{remark}
Property (5) implies that for any character $\lambda$ of $\Td$, the isomorphism $\Mir^{\hp=0}$ from (4) sends $[\lambda\circ p]$ to the \textit{$T$-equivariant Seidel's element} \cite{Seidel} associated to the cocharacter $\lambda:\gm\ra T$.
\end{remark}

%%%%%%%%%%%%%%%%%%%%%%%%%%%%%%%%%%%
%%%%%%%%%%%%%%%%%%%%%%%%%%%%%%%%%%%
%%%%%%%%%%%%%%%%%%%%%%%%%%%%%%%%%%%
\subsection{Outline of proof} \label{outlinefriendly} We explain how to construct $\Mir$.

Let us first explain how to construct the ring isomorphism $\Mir^{\hp=0}$ based on known results from the literature. Introduce two rings:
\begin{enumerate}[(i)]
\item $\OO(\Bed)$, the coordinate ring of the base change of the universal centralizer of $\Gd$ along the quotient morphism $\ttd\ra \ttd/W\simeq \ggd/\!\!/\Gd$.

\item $H_{\bl}^T(\ag)$, the $T$-equivariant Pontryagin homology of the affine Grassmannian $\ag$ of $G$.
\end{enumerate}
Recall the following results:
\begin{enumerate}
\item Rietsch \cite{Rietsch} constructed a ring map
\[ \widetilde{\Phi}_R :\OO(\Bed)\ra \jac(\Xd,W,\pi,p)\]
which induces an isomorphism
\[ \Phi_R: \OO(\Bed\times_{\Gd}\Umd(\wpd)^{-1}\Bd_-) \xrightarrow{\sim} \jac(\Xd,W,\pi,p).\]

\item Using the geometric Satake equivalence, Yun and Zhu \cite{YZ} constructed a ring isomorphism
\[ \Phi_{YZ}:\OO(\Bed)\xrightarrow{\sim} H^T_{-\bl}(\ag).\]

\item Discovered by Peterson \cite{Peter} and proved by Lam and Shimozono \cite{LS}, there exists a ring map
\[ \Phi_{PLS}:H^T_{-\bl}(\ag) \ra QH_T^{\bl}(G/P)[\ecc]\]
which is surjective after localization and sends every affine Schubert class to either 0 or an explicit quantum Schubert class.

\item The author of the present paper \cite{me2} proved that the composition $\Phi_{PLS}\circ\Phi_{YZ}$ descends to a ring isomorphism
\[ \Phi_P :\OO(\Bed\times_{\Gd}\Umd(\wpd)^{-1}\Bd_-)  \xrightarrow{\sim} QH_T^{\bl}(G/P)[\ecc].\]
This gives the Peterson variety presentation.
\end{enumerate}

To summarize, we have the following commutative diagram
\begin{equation}\label{outlinefriendlydiagram1}
\begin{tikzpicture}
\tikzmath{\x1 = 9; \x2 = 4;}
\node (A) at (0,0) {$\OO(\Bed)$} ;
\node (B) at (\x1,0) {$H^T_{-\bl}(\ag)$} ;
\node (C) at (0,-\x2) {$\jac(\Xd,W,\pi,p)$} ;
\node (D) at (\x1,-\x2) {$QH_T^{\bl}(G/P)[\ecc]$} ;
\node (E) at (0.5*\x1,-0.5*\x2) {$\OO(\Bed\times_{\Gd}\Umd(\wpd)^{-1}\Bd_-)$};

\path[->, font=\small] (A) edge node[above]{$\Phi_{YZ}$} (B);
\path[->, font=\small] (A) edge node[left]{$\widetilde{\Phi}_R$} (C);
\path[->, font=\small] (B) edge node[right]{$\Phi_{PLS}$} (D);
\path[->, dashed, font=\small]
(C) edge node[above]{$\exists~ \Mir^{\hp=0}$} (D);
\path[->, font=\small] (A) edge node[right]{} (E);
\path[->, font=\small] (E) edge node[left]{$\Phi_R\quad$} (C);
\path[->, font=\small] (E) edge node[right]{$\quad\!\Phi_P$} (D);
\end{tikzpicture}.
\end{equation}
The composition 
\[ \Mir^{\hp=0}:=\Phi_P\circ \Phi^{-1}_R\]
gives the desired ring isomorphism.

Now, to construct $\Mir$, we ``quantize'' everything. More precisely, we replace the above rings (resp. ring maps) by $\CC[\hp]$-modules (resp. $\CC[\hp]$-linear maps) and establish the analogue of the commutative diagram \eqref{outlinefriendlydiagram1} which recovers the original one by taking $-\ot_{\CC[\hp]}\CC$. 
\begin{enumerate}
\item[($\hp$i)] We replace $\OO(\Bed)$ by the desymmetrized quantum Toda lattice $\dkf(\KK)$ of $\Gd$ in the sense of Bezrukavnikov and Finkelberg \cite{BF}.

\item[($\hp$ii)] We replace $H_{\bl}^T(\ag)$ by the $T\times\gm$-equivariant homology $H_{\bl}^{T\times\gm}(\ag)$ of $\ag$ where $\gm$ corresponds to the loop rotation.
\end{enumerate}

\begin{enumerate}
\item[($\hp$1)] We construct
\[ \widetilde{\Phi}_R^{\hp}:\dkf(\KK)\ra \bries\]
by hand. This map is somewhat canonical: for any $x_1,\ldots,x_n\in\ttd$ and $f\in\OO(\Gd)$ satisfying $f(u\cdot -)=f(-)$ whenever $u\in\Umd$, we have
\[  \widetilde{\Phi}_R^{\hp}([x_1^L\cdots x_n^Lf])=[(x_1+\hp\rhod(x_1))\cdots (x_n+\hp\rhod(x_n))\ot f|_{\Xd}\vol]\]
where $\rhod$ is the Weyl covector. We show that $\widetilde{\Phi}_R^{\hp}$ induces 
\[  \Phi_R^{\hp}:\dkf(\KK)\ot\OO(\ZL)/\WW\ra \bries\]
for some submodule $\WW$ which will be defined in Section \ref{finaltwomodw}. Unlike the semi-classical case, we are only able to show that $\Phi_R^{\hp}$ is surjective.

\item[($\hp$2)] We take $\Phi_{YZ}^{\hp}$ to be the isomorphism
\[ \pbfd: \dkf(\KK)\xrightarrow{\sim} H_{-\bl}^{T\times\gm}(\ag)\]
constructed by Bezrukavnikov and Finkelberg \cite{BF}.

\item[($\hp$3)] We define 
\[ \Phi_{PLS}^{\hp}:  H_{-\bl}^{T\times\gm}(\ag) \ra QH_{T\times\gm}^{\bl}(G/P)[\ecc]\]
to be the natural $T\times\gm$-equivariant lift of $\Phi_{PLS}$. (Recall $\Phi_{PLS}$ sends affine Schubert classes to either 0 or quantum Schubert classes and these classes have natural $T\times\gm$-equivariant lifts.)

\item[($\hp$4)] We show that $\Phi_{PLS}^{\hp}\circ\Phi_{YZ}^{\hp}$ induces 
\[ \Phi^{\hp}_P:\dkf(\KK)\ot\OO(\ZL)/\WW \ra QH_{T\times\gm}^{\bl}(G/P)[\ecc]. \]
Unlike the semi-classical case, we are only able to show that $\Phi_P^{\hp}$ is surjective.
\end{enumerate}
Although we do not know whether $\Phi_R^{\hp}$ and $\Phi_P^{\hp}$ are bijective, we are able to show that there exists a unique map $\Mir$ which is bijective such that 
\[  \Mir\circ \Phi_R^{\hp} = \Phi_P^{\hp}.\]

To conclude, we not only manage to construct $\Mir$ but also establish the analogue of the commutative diagram \eqref{outlinefriendlydiagram1}. For potential applications, let us also quote the latter result as a theorem. For simplicity, we drop $\dkf(\KK)\ot\OO(\ZL)/\WW$ from the diagram below.
\begin{theorem}\label{main2}
The mirror isomorphism $\Mir$ from Theorem \ref{main} fits into the commutative diagram
\begin{equation}\nonumber
\begin{tikzpicture}
\tikzmath{\x1 = 7; \x2 = 3;}
\node (A) at (0,0) {$\dkf(\KK)$} ;
\node (B) at (\x1,0) {$H^{T\times\gm}_{-\bl}(\ag)$} ;
\node (C) at (0,-\x2) {$\bries$} ;
\node (D) at (\x1,-\x2) {$QH_{T\times\gm}^{\bl}(G/P)[\ecc]$} ;

\path[->, font=\small] (A) edge node[above]{$\Phi^{\hp}_{YZ}$} (B);
\path[->, font=\small] (A) edge node[left]{$\widetilde{\Phi}^{\hp}_R$} (C);
\path[->, font=\small] (B) edge node[right]{$\Phi^{\hp}_{PLS}$} (D);
\path[->, font=\small]
(C) edge node[above]{$ \Mir$} (D);
\end{tikzpicture}.
\end{equation}
In fact, as an $\OO(\ZL)$-linear map, $\Mir$ is uniquely determined by this diagram.
\end{theorem}

Finally, let us explain briefly how to verify that $\Mir$ is $\dm{\ZL}$-linear. The $\OO(\ZL)$-linearity follows immediately from the construction. It remains to show that $\Mir$ intertwines the operators $\partial_{q_i}\cdot -$ arising from the source and target. According to Bezrukavnikov-Finkelberg's theory \cite{BF}, $\dkf(\KK)$ is obtained from the quantum Toda lattice $\kf(\KK)$ by extension of scalars. The latter is a ring, and hence $\dkf(\KK)$ is a right $\kf(\KK)$-module. Moreover, $\pbfd$ is obtained from an isomorphism
\[ \pbfk: \kf(\KK)\xrightarrow{\sim} H_{-\bl}^{G\times\gm}(\ag)\]
of rings by extension of scalars, and hence $\pbfd$ is a module isomorphism with respect to $\pbfk$. 

By a straightforward computation, each B-side operator $\partial_{q_i}\cdot -$ can be lifted to an operator on $\dkf(\KK)\ot\OO(\ZL)$ given by $\hp\partial_{q_i}$ plus the right multiplication by an element $x_i$ of $\kf(\KK)\ot\OO(\ZL)$. We show that after extension of scalars, $\Phi_{PLS}^{\hp}\circ\pbfk$ sends $x_i$ to $q_i^{-1}c_1^{T\times \gm}(L_{\omega_i})$, the zeroth order term of the A-side operator $\nabla_{\partial_{q_i}}$. By generalizing our recent new proof \cite{me1} of Peterson-Lam-Shimozono's theorem \cite{LS, Peter}, we show that $\Phi_{PLS}^{\hp}$ is induced by the $H_{\bl}^{G\times\gm}(\ag)$-module action on $QH^{\bl}_{G\times\gm}(G/P)[\ecc]$ (a.k.a. non-abelian shift operators) recently constructed by Gonz\'alez, Mak and Pomerleano \cite{GMP}. We conclude the proof by applying the facts that $\pbfd$ is a module isomorphism with respect to $\pbfk$ and that non-abelian shift operators commute with $\nabla_{\partial_{q_i}}$.
%%%%%%%%%%%%%%%%%%%%%%%%%%%%%%%%%%%
%%%%%%%%%%%%%%%%%%%%%%%%%%%%%%%%%%%
%%%%%%%%%%%%%%%%%%%%%%%%%%%%%%%%%%%
\subsection{Related work}
Mirror symmetry has been extensively researched and we are unable to review the literature. Let us mention however some works with which our results fit.
\begin{enumerate}
\item The $\dh$-module mirror conjecture (DMC) for toric varieties has been proved in great generality. The first non-trivial case, which follows from Givental's mirror theorem \cite{Giventalmirrortheorem}, assumes the toric varieties to be smooth, projective and semi-Fano. The result has been extended numerous times. The most general case, proved by Coates, Corti, Iritani and Tseng \cite{toricstack}, replaces the toric varieties with arbitrary toric Deligne-Mumford stacks with semi-projective coarse moduli and the small quantum $D_{\hp}$-modules with the big quantum $D_{\hp}$-modules. Prior to this, Iritani \cite{Iritanibig} has proved DMC in less generality, using shift operators. We remark that it is the last work which inspired the formulation of Theorem \ref{main}.

\item Much has been known about the case $P=B$. Generalizing Givental's work \cite{Givental} on the type A case, Kim \cite{Kim} proved that $QH_{G\times\gm}^{\bl}(G/B)[q_i^{-1}|~i\in I]$ is isomorphic to the \textit{open quantum Toda lattice} constructed by Kostant \cite{KostantToda}. We expect that after extension of scalars, the latter module is isomorphic to $\bries$. This will give an alternative (and shorter) proof of DMC for $G/B$.

\item Rietsch \cite{Rietsch} proved that $\jac(\Xd,W,\pi,p)$ is isomorphic to $\OO(\Bed\times_{\Gd}\Umd(\wpd)^{-1}\Bd_-) $. Combined with the Peterson variety presentation which is proved recently by the author of the present paper \cite{me2} in full generality, it yields the isomorphism $\Mir^{\hp=0}$ stated in Theorem \ref{main}. See Section \ref{outlinefriendly} for more details. 

\item Katzarkov, Kontsevich and Pantev \cite{KKP} proposed a version of DMC for arbitrary mirror pairs and interpreted it as a (conjectural) consequence of the homological mirror symmetry conjecture \cite{HMS}. 

\item For $G/P$ an even dimensional quadric, Pech, Rietsch and Williams \cite{PRW} constructed an explicit injective $\dm{\ZL}$-linear map from the A-model $\dh$-module to the B-model $\dh$-module.

\item For $G/P$ an odd dimensional quadric, Pech and Rietsch \cite{PR} constructed an explicit injective $\dm{\ZL}$-linear map from the A-model $\dh$-module to the B-model $\dh$-module and proved that it is bijective when $\dim G/P=3$.

\item For $G/P$ a Grassmannian, Marsh and Rietsch \cite{MR} constructed an explicit injective $\dm{\ZL}$-linear map from the A-model $\dh$-module to the B-model $\dh$-module.

\item Lam and Templier \cite{LT} proved DMC for minuscule flag varieties (Grassmannians and even dimensional quadrics, for example) based on a result from the geometric Langlands program, namely Zhu's theorem \cite{Zhu}, which identifies Frenkel-Gross' $D$-module \cite{FG} with Heinloth-Ng\^o-Yun's $D$-module \cite{HNY}.

\item Hu \cite{Hu} proved a big version of DMC for quadrics.

\item For $G=\mathbf{SL_n}$ and $P$ arbitrary, Li, Rietsch, Yang and Zhang \cite{RietschPlucker} analyzed Rietsch's version \cite{RietschJAMS, Rietsch} of the ring isomorphism
\[ \Mir^{\hp,h=0}:\jac(\Xd,W,\pi,p)_{h=0} \xrightarrow{\sim} QH^{\bl}(G/P)[\ecc]\]
and proved that it sends $[W]$ to the first Chern class $c_1(G/P)$. As a corollary, the eigenvalues of $c_1(G/P)\star -$ are equal to the critical values of $W$.
\end{enumerate}
%%%%%%%%%%%%%%%%%%%%%%%%%%%%%%%%%%%
%%%%%%%%%%%%%%%%%%%%%%%%%%%%%%%%%%%
%%%%%%%%%%%%%%%%%%%%%%%%%%%%%%%%%%%
\subsection*{Acknowledgements}
A significant part of the revision was made when I was a postdoctoral fellow at the Max Planck Institute for Mathematics in Bonn. I am grateful to it for its hospitality and financial support.
%%%%%%%%%%%%%%%%%%%%%%%%%%%%%%%%%%%
%%%%%%%%%%%%%%%%%%%%%%%%%%%%%%%%%%%
%%%%%%%%%%%%%%%%%%%%%%%%%%%%%%%%%%%
%%%%%%%%%%%%%%%%%%%%%%%%%%%%%%%%%%%

%%%%%%%%%%%%%%%%%%%%%%%%%%%%%%%%%%%
%%%%%%%%%%%%%%%%%%%%%%%%%%%%%%%%%%%
%%%%%%%%%%%%%%%%%%%%%%%%%%%%%%%%%%%
%%%%%%%%%%%%%%%%%%%%%%%%%%%%%%%%%%%
%%%%%%%%%%%%%%%%%%%%%%%%%%%%%%%%%%%
%%%%%%%%%%%%%%%%%%%%%%%%%%%%%%%%%%%
\section{Notation}\label{prelim}

%%%%%%%%%%%%%%%%%%%%%%%%%%%%%%%%%%%
%%%%%%%%%%%%%%%%%%%%%%%%%%%%%%%%%%%
\subsection{Notation for $G$}\label{notationG}
Let $G$ be a simple simply-connected complex algebraic group. Fix a maximal torus $T$ and a pair of Borel subgroups $B$ and $B_-$ such that $B\cap B_-=T$. Denote by $R$ the set of roots associated to $(G,T)$, by $\{\a_1,\ldots,\a_r\}\subset R$ the fundamental system determined by $B$, and by $R^+$ the set of roots which are positively spanned by $\a_1,\ldots,\a_r$. For any $\a\in R$, denote by $\ad$ the coroot associated to $\a$. Let $\{\w_1,\ldots,\w_r\}$ be the dual basis of $\{\ad_1,\ldots,\ad_r\}$, i.e. the fundamental weights. Denote by $W$ the Weyl group and by $w_0$ the longest element of $W$. Denote by $\Q$ the coroot lattice and by $\cop\subset \Q$ the subset consisting of $\lambda$ such that $\a(\lambda)\geqslant 0$ for any $\a\in R^+$. 

Fix a parabolic subgroup $P$ of $G$ which contains $B$. It is classified by a subset $\ip$ of $I:=\{1,\ldots, r\}$. We have 
\[ \lie(P) = \lie(B)\oplus \bigoplus_{\a\in -R^+_P}\gg_{\a}\]
where $R_P^+:= R^+\cap \sum_{i\in\ip}\ZZ\cdot\a_i$. Denote by $W_P$ the subgroup of $W$ generated by simple reflections $s_{\a_i}$ with $i\in\ip$, and by $w_P$ the longest element of $W_P$. Define $\Q_P:=\bigoplus_{i\in\ip}\ZZ\cdot\ad_i\subseteq \Q$.

Let $\ag:=G(\KK)/G(\OO)$ be the affine Grassmannian of $G$ where $\KK:=\CC((z))$ and $\OO:=\CC[[z]]$. It has a $G(\KK)\rtimes \gm$-action where $\gm$ corresponds to the loop rotation. Every $\lambda\in\Q$ is naturally a point of $G(\KK)$, and hence it gives rise to a point of $\ag$ which we denote by $t^{\lambda}$. It is known that $\{t^{\lambda}\}_{\lambda\in\Q}$ is the set of $T\times\gm$-fixed points of $\ag$ (recall $G$ is assumed to be simply-connected). For any $\lambda\in\cop$, define the spherical affine Schubert variety $\agll:=\ol{\agl}$ where $\agl:=G(\OO)\cdot t^{\lambda}$. Define $\BB:=\ev_{z=0}^{-1}(B_-)$ where $\ev_{z=0}:G(\OO)\ra G$ is the evaluation map at $z=0$. Let $W_{af}^-\subset W_{af}:=W\ltimes\Q$ be the set of minimal length coset representatives in $W_{af}/W$. Then the map $W_{af}^-\ra \Q$ defined by $\wl\mapsto w(\lambda)$ is bijective. For any $\wl\in W_{af}^-$, define the affine Schubert class
\[ \ascl_{\wl} := \left[~\ol{\BB\cdot t^{w(\lambda)}}~\right] \in H_{2\ell(\wl)}^{T\times \gm}(\ag).\]
It is known that $\{\ascl_{\wl}\}_{\wl\in W_{af}^-}$ is an $H_{T\times\gm}^{\bl}(\pt)$-basis of $H_{\bl}^{T\times\gm}(\ag)$.

Put $\That:=T\times\gm$ and $\Ghat:=G\times\gm$. There is a ring structure, called the \textit{convolution product}, defined on $H_{\bl}^{\Ghat}(\ag)$
\[ -\aggg -: H_{\bl}^{\Ghat}(\ag)\ot_{H_{\Ghat}^{\bl}(\pt)}H_{\bl}^{\Ghat}(\ag)\ra H_{\bl}^{\Ghat}(\ag) \]
which recovers the \textit{Pontryagin product} by forgetting the $\gm$-action. It is characterized by the commutative diagram
\begin{equation}\nonumber
\begin{tikzpicture}
\tikzmath{\x1 = 10; \x2 = 2;}
\node (A) at (0,0) {$H_{\bl}^{\Ghat}(\ag)\ot_{H_{\Ghat}^{\bl}(\pt)}H_{\bl}^{\Ghat}(\ag)$} ;
\node (B) at (\x1,0) {$H_{\bl}^{\Ghat}(\ag)$} ;
\node (C) at (0,-\x2) {$\left(\bigoplus_{\mu\in\Q}\mathcal{R}_{\mu}\cdot [t^{\mu}]\right)\ot_{\fof H_{\That}^{\bl}(\pt)} \left(\bigoplus_{\mu\in\Q}\mathcal{R}_{\mu}\cdot [t^{\mu}]\right)$} ;
\node (D) at (\x1,-\x2) {$\bigoplus_{\mu\in\Q}\mathcal{R}_{\mu}\cdot [t^{\mu}]$} ;

\path[->, font=\tiny] (A) edge node[above]{$-\aggg -$} (B);
\path[right hook->] (A) edge node[left]{} (C);
\path[right hook->] (B) edge node[right]{} (D);
\path[->, font=\tiny]
(C) edge node[above]{$[t^{\mu_1}]\ot [t^{\mu_2}]\mapsto [t^{\mu_1+\mu_2}] $} (D);
\end{tikzpicture}
\end{equation}
where the vertical arrows are induced by localization and $\mathcal{R}_{\mu}:=\fof H_{\That}^{\bl}(\pt)$ equipped with a $\fof H_{\That}^{\bl}(\pt)\ot_{\CC[\hp]}\fof H_{\That}^{\bl}(\pt)$-module structure defined by letting the first factor act via the identity and the second factor act via the automorphism of the $\CC[\hp]$-algebra $H_{\That}^{\bl}(\pt)\simeq \CC[\w_1,\ldots,\w_r,\hp]$
\[ \twist_{-\mu}: f(\w_1,\ldots,\w_r,\hp)\mapsto f(\w_1-\w_1(\mu)\hp, \ldots, \w_r-\w_r(\mu)\hp,\hp).\]
(Notice that $H_{\bl}^{\Ghat}(\ag)$ is a module over $H_{\Ghat}^{\bl}(\pt)\ot H_{G}^{\bl}(\pt)\simeq H_{\Ghat}^{\bl}(\pt)\ot_{\CC[\hp]} H_{\Ghat}^{\bl}(\pt)$ where the first factor corresponds to the canonical $\Ghat$-action on $\ag$ and the second corresponds to the canonical $G$-torsor on $\ag$.)

Define
\[ -\aggt -: H_{\bl}^{\That}(\ag)\ot_{H_{\Ghat}^{\bl}(\pt)}H_{\bl}^{\Ghat}(\ag)\ra H_{\bl}^{\That}(\ag) \]  
to be the map obtained from $-\aggg-$ by extension of scalars.  
%%%%%%%%%%%%%%%%%%%%%%%%%%%%%%%%%%%
%%%%%%%%%%%%%%%%%%%%%%%%%%%%%%%%%%%
\subsection{Notation for $\Gd$}\label{notationGd} 
Let $\Gd$ be the Langlands dual group of $G$, defined over $\CC$, and $\Td$ the maximal torus of $\Gd$ dual to $T$. The pair $(\Gd,\Td)$ is characterized by the property that the character (resp. cocharacter) lattice of $\Td$ is equal to the cocharacter (resp. character) lattice of $T$, and the set of roots (resp. coroots) associated to $(\Gd,\Td)$ is equal to the set of coroots (resp. roots) associated to $(G,T)$. Since $G$ is simply-connected, $\Gd$ is of adjoint type. Denote by $\Bd$ the Borel subgroup of $\Gd$ associated to the fundamental system $\{\ad_1,\ldots,\ad_r\}$, and by $\Bmd$ the Borel subgroup opposite to $\Bd$. Denote by $\Ud$ (resp. $\Umd$) the unipotent radical of $\Bd$ (resp. $\Bmd$). The Lie algebras of $\Gd$, $\Td$, $\Bd$, $\Bmd$, $\Ud$, $\Umd$ are denoted by $\ggd$, $\ttd$, $\mathfrak{b}^{\vee}$, $\mathfrak{b}^{\vee}_-$, $\mathfrak{n}^{\vee}$ and $\mathfrak{n}^{\vee}_-$ respectively.

Recall the parabolic subgroup $P$ of $G$ fixed in Section \ref{notationG}. Let $\Pd$ be the parabolic subgroup of $\Gd$ satisfying
\[ \lie(\Pd) = \lie(\Bd)\oplus \bigoplus_{\a\in -R^+_P}\ggd_{\ad}.\]
Denote by $\ZL\subseteq \Td$ the center of the Levi subgroup $\Ld$ of $\Pd$. It is the kernel of the homomorphism $(\ad_i)_{i\in\ip}:\Td\ra (\gm)^{|\ip|}$ of algebraic groups. Since $\Gd$ is of adjoint type, the morphism $(\ad_i|_{\ZL})_{i\in\iip}:\ZL\ra (\gm)^{|\iip|}$ is an isomorphism, and hence it induces an isomorphism $\CC[\eccc]\xrightarrow{\sim}\OO(\ZL)$ of $\CC$-algebras. Define
\begin{equation}\label{mirrormapdef}
\mirror:\OO(\ZL)\xrightarrow{\sim}\CC[\eccc]
\end{equation}
to be the inverse of the last isomorphism. (In the literature, $\spec( \mirror^{-1})$ is called the \textit{mirror map}.) For simplicity, we will sometimes identify these two algebras without mentioning $\mirror$.

For any $\lambda\in\cop$, define $S(\lambda):=H^0(\GB;\Gd\times^{\Bmd}\CC_{\lambda})$ which is naturally a representation of $\Gd$. By Borel-Weil's theorem, $S(\lambda)$ is irreducible and has highest weight $\lambda$. For each $\mu\in\Q$, denote by $S(\lambda)_{\mu}$ the $\mu$-weight space of $S(\lambda)$ and by $S(\lambda)_{\leqslant\mu}$ (resp. $S(\lambda)_{<\mu}$) the sum of $S(\lambda)_{\nu}$ with $\nu\leqslant\mu$ (resp. $\nu\leqslant\mu$ and $\nu\ne\mu$). (Here $\nu\leqslant\mu$ iff $\mu-\nu\in\sum_{i=1}^r\ZZ_{\geqslant 0}\cdot \ad_i$.)

For any Lie algebra $\mathfrak{h}$, let $U_{\hp}(\mathfrak{h})$ be the asymptotic universal enveloping algebra of $\mathfrak{h}$ and $Z_{\hp}(\mathfrak{h})$ the center of $U_{\hp}(\mathfrak{h})$. A well-known result of Harish-Chandra says that $\zgd$ is isomorphic to $\utd^W$ as $\CC[\hp]$-algebras. Let us recall how the isomorphism is constructed. By PBW's theorem, we have $\ugd=\utd\oplus(\mathfrak{n}^{\vee}\cdot \ugd + \ugd\cdot \mathfrak{n}^{\vee}_-)$ as $\CC[\hp]$-modules. Hence every $z\in\ugd$ can be expressed uniquely as the sum $z'+z''$ where $z'\in\utd$ and $z''\in \mathfrak{n}^{\vee}\cdot \ugd + \ugd\cdot \mathfrak{n}^{\vee}_-$. (One can show that $z''$ actually lies in $\mathfrak{n}^{\vee}\cdot \ugd \cap \ugd\cdot \mathfrak{n}^{\vee}_-$.) Let $\twist_{\rhod}$ be the automorphism of the $\CC[\hp]$-algebra $\utd$ defined by
\[ \twist_{\rhod}(x) := x +\hp\rhod(x),\quad x\in\ttd\]
where $\rhod$ is the half-sum of the positive coroots. Define the \textit{$\hp$-Harish-Chandra homomorphism}
\[ \begin{array}{ccccc}
\hci&:&\zgd&\ra& \utd\\ [.5em]
& & z &\mapsto & \twist_{\rhod}(z')
\end{array}.
\] 
Then $\hci$ is an injective homomorphism of $\CC[\hp]$-algebras whose image is equal to $\utd^W$.
%%%%%%%%%%%%%%%%%%%%%%%%%%%%%%%%%%%
%%%%%%%%%%%%%%%%%%%%%%%%%%%%%%%%%%%
%%%%%%%%%%%%%%%%%%%%%%%%%%%%%%%%%%%
%%%%%%%%%%%%%%%%%%%%%%%%%%%%%%%%%%%
%%%%%%%%%%%%%%%%%%%%%%%%%%%%%%%%%%%
%%%%%%%%%%%%%%%%%%%%%%%%%%%%%%%%%%%
\section{A-model}\label{Amodel}
\subsection{Quantum cohomology} \label{Qcohomology}
Let $P$ be the parabolic subgroup of $G$ fixed in Section \ref{notationG}. Let $H$ be either $T$ or $G$. Put $\Hhat:=H\times\gm$. Consider the $\Hhat$-action on $G/P$ where the $H$-action is the standard action and the $\gm$-action is the trivial action. We will denote by $\hp$ the equivariant parameter for this extra $\gm$-action. Recall the \textit{$\Hhat$-equivariant quantum cohomology}
\[ QH_{\Hhat}^{\bl}(G/P):= H_{\Hhat}^{\bl}(G/P)\ot\CC[\qi].\]
Here, each $q_i$ is the quantum parameter which corresponds to the effective curve class $\b_i\in\pi_2(G/P)$ satisfying $\langle c_1(L_{\rho}),\b_i\rangle=\rho(\ad_i)$ for any $\rho\in (\Q/\Q_P)^*$ where $L_{\rho}:=G\times^P\CC_{-\rho}$. In later sections, we will identify $\CC[\eccc]$ with the group algebra $\CC[\Q/\Q_P]$ via $q_i\mapsto q^{[\ad_i]}$. Grade $QH_{\Hhat}^{\bl}(G/P)$ by requiring each $q_i$ to have degree $2\langle c_1(G/P),\b_i\rangle=2\sum_{\a\in R^+\setminus R^+_P}\a(\ad_i)$. It is well-known (see e.g. \cite{Mirror}) that $ QH_{\Hhat}^{\bl}(G/P)$ is a graded commutative $H_{\Hhat}^{\bl}(\pt)$-algebra where the algebra structure, called the \textit{$\Hhat$-equivariant quantum cup product}, is defined by 
\[ y_i\star y_j := \sum_{k}\sum_{(d_i)_{i\in\iip}\in\ZZ_{\geqslant 0}^{|\iip|}} \left(\prod_{i\in\iip} q_i^{d_i}\right) \left(\int_{\MM_{0,3}(G/P,\b_{\mathbf{d}})}\ev_1^*y_i\cup\ev_2^*y_j\cup ev_3^*y^{k} \right) y_{k}.\]
Here, $\b_{\mathbf{d}}:=\sum_{i\in\iip} d_i\b_i$ and $\{y_i\}$, $\{y^i\}$ are any $H_{\Hhat}^{\bl}(\pt)$-bases of $H_{\Hhat}^{\bl}(G/P)$ which are dual to each other with respect to the pairing $\int_{G/P}-\cup-$. For $H=T$, we will take $\{y_i\}$ and $\{y^i\}$ to be the \textit{Schubert basis} $\{\s_v\}_{v\in W^P}$ and the \textit{opposite Schubert basis} $\{\s^v\}_{v\in W^P}$ defined by
\begin{align*}
 \s_v &:=\pd\bk{\ol{B_-\dot{v}P/P}} \in H_{\That}^{2\ell(v)}(G/P)\\
  \s^v & :=\pd\bk{\ol{B\dot{v}P/P}}\in H_{\That}^{\dim_{\RR}(G/P)-2\ell(v)}(G/P)
\end{align*}
where $W^P$ is the set of minimal length coset representatives in $W/W_P$. 
%%%%%%%%%%%%%%%%%%%%%%%%%%%%%%%%%%%
%%%%%%%%%%%%%%%%%%%%%%%%%%%%%%%%%%%
%%%%%%%%%%%%%%%%%%%%%%%%%%%%%%%%%%%
\subsection{Quantum $\hslash$-connection}\label{Qconnection} Let $H$ be either $T$ or $G$.
\begin{definition}\label{Qconnectiondef}
Define the \textit{quantum $\hp$-connection}
\[ \qc{H}: QH_{\Hhat}^{\bl}(G/P)[\ecc]\ra \O^1_{\spec\CC[\eccc]}\ot_{\CC[\eccc]} QH_{\Hhat}^{\bl}(G/P)[\ecc]\]
by
\[\qc{H} := \sum_{i\in\iip} dq_i\ot \qc{H}_{\partial_{q_i}}\]
with 
\[ \qc{H}_{\partial_{q_i}} := \hp\partial_{q_i} + q_i^{-1} c_1^{\Hhat}(L_{\w_i})\star -\]
where $\w_i$ is the $i$-th fundamental weight and the canonical $\Hhat$-linearization of $L_{\w_i}:=G\times^P\CC_{-\w_i}$ is used. 
\end{definition}

It is well-known (see e.g. \cite{Mirror, Manin}) that $\qc{H}$ satisfies
\[ \qc{H}_X\circ \qc{H}_Y - \qc{H}_Y\circ \qc{H}_X = \hp \qc{H}_{[X,Y]} \]
for any $X,Y\in\mathfrak{X}(\spec\CC[\eccc])$. It follows that $\qc{H}$ defines a $\dm{\spec\CC[\eccc]}$-module structure on $QH_{\Hhat}^{\bl}(G/P)[\ecc]$.
%%%%%%%%%%%%%%%%%%%%%%%%%%%%%%%%%%%
%%%%%%%%%%%%%%%%%%%%%%%%%%%%%%%%%%%
%%%%%%%%%%%%%%%%%%%%%%%%%%%%%%%%%%%
\subsection{Shift operators}\label{Ashift}
We recall shift operators arising from \cite{Shift1, Shift2, Shift3, Shift4, Shift5}. The reader is recommended to read this subsection after reading Section \ref{GMPbundle} and Section \ref{GMPnonabelianshift} where we will introduce a more general version.

For any $\lambda\in\Q$, define an automorphism $\twist_{\lambda}$ of the $\CC[\hp]$-algebra $H_{\That}^{\bl}(\pt)$ by 
\begin{equation}\label{Ashifttwistdef}
\twist_{\lambda}(f(\w_1,\ldots,\w_r,\hp)):=f(\w_1+\w_1(\lambda)\hp, \ldots, \w_r+\w_r(\lambda)\hp,\hp)
\end{equation}
for any $f(\w_1,\ldots,\w_r,\hp)\in  \CC[\w_1,\ldots,\w_r,\hp]\simeq H_{\That}^{\bl}(\pt)$. For an $H_{\That}^{\bl}(\pt)$-module $N$, denote by $\twist_{\lambda}(N)$ the $H_{\That}^{\bl}(\pt)$-module $N$ with module structure twisted by $\twist_{\lambda}$.

\begin{definition}\label{Ashifttwistedlinear} Let $N_1$ and $N_2$ be $H_{\That}^{\bl}(\pt)$-modules and $\phi:N_1\ra N_2$ a $\CC[\hp]$-linear map. Let $\lambda\in\Q$. We say that $\phi$ is \textit{$\lambda$-twisted $H_{\That}^{\bl}(\pt)$-linear} if it is $H_{\That}^{\bl}(\pt)$-linear as a map from $N_1$ to $\twist_{\lambda}(N_2)$.
\end{definition} 

\begin{definition}\label{Ashiftpredef}
Let $\lambda\in\Q$. Define
\[ \widetilde{\shift}^A_{\lambda}: QH_{\Ghat}^{\bl}(G/P)[\ecc] \ra QH_{\That}^{\bl}(G/P)[\ecc] \]
by
\[ \widetilde{\shift}^A_{\lambda}(y) := [t^{\lambda}]\gmpt y \]
for any $y\in QH_{\Ghat}^{\bl}(G/P)[\ecc]$, where $-\gmpt -$ will be defined in Definition \ref{GMPactiondef}.
\end{definition} 

\begin{lemma} \label{Ashiftwell} (Shift operators)
$\widetilde{\shift}^A_{\lambda}$ extends uniquely to a $(-\lambda)$-twisted $H_{\That}^{\bl}(\pt)$-linear map 
\[ \shift^A_{\lambda}: QH_{\That}^{\bl}(G/P)[\ecc] \ra QH_{\That}^{\bl}(G/P)[\ecc] .\]
\end{lemma}
\begin{proof}
All notations used below can be found in Section \ref{GMPbundle}. By \eqref{GMPlocaleqn}, we have
\[\widetilde{\shift}^A_{\lambda}(y)= \sum_{v\in W^P}\sum_{\eta\in\Q/\Q_P} q^{\eta} \mathcal{I}(t^{\lambda},[t^{\lambda}],y,v,\eta) \s_v\] 
where 
\[ \mathcal{I}(t^{\lambda},[t^{\lambda}],y,v,\eta)  := \int_{[\MM^{0,\infty}(t^{\lambda},\eta)]^{vir}}\ev_{t^{\lambda},\eta,0}^*m_{t^{\lambda}}([t^{\lambda}]\otimes y)  \cup \ev_{t^{\lambda},\eta,\infty}^*\s^v.\]
We have
\[ \mathcal{E}_{t^{\lambda}}(G/P)\simeq \left(\AA^1_z\times G/P\times \{0,\infty\}\right)/_{(z,x,0)\sim (z^{-1}, \lambda(z)\cdot x,\infty)}, \]
and $\That$ acts on $\mathcal{E}_{t^{\lambda}}(G/P)$ in the following way
\[ \left\{ 
\begin{array}{lcl}
(t,\zeta)\cdot [z,x,0] &=& [\zeta z, t\lambda(\zeta)^{-1}\cdot x,0]\\ [.5em]
(t,\zeta)\cdot [z,x,\infty] &=& [\zeta^{-1} z, t\cdot x,\infty]
\end{array}
\right.
\]
for any $t\in T$, $\zeta\in \gm$, $z\in \AA^1$ and $x\in G/P$. It follows that there is a canonical isomorphism $D_{t^{\lambda},0}\simeq G/P$ which is $\That$-equivariant if $\That$ acts on $G/P$ in the following way
\[ (t,\zeta)\cdot x = t\lambda(\zeta)^{-1}\cdot x\qquad t\in T,~ \zeta\in\gm, ~x\in G/P.\]
Hence there is a unique isomorphism of $H_{\That}^{\bl}(\pt)$-modules 
\begin{equation}\label{Ashiftwelleqn1}
H_{\That}^{\bl}(G/P) \xrightarrow{\sim} \twist_{-\lambda}( H_{\That}^{\bl}(D_{t^{\lambda},0}))
\end{equation}
sending the fundamental class of each torus fixed point of $G/P$ to the fundamental class of the corresponding torus fixed point of $D_{t^{\lambda},0}$. It is straightforward to see that $m_{t^{\lambda}}([t^{\lambda}]\ot -)$ is the restriction of \eqref{Ashiftwelleqn1} to $H_{\Ghat}^{\bl}(G/P)$. Therefore, this map, and hence $\widetilde{\shift}^A_{\lambda}$, extends to a $(-\lambda)$-twisted $H_{\That}^{\bl}(\pt)$-linear map as desired.
\end{proof}

\begin{remark}\label{Ashiftrmk}
Up to a constant multiple, $\shift^A_{\lambda}$ is equal to the operator defined in \cite[Definition 3.9]{Shift2}. See the proof of \cite[Lemma 3.14]{me1} for an explanation of this multiple.
\end{remark}

\begin{lemma}\label{Ashiftlemmamulti}
(\cite[Corollary 3.16]{Shift2}) We have $\shift^A_0=\id$ and 
\[ \shift^A_{\lambda_1}\circ \shift^A_{\lambda_2} = \shift^A_{\lambda_1+\lambda_2} \qquad \lambda_1,\lambda_2\in\Q.\]
In particular, each $\shift^A_{\lambda}$ is invertible.  \hfill$\square$
\end{lemma}

\begin{lemma}\label{Ashiftlemmacomm} (\cite[Corollary 3.15]{Shift2})
For any $\lambda\in\Q$ and $i\in\iip$, $\shift^A_{\lambda}$ commutes with $\qc{T}_{\partial_{q_i}}$.  \hfill$\square$
\end{lemma}

%%%%%%%%%%%%%%%%%%%%%%%%%%%%%%%%%%%
%%%%%%%%%%%%%%%%%%%%%%%%%%%%%%%%%%%
%%%%%%%%%%%%%%%%%%%%%%%%%%%%%%%%%%%
%%%%%%%%%%%%%%%%%%%%%%%%%%%%%%%%%%%
%%%%%%%%%%%%%%%%%%%%%%%%%%%%%%%%%%%
%%%%%%%%%%%%%%%%%%%%%%%%%%%%%%%%%%%
\section{B-model}\label{Rietsch}

\subsection{Rietsch mirror} \label{Rietschmirror}
We recall the definition of the Rietsch mirror. Our version is close to the one defined in \cite[Section 6]{LT}. See also \cite[Section 4]{Rietsch} and \cite[Section 6.6]{T1}.

Fix generators $e^{\vee}_i$ ($1\leqslant i \leqslant r$) of $\ggd_{\ad_i}$ and let $f^{\vee}_i$ be the generators of $\ggd_{-\ad_i}$ satisfying $[e^{\vee}_i, f^{\vee}_i]=\a_i$. Denote by $x_i,y_i:\ga\ra\Gd$ the unique group homomorphisms satisfying $\lie(x_i)(1)=e^{\vee}_i$ and $\lie(y_i)(1)=f^{\vee}_i$. Define
\[ \dot{s}_{\a_i} := y_i(-1)x_i(1)y_i(-1).\]
It is well-known that $\dot{s}_{\a_i}\in N(\Td)$ and represents $s_{\a_i}\in W\simeq N(\Td)/\Td$. Moreover, these elements satisfy the standard braid relations so that we can define $\dot{w}\in N(\Td)$ for any $w\in W$ which represents $w\in W$ such that $\dot{w} = \dot{s}_{\a_{i_1}}\cdots \dot{s}_{\a_{i_{\ell(w)}}}$ whenever $w= s_{\a_{i_1}}\cdots s_{\a_{i_{\ell(w)}}}$ is a reduced word decomposition.

There is a unique character 
\begin{equation}\label{cchhdef}
\cchh :\nd \ra \CC
\end{equation}
satisfying $\cchh(f^{\vee}_i)=1$ for any $1\leqslant i \leqslant r$. It gives rise to a group homomorphism
\[ e^{\cchh}: \Umd \ra \ga.\]
We can also define $\cchh$ in the following equivalent way. Let $Kil(-,-)$ be the Killing form of $\ggd$ and $\a_0$ the highest positive root. Define 
\begin{equation}\label{betaande}
\b(-,-):= \frac{2 Kil(-,-)}{Kil(\a_0,\a_0)} \quad\text{ and }\quad e:=\sum_{i=1}^r|\ad_i|^2e_i^{\vee}
\end{equation}
where $|\ad_i|^2:= Kil(\a_0,\a_0)/Kil(\a_i,\a_i)=2\b(\a_i,\a_i)^{-1}$. Then we have
\[ \cchh = \b(e,-)|_{\nd}.\]

\begin{remark}\label{Rietschmirrorrmk}
Notice that $\dot{w}$ and $\cchh$, and hence the Rietsch mirror we are going to define, depend on the generators $e_i^{\vee}$. But since any two sets of generators are $\adjoint$-related to each other, the Rietsch mirror does not depend on them up to isomorphism. For convenience, we will take $e_i^{\vee}$ to be $x_i$ from \cite[Proposition 5.6]{YZ} or $e_{\ad_i}$ from \cite[Section 4.1]{me2}. Roughly speaking, these generators are characterized by the property that the element $e:=\sum_{i=1}^r|\ad_i|^2 e_i^{\vee}$ corresponds to the cup product by the positive generator of $H^2(\ag;\ZZ)$ via the geometric Satake equivalence \cite{MV}.
\end{remark}

\begin{definition}\label{Rietschmirrordef}
The \textit{Rietsch mirror} of $G/P$ is a quintuple $(\Xd,W,\pi,p,\vol)$ where
\begin{enumerate}
\item \[ \Xd:= \{ (g,t)\in \Gd\times\ZL|~g\in \Bd\cap \Umd (\wpd)^{-1} t\Umd\} ;\]

\item
\[ W:\Xd\ra\AA^1,\quad \pi:\Xd\ra \ZL,\quad p:\Xd\ra \Td\]
are morphisms defined by
\[ W(x):=e^{\cchh}(u_1)+e^{\cchh}(u_2),\quad \pi(x):=t,\quad p(x):=t_0\]
for any $x=(g,t)\in \Xd$ with 
\[ g = u_0t_0=u_1(\wpd)^{-1}t u_2\]
($u_0\in\Ud$, $t_0\in\Td$, $u_1,u_2\in\Umd$ ); and 

\item 
\[\vol\in \O^{top}(\Xd/\ZL)\]
is a fiberwise volume form with respect to $\pi$ whose definition is postponed to Definition \ref{Fiberwisevolumeformdef} in Section \ref{Fiberwisevolumeform}.
\end{enumerate}

\end{definition}

\begin{remark}\label{Rietschwell}
Notice that the point $t_0$ from Definition \ref{Rietschmirrordef}(2) is unique, and hence $p$ is well-defined. On the other hand, $u_1$ and $u_2$ are not unique. But one can show that $e^{\cchh}(u_1)+e^{\cchh}(u_2)$ depends only on $x$, and hence $W$ is also well-defined.
\end{remark}

\begin{remark}\label{RietschcomparetoLT}
Lam and Templier \cite{LT} defined the Rietsch mirror to be Berenstein-Kazhdan's \textit{parabolic geometric crystal} \cite{BK}. In Appendix \ref{C}, we will show that these two definitions are equivalent.
\end{remark}

Define  
\[ \RP := \Ud\Bmd/\Bmd\cap \Umd (\wpd)^{-1}\Bmd/\Bmd \subseteq \GB.\]
Let $\Pmd$ be the parabolic subgroup of $\Gd$ satisfying
\[ \lie(\Pmd) = \lie(\Bmd)\oplus \bigoplus_{\a\in R^+_P}\ggd_{\ad}.\]
It is not difficult to see that the restriction of the projection $\GB\ra\Gd/\Pmd$ to $\RP$ is an isomorphism onto its image which is an open subscheme of $\Gd/\Pmd$. We denote the image by $\UP$. In the literature, $\RP$ (resp. $\UP$) is called an open (resp. open projected) Richardson variety.

\begin{lemma} \label{Rietschfiblemma}
The morphism $\Gd\times\ZL\ra \Gd/\Pmd\times\ZL$ defined by $(g,t)\mapsto (g\Pmd,t)$ induces an isomorphism
\[ \nu:\Xd\xrightarrow{\sim} \UP\times\ZL\]
of $\ZL$-schemes.
\end{lemma}
\begin{proof}
The proof is straightforward and left to the reader. 
\end{proof}

Finally, we define a $\gm$-action on $\Xd$ as in \cite[Section 6.21]{LT}. Let $\gm$ act on 
\begin{itemize}
\item $\Gd$ via the conjugate action by the cocharacter $2\rho = \sum_{\a\in R^+}\a$;

\item $\ZL$ via the multiplication by the cocharacter $-2\sum_{\a\in R^+\setminus R_P^+}\a$;

\item $\Td$ via the trivial action; and

\item $\AA^1$ via the multiplication of weight $-2$.
\end{itemize}

\begin{lemma}\label{Rietschmirroraction} The $\gm$-action on $\Gd\times\ZL$ induces a $\gm$-action on $\Xd$ such that $W$, $\pi$ and $p$ are equivariant.
\end{lemma}
\begin{proof}
The proof is similar to the proof of \cite[Proposition 6.24]{LT}.
\end{proof}
%%%%%%%%%%%%%%%%%%%%%%%%%%%%%%%%%%%
%%%%%%%%%%%%%%%%%%%%%%%%%%%%%%%%%%%
\subsection{Brieskorn lattice} \label{Brieskornlattice}
We define the Brieskorn lattice associated to the quadruple $(\Xd,W,\pi,p)$ defined in Definition \ref{Rietschmirrordef}. (The fiberwise volume form $\vol$ is not needed here.) It is the zeroth cohomology of $WCr^{1/\hp}_{(G,P)}$ defined in \cite{LT}. See also \cite{MR, PR, PRW}. 

\begin{definition}\label{RietschBrieskorndef}$~$
\begin{enumerate}
\item Define the \textit{Brieskorn lattice} associated to $(\Xd,W,\pi,p)$
\[ \quad \bries := \coker\left( \utd\ot \O^{top-1}(\Xd/\ZL)\xrightarrow{\partial}\utd\ot  \O^{top}(\Xd/\ZL) \right)\]
where 
\begin{itemize}
\item $\O^i(\Xd/\ZL)$ is the space of relative $i$-forms on $\Xd$ with respect to $\pi$; and

\item $\partial$ is defined by
\begin{equation}\label{Rietschpartialdef}
  \qquad \qquad \quad \partial(z\ot \w):= z\ot (\hp d\w + dW\wedge\w) - \sum_i zh_i\ot(p^*\langle h^i,\mccccc_{\Td} \rangle)\wedge\w 
\end{equation}
where $\{h_i\}\subset \ttd$ and $\{h^i\}\subset (\ttd)^*$ are dual bases and $\mccccc_{\Td}\in\Omega^1(\Td;\ttd)$ is the Maurer-Cartan form of $\Td$.
\end{itemize}
\item Define the \textit{Gauss-Manin $\hp$-connection} on $\bries$ which is a $\dm{\ZL}$-module structure by the formula
\[ \eta\cdot [z\ot \w]:= \left[z\ot(\hp\LL_{\widetilde{\eta}}\w + (\LL_{\widetilde{\eta}}W)\w) - \sum_i zh_i\ot (\iota_{\widetilde{\eta}} p^*\langle h^i,\mccccc_{\Td} \rangle)\w \right]\]
for any $\eta\in\mathfrak{X}(\ZL)$ and $[z\ot \w]\in \bries$ where $\widetilde{\eta}\in \mathfrak{X}(\Xd)$ is a lift of $\eta$. 
\end{enumerate}
\end{definition}

We define a grading on $\bries$ as in \cite[Section 11]{LT}. Recall the $\gm$-action on $\Xd$ (see Lemma \ref{Rietschmirroraction}). Since $\pi$ is equivariant, the $\gm$-action on $\Xd$ induces a linear $\gm$-action, and hence a grading, on $\O^{\bl}(\Xd/\ZL)$. (Warning: we do not take the de Rham degree into account.) Grade $\utd$ by requiring $\hp$ and every element of $\ttd$ to have degree 2. Then $\utd\ot\O^{\bl}(\Xd/\ZL)$ has a grading. Consider $\partial$ defined in \eqref{Rietschpartialdef}. Since $W$ is equivariant, $dW$ has degree 2, and since $p$ is equivariant, $p^*\langle h^i,\mccccc_{\Td} \rangle$ has degree 0. It follows that $\partial(z\ot \w)$ is homogeneous whenever $z\ot\w$ is. We have proved
\begin{lemma}\label{Brieskornlatticegrading}
The grading on $\utd\ot\O^{top}(\Xd/\ZL)$ induces a grading on $\bries$. \hfill $\square$
\end{lemma}
%%%%%%%%%%%%%%%%%%%%%%%%%%%%%%%%%%%
%%%%%%%%%%%%%%%%%%%%%%%%%%%%%%%%%%%
\subsection{Jacobi ring} \label{Jacobialgebra}

\begin{definition}\label{RietschJacobidef} 
Define the \textit{Jacobi ring} $\jac(\Xd,W,\pi,p)$ to be the coordinate ring of the scheme-theoretic zero locus of the relative 1-form 
\[ \pr_{\Xd}^*dW -\langle\pr_{\mathfrak{t}} , (p\circ \pr_{\Xd})^*\mccccc_{\Td}\rangle \in \O^1(\Xd\times\mathfrak{t}/\ZL\times\mathfrak{t})\]
where $\pr_{\Xd}:\Xd\times \mathfrak{t}\ra\Xd$ and $\pr_{\mathfrak{t}}:\Xd\times \mathfrak{t}\ra\mathfrak{t}$ are the projections.
\end{definition}

Let $\FF_{\hp=0}$ denote the functor $N\mapsto N/\hp N$.
\begin{lemma}\label{RietschJacobilemma}
Every fiberwise volume form on $\Xd$ induces an isomorphism
\[\jac(\Xd,W,\pi,p)\xrightarrow{\sim}\FF_{\hp=0}(\bries) \] 
of $\sym^{\bl}(\ttd)\ot\OO(\ZL)$-modules.
\end{lemma}
\begin{proof}
Identify $\O^{top}(\Xd/\ZL)$ with $\OO(\Xd)$ using the given fiberwise volume form. By \eqref{Rietschpartialdef}, the $\sym^{\bl}(\ttd)\ot\OO(\ZL)$-module $\FF_{\hp=0}(\bries)$ is isomorphic to the quotient of $\sym^{\bl}(\ttd)\ot\OO(\Xd)$ by the ideal generated by elements of the form
\[ 1\ot \LL_{\zeta_j}W- \sum_{i} h_i\ot\iota_{\zeta_j}p^*\langle h^i,\mccccc_{\Td}\rangle\]
where $\{\zeta_j\}$ is a global frame of $\Xd$ relative to $\ZL$. It is clear that the latter module is in fact equal to $\jac(\Xd,W,\pi,p)$.
\end{proof}
%%%%%%%%%%%%%%%%%%%%%%%%%%%%%%%%%%%
%%%%%%%%%%%%%%%%%%%%%%%%%%%%%%%%%%%
\subsection{Fiberwise volume form} \label{Fiberwisevolumeform}
We finish the definition of the Rietsch mirror (Definition \ref{Rietschmirrordef}) by defining the fiberwise volume form $\vol$ on $\Xd$.

Recall the open projected Richardson variety $\UP\subseteq \Gd/\Pmd$ defined in Section \ref{Rietschmirror} (before Lemma \ref{Rietschfiblemma}). By \cite[Lemma 5.4]{KLS}, the complement $D:= (\Gd/\Pmd)\setminus \UP$ has pure codimension one, and there exists a volume form $\w_{\UP}$ on $\UP$ (unique up to scalar) which has simple pole along every irreducible component of $D$.

\begin{definition}\label{Fiberwisevolumeformdef} (Definition \ref{Rietschmirrordef} continued) Define $\vol\in \O^{top}(\Xd/\ZL)$ by 
\[ \vol:= (\pr_{\UP}\circ \nu)^*\w_{\UP}\]
where $\nu:\Xd\xrightarrow{\sim}\UP\times\ZL$ is the isomorphism from Lemma \ref{Rietschfiblemma} and $\pr_{\UP}:\UP\times\ZL\ra\UP$ is the projection.
\end{definition}

\begin{remark}\label{Fiberwisevolumeformrmkscalar}
Notice that $\vol$ is only defined up to scalar. It can be normalized by choosing an orientation of an integration cycle. See Remark \ref{Fiberwisevolumeformrmk} below. But for our purpose, such normalization is unnecessary.
\end{remark}

\begin{remark}\label{Fiberwisevolumeformrmk} The above definition is due to Lam and Templier \cite[Section 6.6]{LT}. See also \cite[Section 8]{MR} for the case of Grassmannian. It is a priori different from the one by Rietsch \cite[Section 7]{Rietsch} which uses, instead of $\w_{\UP}$, the unique (up to sign) volume form whose restrictions to every torus chart is equal to the standard volume form $\pm dt_1\wedge \cdots\wedge dt_N/ t_1\cdots t_N$. But Lam \cite[Lemma 2.9]{LGC} proved that these two definitions are the same (up to scalar). (Strictly speaking, he only dealt with the case $P=B$ but his proof clearly works for the general case.)
\end{remark}

\begin{lemma}\label{Fiberwisevolumeformdegzero} The degree of $[\vol]$ with respect to the grading from Lemma \ref{Brieskornlatticegrading} is 0.
\end{lemma}
\begin{proof}
This can be proved in three ways. First, it follows immediately from \cite[Lemma 6.26]{LT}. Second, by Remark \ref{Fiberwisevolumeformrmk}, $\vol$ is locally equal to a scalar multiple of $dt_1\wedge \cdots\wedge dt_N/ t_1\cdots t_N$ which is easily seen to have degree 0. Third, we have $\vol=\ol{f}^{-1}\ol{\w}_{\UUU}$ where $\ol{f}$ and $\ol{\w}_{\UUU}$ come from the proof of Lemma \ref{ApplemmaA}. The result can be proved by computing their degrees directly. Details are left to the reader.
\end{proof}
%%%%%%%%%%%%%%%%%%%%%%%%%%%%%%%%%%%
%%%%%%%%%%%%%%%%%%%%%%%%%%%%%%%%%%%
\subsection{Shift operators} \label{Bshift}
Following \cite[Proposition 3.20]{Iritanibig}, we define shift operators for the Brieskorn lattice $\bries$ as follows.

For any $\lambda\in\Q$, define an automorphism $\twist_{\lambda}$ of the $\CC[\hp]$-algebra $\utd$ by 
\[ \twist_{\lambda}(x) := x+\hp\lambda(x),\quad x\in\ttd.\]
For a $\utd$-module $N$, denote by $\twist_{\lambda}(N)$ the $\utd$-module $N$ with module structure twisted by $\twist_{\lambda}$.

\begin{definition}\label{Bshifttwistedlinear} Let $N_1$ and $N_2$ be $\utd$-modules and $\phi:N_1\ra N_2$ a $\CC[\hp]$-linear map. Let $\lambda\in\Q$. We say that $\phi$ is \textit{$\lambda$-twisted $\utd$-linear} if it is $\utd$-linear as a map from $N_1$ to $\twist_{\lambda}(N_2)$.
\end{definition} 

\begin{definition}\label{Bshiftpredef}
Let $\lambda\in\Q$. Define
\[ \widetilde{\shift}^B_{\lambda}: \utd\ot\O^{top}(\Xd/\ZL)\ra \utd\ot\O^{top}(\Xd/\ZL)  \]
by
\[ \widetilde{\shift}^B_{\lambda}(z\ot\w) := \twist_{-\lambda}(z)\ot (\lambda\circ p)\w \]
for any $z\in\utd$ and $\w\in\O^{top}(\Xd/\ZL) $. It is $(-\lambda)$-twisted $\utd$-linear in the sense of Definition \ref{Bshifttwistedlinear}.
\end{definition} 

\begin{lemma} \label{Bshiftwell} (Shift operators)
$\widetilde{\shift}^B_{\lambda}$ descends to a $(-\lambda)$-twisted $\utd$-linear map 
\[ \shift^B_{\lambda}:\bries\ra\bries \]
\end{lemma}
\begin{proof}
It suffices to show 
\[  \widetilde{\shift}^B_{\lambda}(\partial(z\ot\w))=\partial(\twist_{-\lambda}(z)\ot (\lambda\circ p)\w)\]
for any $z\in\utd$ and $\w\in\O^{top-1}(\Xd/\ZL) $, where $\partial$ is defined in \eqref{Rietschpartialdef}. We have
\begin{align*}
 \widetilde{\shift}^B_{\lambda}(\partial(z\ot\w))= ~&~
 \twist_{-\lambda}(z)\ot (\hp (\lambda\circ p)d\w + (\lambda\circ p) dW\wedge\w) \\
 & ~-\sum_i \twist_{-\lambda}(zh_i)\ot (\lambda\circ p)( p^*\langle h^i,\mccccc_{\Td}\rangle)\wedge\w\\
 =~&~ \partial(\twist_{-\lambda}(z)\ot (\lambda\circ p)\w) - \hp \twist_{-\lambda}(z)\ot d(\lambda\circ p)\wedge\w\\
 & ~+ \sum_i \twist_{-\lambda}(z)(h_i-\twist_{-\lambda}(h_i))\ot (\lambda\circ p) (p^*\langle h^i,\mccccc_{\Td}\rangle)\wedge\w \\
 =~&~ \partial(\twist_{-\lambda}(z)\ot (\lambda\circ p)\w)  \\
 & ~-\hp  \twist_{-\lambda}(z)\ot (d(\lambda\circ p)-(\lambda\circ p)p^*\langle \lambda,\mccccc_{\Td}\rangle)\wedge\w.
\end{align*}
It is not difficult to see that $d(\lambda\circ p)=(\lambda\circ p)p^*\langle \lambda,\mccccc_{\Td}\rangle$, and hence the last expression is equal to $\partial(\twist_{-\lambda}(z)\ot (\lambda\circ p)\w)$ as desired.
\end{proof}

The following are the parallel properties of the A-model shift operators from Section \ref{Ashift}.

\begin{lemma}\label{Bshiftlemmamulti}
We have $\shift^B_0=\id$ and 
\[ \shift^B_{\lambda_1}\circ \shift^B_{\lambda_2} = \shift^B_{\lambda_1+\lambda_2} \qquad \lambda_1,\lambda_2\in\Q.\]
In particular, each $\shift^B_{\lambda}$ is invertible. 
\end{lemma}
\begin{proof}
This is clear from definition.
\end{proof}

\begin{lemma}\label{Bshiftlemmacomm} 
For any $\lambda\in\Q$ and $i\in I\setminus I_P$, $\shift^B_{\lambda}$ is $\dm{\ZL}$-linear.
\end{lemma}
\begin{proof}
This is straightforward. Since we will not use this result, we omit the details.
\end{proof}
%%%%%%%%%%%%%%%%%%%%%%%%%%%%%%%%%%%
%%%%%%%%%%%%%%%%%%%%%%%%%%%%%%%%%%%

%%%%%%%%%%%%%%%%%%%%%%%%%%%%%%%%%%%
%%%%%%%%%%%%%%%%%%%%%%%%%%%%%%%%%%%
\section{Non-abelian shift operators}\label{GMP}
\textit{Non-abelian shift operators} are defined by Gonz\'alez, Mak and Pomerleano \cite{GMP} in the symplectic category. They generalize \textit{abelian shift operators} (see Section \ref{Ashift}) and \textit{Savelyev-Seidel's homomorphisms} \cite{Savelyev1, Savelyev2, Savelyev3, Seidel} simultaneously. In this section, we give an exposition in the algebraic category, where the target space is restricted to $G/P$. We also prove some computational results based on \cite{me1, me2}.

\subsection{$G/P$-bundles and moduli of sections}\label{GMPbundle}
By Beauville-Laszlo's theorem \cite{BL}, the affine Grassmannian $\ag$ parametrizes $G$-torsors over $\PP^1$ with a trivialization over $\PP^1\setminus 0$. In other words, every morphism $f:\G\ra\ag$ is represented by a $G$-torsor $\fib{f}$ over $\PP^1\times \G$ with a trivialization $\nu_f:\fib{f}|_{(\PP^1\setminus 0)\times \G}\xrightarrow{\sim}  (\PP^1\setminus 0)\times \G\times G$. In what follows, we assume $\G$ is a smooth projective variety. Then the associated fiber bundle
\[ \ffib{f}:=\fib{f}\times^G G/P\]
exists as a smooth projective variety. Denote by $\pi_f:\ffib{f}\ra\PP^1\times\G$ the projection. Define 
\[D_{f,0}:=\pi_f^{-1}(0\times \G)\quad\text{ and }\quad D_{f,\infty}:=\pi_f^{-1}(\infty\times \G)\]
which are smooth divisors on $\ffib{f}$. Denote by $\iota_{f,0/\infty}:D_{f,0/\infty}\hookrightarrow \ffib{f}$ the inclusions. The trivialization $\nu_f$ induces an isomorphism $D_{f,\infty}\simeq \G\times G/P$. Let $\rho\in (\Q/\Q_P)^*$. Define $L_{\rho}:=G\times^P\CC_{-\rho}$ which is a line bundle on $G/P$. Since $L_{\rho}$ has a (unique) $G$-linearization, there is a natural line bundle $\LL_{\rho}$ on $\ffib{f}$, constructed by descent, whose restriction to every fiber of $\pi_f$ is isomorphic to $L_{\rho}$. 

\begin{definition}\label{GMPmodulidef}
Let $f:\G\ra\ag$ be a morphism and $\eta\in\Q/\Q_P$.
\begin{enumerate}
\item  Define 
\[\MM^{0,\infty}(f,\eta):= \bigcup_{\b} ~\MM_{0,2}(\ffib{f},\b)\times_{((\ev_1,\ev_2),\iota_{f,0}\times\iota_{f,\infty})} (D_{f,0}\times D_{f,\infty})\]
where $\b$ runs over the set of elements of $H_2(\ffib{f})$ satisfying $(\pi_f)_*\b=[\PP^1\times\pt]$ and $\langle c_1(\LL_{\rho}),\b\rangle = \rho(\eta)$ for any $\rho\in (\Q/\Q_P)^*$. 

\item Define 
\[ \ev_{f,\eta,0}:\MM^{0,\infty}(f,\eta) \ra D_{f,0}\]
to be the morphism induced by $\ev_1$.

\item  Define 
\[\ev_{f,\eta,\infty}:\MM^{0,\infty}(f,\eta) \ra D_{f,\infty}\simeq \G\times G/P \ra G/P\]
to be the composition of the morphism induced by $\ev_2$, the isomorphism induced by $\nu_f$ and the projection.
\end{enumerate}
\end{definition}
\begin{remark}\label{GMPmodulirmk}
The superscripts $0,\infty$ in $\MM^{0,\infty}(f,\eta)$ refer to the marked points $0,\infty\in \PP^1$ while the subscripts $0,2$ in $\MM_{0,2}(\ffib{f},\b)$ refer to the genus ($=0$) and the number of marked points ($=2$) as usual.
\end{remark}

Let $H$ be either $T$ or $G$. Put $\Hhat:= H\times \gm$. Let $f:\G\ra\ag$ be a morphism. Suppose $\G$ has an $\Hhat$-action and $f$ is $\Hhat$-equivariant. By modifying the proof of \cite[Lemma 3.5]{me1}, we see that $\ffib{f}$ has an $\Hhat$-action such that $\pi_f$ is $\Hhat$-equivariant, $D_{f,0}$ and $D_{f,\infty}$ are $\Hhat$-invariant and the isomorphism $D_{f,\infty}\simeq \G\times G/P$ induced by $\nu_f$ is $\Hhat$-equivariant. (Here, $H$ acts trivially on $\PP^1$ and $\gm$ acts on $\PP^1$ by rotation such that $0\in\PP^1$ has weight 1.) It follows that the stack $\MM^{0,\infty}(f,\eta)$ has a natural $\Hhat$-action and the morphisms $\ev_{f,\eta,0}$ and $\ev_{f,\eta,\infty}$ are $\Hhat$-equivariant.

We will use the stacks $\MM^{0,\infty}(f,\eta)$ to define the non-abelian shift operators. Before doing so, we need to define the following convolution-type operation. As before, let $H$ be either $T$ or $G$ and $f:\G\ra\ag$ an $\Hhat$-equivariant morphism.
\begin{definition}(C.f. \cite[end of Section 3.5]{GMP}) \label{GMPtwist}
Define an $H_{\Hhat}^{\bl}(\pt)$-linear map
\[ m_f: H^{\Hhat}_{-\diamond}(\G)\otimes_{H_{\Ghat}^{\bl}(\pt)} H_{\Ghat}^*(G/P)\ra H_{\Hhat}^{\dim_{\RR}\G+\diamond+*}(D_{f,0}) \]
to be the composition
\begin{align*}
 m_f : ~& H^{\Hhat}_{-\diamond}(\G)\otimes_{H_{\Ghat}^{\bl}(\pt)} H_{\Ghat}^*(G/P) \xrightarrow[(a)]{\sim} H_{H\times \Ghat}^{\dim_{\RR}\G+\diamond}(\fib{f,0})\otimes_{H_{\Ghat}^{\bl}(\pt)} H_{\Ghat}^{*}(G/P)\xrightarrow[(b)]{} \\
 &  H_{H\times(G\times\gm)}^{\dim_{\RR}\G+\diamond+*}( \fib{f,0}\times G/P) \xrightarrow[(c)]{\sim} H_{\Hhat}^{\dim_{\RR}\G+\diamond+*}(\fib{f,0}\times^G G/P) \xrightarrow[(d)]{\sim} H^{\dim_{\RR}\G+\diamond+*}_{\Hhat}(D_{f,0})
\end{align*}
where $\fib{f,0}:=\fib{f}|_{0\times \G}$ and the above morphisms are induced by
\begin{enumerate}[(a)]
\item the sequence of canonical isomorphisms 
\begin{equation}\nonumber\label{GMPtwisteq1}
H^{\Hhat}_{-\diamond}(\G)\simeq H_{\Hhat}^{\dim_{\RR}\G+\diamond}(\G)\simeq  H_{\Hhat}^{\dim_{\RR}\G+\diamond}(\fib{f,0}/G) \simeq H_{\Hhat\times G}^{\dim_{\RR}\G+\diamond}(\fib{f,0})\simeq H_{H\times\Ghat}^{\dim_{\RR}\G+\diamond}(\fib{f,0});
\end{equation} 

\item the external cup product ($G\times\gm$ acts on $\fib{f,0}\times G/P$ diagonally);

\item the free group action on $\fib{f,0}\times G/P$ by the diagonal subgroup of $G\times G$; and

\item the canonical isomorphism $\fib{f,0}\times^G G/P\simeq D_{f,0}$.
\end{enumerate}
\end{definition}

Recall the opposite Schubert basis $\{\s^v\}_{v\in W^P}$ of $H_{\That}^{\bl}(G/P)$ defined in Section \ref{Qcohomology}.
\begin{definition}\label{GMPintegral}
Let $f:\G\ra\ag$ be a $\That$-equivariant morphism. For any $x\in H_{-\bl}^{\That}(\G)$, $y\in H_{\Ghat}^{*}(G/P)$, $v\in W^P$ and $\eta\in\Q/\Q_P$, define
\[\mathcal{I}(f,x,y,v,\eta):=\int_{[\MM^{0,\infty}(f,\eta)]^{vir}}\ev_{f,\eta,0}^*m_{f}(x\otimes y)  \cup \ev_{f,\eta,\infty}^*\s^v\in H^{\diamond}_{\That}(\pt)\] 
where $ \diamond:= \bl+*-2\ell(v)-2\sum_{\a\in R^+\setminus R^+_P}\a(\eta)$.
\end{definition}

%%%%%%%%%%%%%%%%%%%%%%%%%%%%
%%%%%%%%%%%%%%%%%%%%%%%%%%%%
%%%%%%%%%%%%%%%%%%%%%%%%%%%%
\subsection{Definition}\label{GMPnonabelianshift} 
Recall the affine Schubert basis $\{\ascl_{\wl}\}_{\wl\in W_{af}^-}$ and the Schubert basis $\{\s_v\}_{v\in W^P}$ defined in Section \ref{notationG} and Section \ref{Qcohomology} respectively.
\begin{definition}\label{GMPactiondef}
Define an $H_{\That}^{\bl}(\pt)[\eccc]$-linear map
\[ -\gmpt -: H_{-\diamond}^{\That}(\ag)\otimes_{H_{\Ghat}^{\bl}(\pt)} QH_{\Ghat}^{*}(G/P)[\ecc]\ra QH_{\That}^{\diamond+*}(G/P)[\ecc] \]
by
\begin{equation} \label{GMPactiondefeq}
\ascl_{\wl} \gmpt y := \sum_{v\in W^P}\sum_{\eta\in\Q/\Q_P} q^{\eta} \mathcal{I}(f_{\ag,\wl},[\G_{\wl}],y,v,\eta) \s_v, 
\end{equation}
where $\mathcal{I}(f_{\ag,\wl},[\G_{\wl}],y,v,\eta)$ is the integral defined in Definition \ref{GMPintegral} with $f$ taken to be a $\That$-equivariant morphism $f_{\ag,\wl}:\G_{\wl}\ra \ag$ which factors through a resolution $\G_{\wl}\ra \ol{\BB\cdot t^{w(\lambda)}}$.
\end{definition}

\begin{lemma}\label{GMPactiongraded}
$-\gmpt -$ is graded.
\end{lemma}
\begin{proof}
By Definition \ref{GMPintegral}, the integral $\mathcal{I}(f_{\ag,\wl},[\G_{\wl}],y,v,\eta)$ from \eqref{GMPactiondefeq} has degree $-2\ell(\wl)+\deg y -2\ell(v) - 2\sum_{\a\in R^+\setminus R^+_P}\a(\eta)$. Since $q^{\eta}$ and $\s_v$ has degree $2\sum_{\a\in R^+\setminus R^+_P}\a(\eta)$ and $2\ell(v)$ respectively, $\ascl_{\wl} \gmpt y$ has degree $-2\ell(\wl)+\deg y$ which is the degree of $\ascl_{\wl} \ot y$ in $H_{-\bl}^{\That}(\ag)\ot_{H_{\Ghat}^{\bl}(\pt)} QH_{\Ghat}^{\bl}(G/P)[\ecc]$.
\end{proof}

Let $f:\G\ra\ag$ be a $\That$-equivariant morphism. By a localization argument (see e.g. the proof of \cite[Proposition 3.12]{me1}), we have 
\begin{equation}\label{GMPlocaleqn}
f_*x\gmpt y= \sum_{v\in W^P}\sum_{\eta\in\Q/\Q_P} q^{\eta} \mathcal{I}(f,x,y,v,\eta) \s_v
\end{equation}
for any $x\in H_{-\diamond}^{\That}(\G)$ and $y\in QH_{\Ghat}^{*}(G/P)[\ecc]$. Suppose further $f$ is $\Ghat$-equivariant. Then for any $\eta\in\Q/\Q_P$, the stack $\MM^{0,\infty}(f,\eta)$ has a natural $\Ghat$-action such that $\ev_{f,\eta,0}$ and $\ev_{f,\eta,\infty}$ are $\Ghat$-equivariant. It follows that, by \eqref{GMPlocaleqn}, we have 
\[ f_*(w\cdot x)\gmpt y = w\cdot (f_*x\gmpt y)\]
for any $w\in W$. Since every $\agll$ has a $\Ghat$-equivariant resolution, we have proved
\begin{lemma}\label{GMPpreserveW} 
$-\gmpt -|_{H_{-\diamond}^{\Ghat}(\ag)\otimes_{H_{\Ghat}^{\bl}(\pt)} QH_{\Ghat}^{*}(G/P)[\ecc]}$ lands in $QH_{\Ghat}^{\diamond+*}(G/P)[\ecc]$. \hfill$\square$
\end{lemma}

\begin{definition} (Non-abelian shift operators) Define
\[ -\gmpg -: H_{-\diamond}^{\Ghat}(\ag)\otimes_{H_{\Ghat}^{\bl}(\pt)} QH_{\Ghat}^{*}(G/P)[\ecc]\ra QH_{\Ghat}^{\diamond+*}(G/P)[\ecc] \]
to be the map induced by $-\gmpt-$ via Lemma \ref{GMPpreserveW}.
\end{definition}

%%%%%%%%%%%%%%%%%%%%%%%%%%%%
%%%%%%%%%%%%%%%%%%%%%%%%%%%%
%%%%%%%%%%%%%%%%%%%%%%%%%%%%
\subsection{Some properties}\label{GMPnonabelianshiftproperties}
Let $H$ be either $T$ or $G$.
\begin{proposition}\label{GMPactionmodule}
For any $x_1\in H^{\Hhat}_{\bl}(\ag)$, $x_2\in H^{\Ghat}_{\bl}(\ag)$ and $y\in QH^{\bl}_{\Ghat}(G/P)[\ecc]$, we have 
\[ x_1\gmph (x_2\gmpg y) = (x_1\aggh x_2)\gmph y\]
where $-\aggh-$ is defined in Section \ref{notationG}.
\end{proposition}
\begin{proof}
This is the algebraic version of \cite[Theorem 4.3]{GMP}. Our proof is nothing but a translation of the proof given therein.

By localization, it suffices to assume $H=T$ and $x_1=[t^{\mu_0}]$ for some $\mu_0\in\Q$. There exist $p_0,p_1,\ldots,p_N\in H_{\That}^{\bl}(\pt)$ with $p_0\ne 0$ and $\mu_1,\ldots,\mu_N\in\Q$ such that
\[ p_0x_2 = \sum_{j=1}^N p_j [t^{\mu_j}].\]
For $j=0,\ldots,N$, put $p_j':=\twist_{-\mu_0}(p_j)$ where $\twist_{-\mu_0}$ is defined in \eqref{Ashifttwistdef}. Using the fact that $\shift^A_{\mu_0}$ is $(-\mu_0)$-twisted $H_{\That}^{\bl}(\pt)$-linear (Lemma \ref{Ashiftwell}), we have
\begin{align*} 
 p'_0 x_1\gmpt (x_2\gmpg y) & = p'_0 \shift^A_{\mu_0}(x_2\gmpg  y) = \shift^A_{\mu_0}(p_0x_2\gmpg y)\\
 & = \sum_{j=1}^N\shift^A_{\mu_0}(p_j\shift^A_{\mu_j}(y)) = \sum_{j=1}^Np'_j\shift^A_{\mu_0}\circ \shift^A_{\mu_j}(y).
\end{align*}
By Lemma \ref{Ashiftlemmamulti}, we have $\shift^A_{\mu_0}\circ \shift^A_{\mu_j}=\shift^A_{\mu_0+\mu_j}$, and hence
\[ p_0'x_1\gmpt(x_2\gmpg y)=\sum_{j=1}^N p'_j \shift^A_{\mu_0+\mu_j}(y) = \left( \sum_{j=1}^N p'_j [t^{\mu_0+\mu_j}]\right)\gmpg y.   \]
By the definition of $-\aggt-$ (see Section \ref{notationG}), we have $\sum_{j=1}^N p_j'[t^{\mu_0+\mu_j}]=p_0'x_1\aggt x_2$, and hence
\[ p'_0 x_1\gmpt (x_2\gmpg y) =p'_0( x_1\aggt x_2)\gmpg y. \]
The result follows.
\end{proof}

\begin{proposition}\label{GMPactioncommute}
For any $x\in H^{\Hhat}_{\bl}(\ag)$, $y\in QH_{\Ghat}^{\bl}(G/P)[\ecc]$ and $i\in\iip$, we have
\begin{equation}\label{GMPactioncommuteeq1}
x\gmph (\qc{G}_{\partial_{q_i}}y) = \qc{H}_{\partial_{q_i}}(x\gmph y).
\end{equation}
\end{proposition}
\begin{proof}
This is the algebraic version of \cite[Theorem 4.5]{GMP}. Our proof is nothing but a translation of the proof given therein.

By localization, it suffices to assume $H=T$ and $x=[t^{\mu}]$ for some $\mu\in\Q$. In this case, the LHS of \eqref{GMPactioncommuteeq1} is equal to $\shift^A_{\mu}(\qc{T}_{\partial_{q_i}}y)$. By Lemma \ref{Ashiftlemmacomm}, the latter expression is equal to $\qc{T}_{\partial_{q_i}}\shift^A_{\mu}(y)$ which is equal to the RHS of \eqref{GMPactioncommuteeq1}.
\end{proof}

%%%%%%%%%%%%%%%%%%%%%%%%%%%%%%%%%%%
%%%%%%%%%%%%%%%%%%%%%%%%%%%%%%%%%%%
\subsection{Some computational results}\label{GMPcompute}

\begin{definition} \label{GMPcomputedef} Let $H$ be either $T$ or $G$. Define an $H_{\Hhat}^{\bl}(\pt)$-linear map
\[ \pgmph: H^{\Hhat}_{-\bl}(\ag) \ra QH_{\Hhat}^{\bl}(G/P)[\ecc] \]
by 
\[ \pgmph(x) := x\gmph 1,\quad x\in H_{-\bl}^{\Hhat}(\ag).\]
It is graded by Lemma \ref{GMPactiongraded}.
\end{definition}

Observe that $m_f([\G]\otimes 1)=1$ for any $\That$-equivariant morphism $f:\G\ra\ag$. Hence 
\[\mathcal{I}(f,[\G],1,v,\eta)= \int_{[\MM^{0,\infty}(f,\eta)]^{vir}} \ev_{f,\eta,\infty}^*\s^v=\int_{[\MM^{\infty}(f,\eta)]^{vir}} \ev_{f,\eta,\infty}^*\s^v\]
where $\MM^{\infty}(f,\eta):=\bigcup_{\b}\MM_{0,1}(\ffib{f},\b)\times_{(\ev_1,\iota_{f,\infty})}D_{f,\infty}$. It follows that $\pgmpt$ has the same defining expression (i.e. the expression \eqref{GMPactiondefeq} after putting $y=1$) as that of \textit{Savelyev-Seidel's homomorphism} $\Phi_{SS}$ defined in \cite[Definition 3.10]{me1}, except that the latter is in the $T$-equivariant settings. This observation allows us to obtain some computational results for $\pgmpt$ by extending those for $\Phi_{SS}$ which have already been established in \cite{me1} and \cite{me2}.

Denote by $(W^P)_{af}$ the set of $\wl\in W_{af}$ satisfying
\[\left\{
\begin{array}{rcl}
\a\in R_P^+\cap (-w^{-1}R^+) &\Longrightarrow& \a(\lambda)=-1\\ [.5em]
\a\in R_P^+\cap w^{-1}R^+&\Longrightarrow& \a(\lambda)=0
\end{array}
\right. .\]

\begin{proposition}\label{GMPPLS} For any $\wl\in W_{af}^-$, we have
\begin{equation}\label{GMPPLSeq1}\nonumber
\pgmpt(\ascl_{\wl})=\left\{
\begin{array}{cc}
q^{[\lambda]}\s_{\widetilde{w}}& \wl\in (W^P)_{af}\\ [.5em]
0& \text{otherwise}
\end{array}
\right.
\end{equation}
where $\widetilde{w}\in W^P$ is the unique element satisfying $\widetilde{w}W_P=wW_P$. 
\end{proposition}
\begin{proof}
This is the $\That$-equivariant version of \cite[Theorem 4.9]{me1}. Recall the proof given therein consists of the following steps:
\begin{enumerate}
\item Show that each $f_{\ag,\wl}:\G_{\wl}\ra \ag$ (see Definition \ref{GMPactiondef}) can be chosen to be equivariant with respect to the $B_-$-action and the $U_{\a,k}$-action for any $\a\in R$ and $k>0$ \cite[Lemma 2.4]{me1}, where $U_{\a,k}:=\exp(z^k\gg_{\a})\subset\mathcal{L}G$ is the affine root group.

\item Show that the stack $\MM^{\infty}(f_{\ag,\wl},\eta)$ (denoted by $\MM(f_{\ag,\wl},\eta)$ in \textit{op. cit.}) is smooth and of expected dimension \cite[Proposition 4.5]{me1}, and the evaluation morphism 
\[\ev_{f_{\ag,\wl},\eta,\infty}:\MM^{\infty}(f_{\ag,\wl},\eta)\ra G/P\]
(denoted by $\ev$ in \textit{op. cit.}) is transverse to any $B$-equivariant morphisms to $G/P$ \cite[before Lemma 4.7]{me1}.

\item Compute the $T$-equivariant analogue of the integral $\mathcal{I}(f_{\ag,\wl},[\G_{\wl}],1,v,\eta)$ by counting the elements of $\MM^{\infty}(f_{\ag,\wl},\eta)\times_{(\ev_{f_{\ag,\wl},\eta,\infty},f_{G/P,v})} \G_v$ directly \cite[Proposition 4.8]{me1}, where $f_{G/P,v}:\G_v\ra G/P$ is a $B$-equivariant morphism which factors through a resolution of $\ol{B\dot{v}P/P}$.
\end{enumerate}

It is not difficult to see that the arguments in \textit{op. cit.} for completing the above steps work in the current settings. The only extra task is to show that the chosen $f_{\ag,\wl}$ is also $\gm$-equivariant. But this is automatic because it satisfies the property that every algebraic group action on $\ol{\mathcal{B}\cdot t^{w(\lambda)}}$ lifts to an algebraic group action on $\G_{\wl}$.
\end{proof}

\begin{corollary}\label{GMPPLScor}
The map
\[ (\pgmpt)_{\CC[\eccc]} : H^{\That}_{-\bl}(\ag)[\eccc] \ra QH_{\That}^{\bl}(G/P)[\ecc]\]
defined by extension of scalars is surjective. \hfill$\square$
\end{corollary}

\begin{proposition}\label{GMPcompute0}
For any $\lambda\in\cop$, $\pgmpt$ vanishes on $H_{>2\rho(\lambda-w_Pw_0(\lambda))}^{\That}(\agll)$.
\end{proposition}
\begin{proof}
The proof is identical to the proof of \cite[Proposition 3.9]{me2}.
\end{proof}

\begin{proposition}\label{GMPcompute1}
For any $\lambda\in\cop$ and $x\in H_{2\rho(\lambda-w_Pw_0(\lambda))}^{\That}(\agll)$, we have $\pgmpt(x)=c_x q^{[w_0(\lambda)]}$ for some $c_x\in \CC$.
\end{proposition}

\begin{proof}
This follows from the proof of \cite[Proposition 3.11]{me2}.
\end{proof}

%%%%%%%%%%%%%%%%%%%%%%%%%%%%%%%%%%%%
%%%%%%%%%%%%%%%%%%%%%%%%%%%%%%%%%%%%
%%%%%%%%%%%%%%%%%%%%%%%%%%%%%%%%%%%%
%%%%%%%%%%%%%%%%%%%%%%%%%%%%%%%%%%%%
%%%%%%%%%%%%%%%%%%%%%%%%%%%%%%%%%%%
%%%%%%%%%%%%%%%%%%%%%%%%%%%%%%%%%%%
\section{Bezrukavnikov-Finkelberg's theorems}\label{BF}
In this section, we give a brief exposition of two results of Bezrukavnikov and Finkelberg \cite{BF}. Recall the geometric Satake equivalence \cite{Ginzburg, MV} is an equivalence 
\[\gs:\rep(\Gd)\xrightarrow{\sim} \per\]
of tensor categories between the category $\rep(\Gd)$ of finite dimensional representations of $\Gd$ and the category $\per$ of $G(\OO)$-equivariant perverse sheaves on $\ag$ with compact support. The first of their results, discussed in Section \ref{BFenhance}, is an enhancement of this equivalence: the category of free $\hp$-Harish-Chandra bimodules for $\Gd$ is equivalent as monoidal categories to the category of semi-simple complexes in $D_{G(\OO)\rtimes\gm}(\ag)$. The second result, discussed in Section \ref{BFtoda}, is an application of the first: the quantum Toda lattice of $\Gd$ is isomorphic as rings to the $G\times\gm$-equivariant homology of $\ag$. 

We also prove a computational result on the isomorphism from the second result and discuss its semi-classical limit, which will be done in Section \ref{BFcompute} and Section \ref{BFlimit} respectively.

%%%%%%%%%%%%%%%%%%%%%%%%%%%%%%%%%%%
%%%%%%%%%%%%%%%%%%%%%%%%%%%%%%%%%%%
\subsection{Enhancement of geometric Satake equivalence}\label{BFenhance}
Let $\rho$ be the half-sum of the positive roots. We grade $\ugd$ by requiring $\hp$ to have degree 2; every element of $\ttd\subset\ggd\subset\ugd$ to have degree 2; and every element of $\ggd_{\ad}\subset\ggd\subset\ugd$ ($\a\in R$) to have degree $2\rho(\ad)+2$. 
\begin{definition}\label{BFenhanceHCdef}
An \textit{$\hp$-Harish-Chandra bimodule} $M$ is a graded left $\ugd\otimes_{\CC[\hp]}\ugd^{op}$-module equipped with a linear algebraic $\Gd$-action satisfying
\begin{enumerate}
\item $\s(\ttd)\cdot M_i\subseteq M_i$ and $\s(\ggd_{\ad})\cdot M_i\subseteq M_{i+2\rho(\ad)}$ for any $\a\in R$ and $i\in\ZZ$, where $\s:\ggd\ra\eendo_{\CC}(M) $ is the linearization of the $\Gd$-action and $M_i$ is the $i$-th graded piece of $M$;

\item $g\cdot ((x\otimes y)\hcdot m)=((g\cdot x)\otimes (g\cdot y))\hcdot (g\cdot m)$ for any $m\in M$, $x,y\in\ugd$ and $g\in\Gd$; and
\item $(x\otimes 1-1\otimes x)\hcdot m= \hp\s(x) m$ for any $m\in M$ and $x\in\ggd$.
\end{enumerate}
\end{definition}

\begin{example}\label{BFenhanceHCeg1}
Let $V=\bigoplus_{i\in\ZZ}V_i$ be a finite dimensional graded $\Gd$-module satisfying $\ttd\cdot V_i\subseteq V_i$ and $\ggd_{\ad}\cdot V_i\subseteq V_{i+2\rho(\ad)}$ for any $\a\in R$ and $i\in\ZZ$. (Up to a degree shift every irreducible representation has a unique such grading.) Define $\fr(V):=\ugd\ot V$ and a left $\ugd\otimes_{\CC[\hp]}\ugd^{op}$-module structure on $\fr(V)$ by the formula
\begin{equation}\label{BFenhanceHCeg1eq1}\nonumber
 (x\otimes y)\hcdot (z\otimes v):= xzy\otimes v - \hp xz\otimes (y\cdot v),\quad v\in V,~ x,y\in \ggd,~ z\in \ugd.
\end{equation}
Then $\fr(V)$ is an $\hp$-Harish-Chandra bimodule. We call bimodules of this form \textit{free $\hp$-Harish-Chandra bimodules}.
\end{example}

\begin{remark}\label{BFenhanceHCeg1rmk}
Similarly, we define an $\hp$-Harish-Chandra bimodule structure on $V\ot\ugd$ by the formula
\[  (x\otimes y)\hcdot (v\otimes z):= v\otimes xzy + \hp (x\cdot v)\otimes zy,\quad v\in V,~ x,y\in \ggd,~ z\in \ugd.\]
However, we obtain nothing new because it is just isomorphic to $\fr(V)$. The isomorphism is constructed by ``moving'' every factor $z\in\ggd\subset\ugd$ from right to left using the exchange relation $v\ot z\mapsto z\ot v-\hp \ot (z\cdot v)$.
\end{remark}

\begin{example}\label{BFenhanceHCeg2}
Let $\dm{\Gd}$ be the ring of $\hslash$-differential operators on $\Gd$. Define a grading on $\dm{\Gd}$ as follows. Consider a $\gm$-action on $\Gd$ defined by $c\cdot  g:=(2\rho(c))g(2\rho(c))^{-1}$. Regular functions which have weight $m_1$ with respect to this $\gm$-action are required to have degree $m_1$ and vector fields which have weight $m_2$ are required to have degree $m_2+2$. For any $x\in\ggd$, denote by $x^L$ (resp. $x^R$) the left-invariant (resp. right-invariant) vector field on $\Gd$ generated by $x$. Define two left $\ugd\otimes_{\CC[\hp]}\ugd^{op}$-module structures $\hcdot^L$ and $\hcdot^R$ on $\dm{\Gd}$:
\begin{equation}\label{BFenhanceHCeg2eq1}\nonumber 
 (x\otimes 1)\hcdot^L z:= x^Lz~\text{ and }~(1\otimes x)\hcdot^L z:=zx^L,\quad x\in\ggd,~ z\in \dm{\Gd}
\end{equation}
and 
\begin{equation}\label{BFenhanceHCeg2eq2} \nonumber 
(x\otimes 1)\hcdot^R z:=- x^Rz~\text{ and }~(1\otimes x)\hcdot^R z:=-zx^R,\quad x\in\ggd, ~z\in \dm{\Gd}.
\end{equation}
Then $\hcdot^L$ and $\hcdot^R$ define two $\hp$-Harish-Chandra bimodule structures on $\dm{\Gd}$, where the right and left translations are used for the corresponding $\Gd$-action respectively. These two structures commute with each other in the obvious sense.
\end{example}
\begin{definition}\label{BFenhancecatdef}$~$
\begin{enumerate}
\item Define $\hct$ to be the category of $\hp$-Harish-Chandra bimodules where the morphism spaces consist of graded $\Gd$-equivariant $\ugd\otimes_{\CC[\hp]}\ugd^{op}$-linear maps of degree zero.

\item Define $\hcf$ to be the full subcategory of $\hct$ consisting of free $\hp$-Harish-Chandra bimodules defined in Example \ref{BFenhanceHCeg1}.

\item Define $\hc$ to be the full subcategory of $\hct$ consisting of objects which are quotients of objects of $\hcf$.
\end{enumerate}
\end{definition}

\begin{definition}
Let $M_1,M_2\in\hct$. Define
\[ M_1\hcprod M_2 := M_1\otimes_{\ugd} M_2\]
where the right (resp. left) $\ugd$-module structure on $M_1$ (resp. $M_2$) for the tensor product is the one induced by the given $\hp$-Harish-Chandra bimodule structure via the map $x\mapsto 1\otimes x$ (resp. $x\mapsto x\otimes 1$). Clearly $M_1\hcprod M_2$ has a natural $\hp$-Harish-Chandra bimodule structure and lies in $\hcf$ or $\hc$ if both $M_1$ and $M_2$ do. Moreover, $\hcprod$ makes $\hct$, $\hcf$ and $\hc$ into monoidal categories.
\end{definition}

Take a character $\cchh_0:\nd\ra\CC$ which is non-zero on each direct summand  $\ggd_{-\ad_i}$. Let $M\in\hct$. Denote by $\mathcal{I}^M_{-,\cchh_0}$, or simply $\mathcal{I}_{-,\cchh_0}$ if no confusion arises, the sub-$\CC[\hp]$-module of $M$ generated by elements of the form $(1\otimes (x-\cchh_0(x)))\hcdot m$ where $x\in\nndm$ and $m\in M$. It is graded and preserved by $\ugd\ot_{\CC[\hp]}\zgd$ and $U^{\vee}_-$. Since $U^{\vee}_-$ fixes every element of $\zgd$, $(M/\mathcal{I}_{-,\cchh_0})^{U^{\vee}_-}$ is naturally a $\zgd\ot_{\CC[\hp]}\zgd$-module. Denote by $\mathcal{I}^M_{+}$, or simply $\mathcal{I}_{+}$, the sub-$\CC[\hp]$-module of $M$ generated by elements of the form $(x\otimes 1)\hcdot m$ where $x\in \nnd$ and $m\in M$. It is graded and preserved by $\utd\otimes_{\CC[\hp]}\ugd$. For any $\lambda\in\Q$, define an automorphism $\twist_{\lambda}$ of the $\CC[\hp]$-algebra $\utd$ by 
\[ \twist_{\lambda}(x) := x+\hp\lambda(x),\quad x\in\ttd.\]
For a $\utd$-module $N$, denote by $\twist_{\lambda}(N)$ the $\utd$-module $N$ with module structure twisted by $\twist_{\lambda}$.
 
\begin{definition}\label{BFenhanceKFdef}$~$
\begin{enumerate}
\item (Kostant functor) Define an additive functor
\[ \kf:\hct\ra \left\{\text{graded }\zgd\otimes_{\CC[\hp]}\zgd\text{-modules}\right\}\]
by
\[ \kf(M):= (M/\mathcal{I}_{-,\cchh_0})^{\Umd}.\]

\item (Desymmetrized Kostant functor) Define an additive functor
\[ \dkf:\hct\ra \left\{\text{graded }\utd\otimes_{\CC[\hp]}\zgd\text{-modules}\right\}\]
by
\[ \dkf(M):= \twist_{-\rhod}(M/(\mathcal{I}_{-,\cchh_0}+\mathcal{I}_+))\]
where $\rhod$ is the half-sum of the positive coroots.
\end{enumerate}
\end{definition}

\begin{remark}\label{BFenhanceKFrmk}
Up to natural equivalence, the (resp. desymmetrized) Kostant functor is independent of the choice of $\cchh_0$. Indeed, for different choices of $\cchh_0$, the multiplication by a suitable element of $\Td$ induces a natural equivalence between the corresponding (resp. desymmetrized) Kostant functors. Unless otherwise specified, we will take $\cchh_0$ to be the character $\cchh$ defined in \eqref{cchhdef}.
\end{remark}

The functors $\kf$ and $\dkf$ are related as follows. Let $M\in\hct$. The composition
\[ \kf(M)=(M/\mathcal{I}_{-,\cchh})^{\Umd} \hookrightarrow M/\mathcal{I}_{-,\cchh}\ra M/(\mathcal{I}_{-,\cchh}+\mathcal{I}_+)=\dkf(M) \]
is a homomorphism of modules with respect to the ring map
\[  \hci\otimes_{\CC[\hp]}\id_{\zgd}: \zgd\otimes_{\CC[\hp]}\zgd\ra \utd\otimes_{\CC[\hp]}\zgd \]
where $\hci$ is the $\hp$-Harish-Chandra homomorphism from Section \ref{notationGd}. By extension of scalars, we obtain a $\utd\otimes_{\CC[\hp]}\zgd$-linear map
\begin{equation}\label{BFenhancettmdef}
\tt_M: \utd\otimes_{\zgd}\kf(M)\ra \dkf(M).
\end{equation}

\begin{lemma}\label{BFenhanceKFisom}
(\cite[Lemma 5]{BF}) If $M\in\hcf$, then $\tt_M$ is an isomorphism. \hfill$\square$
\end{lemma}

\begin{remark} See Appendix \ref{A} for more discussions on Lemma \ref{BFenhanceKFisom}.
\end{remark}

Define $\mathcal{R}$ to be the sub-$\zgd\ot_{\CC[\hp]}\zgd$-algebra of $\zgd\ot_{\CC[\hp]}\zgd[\hp^{-1}]$ generated by $\hp^{-1}(x\ot 1-1\ot x)$, $x\in \zgd$. Let $M\in\hct$. Suppose $\kf(M)$ is $\hp$-torsion-free. (It is the case for $M\in \hcf$.) Then $\kf(M)$ is naturally a graded $\mathcal{R}$-module.

Let $\ic$ be the full subcategory of $D_{G(\OO)\rtimes\gm}(\ag)$ consisting of objects which are direct sums of degree shifts of simple objects of $\mathcal{P}_{G(\OO)\rtimes \gm}(\ag)$. Put $\Ghat:=G\times\gm$. Consider the $\Ghat$-equivariant hypercohomology functor $H_{\Ghat}^{\bl}$ on $\ic$. A priori, it lands in the category of graded $H_{\Ghat}^{\bl}(\ag)$-modules. Observe that $H_{\Ghat}^{\bl}(\ag)$ is naturally a graded algebra over $H_{\Ghat}^{\bl}(\pt)\ot H_{G}^{\bl}(\pt)\simeq H_{\Ghat}^{\bl}(\pt)\ot_{\CC[\hp]} H_{\Ghat}^{\bl}(\pt)$. (The first factor corresponds to the canonical $\Ghat$-action on $\ag$, and the second corresponds to the canonical $G$-torsor on $\ag$.) Identify $\zgd$ with $H_{\Ghat}^{\bl}(\pt)$ via the $\hp$-Harish-Chandra homomorphism $\hci$ (see Section \ref{notationGd}) and the canonical isomorphism $\utd\simeq H_{\That}^{\bl}(\pt)$ so that $H_{\Ghat}^{\bl}(\ag)$ is now a $\zgd\ot_{\CC[\hp]}\zgd$-algebra. One can show that $\hp^{-1}(x\ot 1-1\ot x)\cdot 1$ exists in $H_{\Ghat}^{\bl}(\ag)$ for any $x\in \zgd$. This gives rise to a graded homomorphism
\[\phi:\mathcal{R}\ra H_{\Ghat}^{\bl}(\ag)\]
of $\zgd\ot_{\CC[\hp]}\zgd$-algebras.

\begin{theorem}\label{BFenhancesame} 
(\cite[Theorem 1]{BF}) $\phi$ is bijective. \hfill$\square$
\end{theorem} 

Now, we state the first result mentioned at the beginning of Section \ref{BF}.
\begin{theorem}\label{BFenhancethm}
(\cite[Theorem 2]{BF}) There is a unique functor 
\[ \GS:\hcf\ra\ic\]
such that 
\begin{enumerate}
\item $\GS$ commutes with the degree shift functors;
\item for any $\lambda\in\cop$, $\GS(\fr(S(\lambda)))$ is the $G(\OO)\rtimes \gm$-equivariant intersection complex of $\agll$ ($S(\lambda)$ is graded such that every element of $S(\lambda)_{\mu}$ has degree $2\rho(\mu)$); and
\item the diagram
\begin{equation}\nonumber
\begin{tikzpicture}
\tikzmath{\x1 = 7; \x2 = 2;}
\node (A) at (0,0) {$\hcf$} ;
\node (B) at (\x1,0) {$\ic$} ;
\node (C) at (0,-\x2) {$\{\text{gr. }\mathcal{R}\text{-mod}\}$} ;
\node (D) at (\x1,-\x2) {$\{\text{gr. }H_{\Ghat}^{\bl}(\ag)\text{-mod}\}$} ;

\path[->] (A) edge node[above]{$\GS$} (B);
\path[->] (A) edge node[left]{$\kf|_{\hcf}$} (C);
\path[->] (B) edge node[right]{$H_{\Ghat}^{\bl}|_{\ic}$} (D);
\path[->]
(C) edge node[above]{$\sim$} (D);
\end{tikzpicture}
\end{equation}
is commutative, where the bottom isomorphism is induced by $\phi$ from Theorem \ref{BFenhancesame}.
\end{enumerate}
Moreover, $\GS$ is an equivalence of monoidal categories and the natural equivalence connecting the composite functors in the above diagram is compatible with the monoidal structures on these functors. \hfill$\square$
\end{theorem} 

%%%%%%%%%%%%%%%%%%%%%%%%%%%%%%%%%%%%
%%%%%%%%%%%%%%%%%%%%%%%%%%%%%%%%%%%%
\subsection{Quantum Toda lattice and equivariant homology of affine Grassmannian}\label{BFtoda}
Recall the functors $\kf$ and $\dkf$ defined in Definition \ref{BFenhanceKFdef} and the two $\hp$-Harish-Chandra bimodule structures $\hcdot^L$ and $\hcdot^R$ on $\dm{\Gd}$ defined in Example \ref{BFenhanceHCeg2}.

\begin{definition}\label{KKdef}
Define
\[ \KK:=\kf(\dm{\Gd})\]
where we apply $\kf$ to $\hcdot^R$.
\end{definition}

Since these two $\hp$-Harish-Chandra bimodule structures commute with each other, it follows that $\hcdot^L$ induces an $\hp$-Harish-Chandra bimodule structure on $\KK$. 

\begin{definition}\label{BFQTLdef}$~$
\begin{enumerate}
\item The graded $\zgd\ot_{\CC[\hp]}\zgd$-module $\kf(\KK)$ is called the \textit{quantum Toda lattice}.

\item The graded $\utd\ot_{\CC[\hp]}\zgd$-module $\dkf(\KK)$ is called the \textit{desymmetrized quantum Toda lattice}.
\end{enumerate}
\end{definition}

The ring structure on $\dm{\Gd}$ induces an algebra structure on $\KK$, i.e. a morphism $\KK\hcprod\KK
\ra \KK$ in $\hct$ satisfying the unitality and associativity axioms. Since $\kf$ is a monoidal functor, it follows that $\kf(\KK)$ is naturally a graded $\CC[\hp]$-algebra and $\dkf(\KK)$ is naturally a graded right $\kf(\KK)$-module.

On the other hand, define 
\[ \daggg := \underset{\lambda\in\cop}{\varinjlim} \dagl \]
to be the filtered colimit of the dualizing complexes $\dagl$ of $\agll$ ($\lambda\in\cop$). It has an algebra structure given by convolution so that $H_{\Ghat}^{\bl}(\daggg)$ is a graded $\CC[\hp]$-algebra. Notice that 
\[ H_{\Ghat}^{\bl}(\daggg)\simeq H_{-\bl}^{\Ghat}(\ag)\]
as graded $\CC[\hp]$-algebras where the algebra structure on the RHS is the one defined in Section \ref{notationG}. 

Now, we state the second result mentioned at the beginning of Section \ref{BF}.
\begin{theorem} (\cite[Section 6.4]{BF})\label{BFqttheorem}
\begin{enumerate}
\item There exists a cohomological functor
\[ F: D_{G(\OO)\rtimes\gm}(\ag)\ra \hc\]
such that 
\begin{enumerate}[(i)]
\item $F$ extends $\GS^{-1}$ (see Theorem \ref{BFenhancethm} for the definition of $\GS$);

\item for any $X,Y\in D_{G(\OO)\rtimes\gm}(\ag)$ with $X\in\ic$, the canonical map
\[ F_{X,Y}:\Hom_{D_{G(\OO)\rtimes\gm}(\ag)}(X,Y)\ra \Hom_{\hc}(F(X),F(Y)) \]
is bijective;

\item $F$ is quasi-monoidal in the sense that there exist a morphism $\text{unit}\ra F(\text{unit})$ and for each pair $(X,Y)$ of objects of the source category a morphism $F(X)\hcprod F(Y)\ra F(X* Y)$ such that these morphisms are functorial and satisfy the unitality and associativity axioms; 

\item the quasi-monoidal structure on $F$ from (iii) extends the monoidal structure on $\GS^{-1}$ from Theorem \ref{BFenhancethm}; and

\item $\kf\circ F\simeq H_{\Ghat}^{\bl}$ as quasi-monoidal functors. 
\end{enumerate} 

\item The functor $F$ from (1) sends $\daggg$ to $\KK$ in the sense that $\KK$ represents the functor $\hcf\ni M\mapsto \underset{\lambda\in\cop}{\varinjlim} \Hom_{\hc}(M,F(\dagl))$. Moreover, it induces a graded isomorphism 
\[ \bfisok_{\KK}: \kf(\KK) \simeq \underset{\lambda\in\cop}{\varinjlim} \kf\circ F(\dagl)\simeq \underset{\lambda\in\cop}{\varinjlim} H_{\Ghat}^{\bl}(\dagl) \simeq H_{-\bl}^{\Ghat}(\ag)\]
of $\CC[\hp]$-algebras and of $\mathcal{R}$-modules.
\end{enumerate}
\hfill $\square$
\end{theorem}

%%%%%%%%%%%%%%%%%%%%%%%%%%%%%%%%%%%%
%%%%%%%%%%%%%%%%%%%%%%%%%%%%%%%%%%%%
\subsection{A computational result}\label{BFcompute}
\begin{definition}\label{BFpbfdef}
Define
\[ \pbfk :=\bfisok_{\KK}: \kf(\KK) \xrightarrow{\sim} H_{-\bl}^{\Ghat}(\ag)\]
where $\bfisok_{\KK}$ comes from Theorem \ref{BFqttheorem}(2), and 
\[ \pbfd : \dkf(\KK) \xrightarrow{\sim} H_{-\bl}^{\That}(\ag)\]
to be the map obtained from $\pbfk$ by extension of scalars with respect to the canonical map $\zgd\simeq H_{\Ghat}^{\bl}(\pt)\ra \utd\simeq H_{\That}^{\bl}(\pt)$. (We have $\dkf(\KK)\simeq\utd\ot_{\zgd}\kf(\KK)$ by Lemma \ref{Alemma2}.)
\end{definition}

In this subsection, we prove a computational result about $\pbfd$. Recall we have been using $\cchh$ to define $\kf$. See Remark \ref{BFenhanceKFrmk}. But we would like to make an exception: instead of $\cchh$, we use $-\cchh$ for $\kf$ applied to $\dm{\Gd}$ when we define $\KK$. Although changing the character does not change the isomorphism class of algebra objects of $\hct$ represented by $\KK$, it does change how $\Phi_{BF}^{\delta/\kappa}$ will look. Our choice will make Proposition \ref{BFenhancedivisor} below hold. Notice that we will still use $\cchh$ for $\kf$ (resp. $\dkf$) applied to $\KK$ when we define the (resp. desymmetrized) quantum Toda lattices.

Let $\lambda\in\cop$ and $\mu\in\Q$. Recall an MV cycle of type $\lambda$ and weight $\mu$ is any irreducible component of $\ol{S^-_{\mu}}\cap\agll$ of dimension $\rho(\lambda-\mu)$ where 
\[S^-_{\mu}:= \left\{ y\in\ag\left|~\lim_{s\to\infty}2\rhod(s)\cdot y= t^{\mu}\right.\right\} \]
and $\rho$ (resp. $\rhod$) is the half-sum of the positive roots (resp. coroots). By \cite[Proposition 13.1]{MV}, $\gs(S(\lambda))$ is the costandard sheaf ${}^p\mathcal{H}^0((\iota_{\lambda})_*\CC[2\rho(\lambda)])$ where $\iota_{\lambda}:\agl\hookrightarrow \ag$ is the inclusion (recall $S(\lambda):=H^0(\GB;\Gd\times^{\Bmd}\CC_{\lambda})$). There are canonical isomorphisms
\begin{equation}\label{MVbasis}
S(\lambda)_{\mu}\simeq H_{\ol{S^-_{\mu}}}^{2\rho(\mu)}(\gs(S(\lambda))) \simeq H_{\ol{S^-_{\mu}}}^{2\rho(\mu)}((\iota_{\lambda})_*\CC[2\rho(\lambda)])\simeq \CC\langle\text{MV cycles of type }\lambda\text{ and weight }\mu\rangle
\end{equation}
where the first isomorphism is part of the geometric Satake equivalence, the second is induced by the canonical morphism $\gs(S(\lambda))\ra (\iota_{\lambda})_*\CC[2\rho(\lambda)]$, and the third follows from base-change and excision (see \cite[Proposition 3.10]{MV} and also \cite[Proposition 11.1]{Satakeluminy}). 

Let $v_{\lambda}^*\in S(\lambda)^*$ be the unique vector satisfying $v_{\lambda}^*|_{S(\lambda)_{\ne \lambda}}\equiv 0$ and $\langle v_{\lambda}^*,v_e\rangle =1$ where $v_e\in S(\lambda)_{\lambda}\simeq\CC$ is the vector corresponding to the unique MV cycle $\{t^{\lambda}\}$ of type $\lambda$ and weight $\lambda$ via the composite isomorphism \eqref{MVbasis}. For any $v\in S(\lambda)$, denote by $f_{\lambda,v}\in\OO(\Gd)$ the regular function $g\mapsto \langle v_{\lambda}^*,g\cdot v\rangle$. Observe that $f_{\lambda,v}(u\cdot -)=f_{\lambda,v}(-)$ for any $u\in \Umd$. It follows that $f_{\lambda,v}$ represents an element of $\KK$ and hence an element of $\dkf(\KK)$. By abuse of notation, both of these elements will be denoted by $[f_{\lambda,v}]$.

\begin{proposition}\label{BFenhancedivisor}
Let $\lambda\in\cop$, $\mu\in\Q$ and $v\in S(\lambda)_{\mu}$. Suppose $v$ corresponds to an MV cycle $Z$ (of type $\lambda$ and weight $\mu$) via the composite isomorphism \eqref{MVbasis}. Then 
\[ \pbfd([f_{\lambda,v}]) = [Z]\in H_{2\rho(\lambda-\mu)}^{\That}(\ag).\]
\end{proposition}

Before proving Proposition \ref{BFenhancedivisor}, let us go back to Theorem \ref{BFenhancethm} in Section \ref{BFenhance}. A key step of the proof given in \cite{BF} is to construct for each $V\in\rep\Gd$ an isomorphism
\[ \bfisok_{\fr(V)}: \kf(\fr(V))\xrightarrow{\sim} H_{\Ghat}^{\bl}(\GS(V))\]
of $\zgd\ot_{\CC[\hp]}\zgd$-modules. This amounts to constructing a $W$-equivariant isomorphism
\[ \bfisod_{\fr(V)}: \dkf(\fr(V))\xrightarrow{\sim} H_{\That}^{\bl}(\GS(V)) \]
of $\utd\ot_{\CC[\hp]}\zgd$-modules. Consider the filtrations 
\[\{F^{\mu}\dkf(\fr(V))\}_{\mu\in\co}\quad\text{ and }\quad\{F^{\mu}H_{\That}^{\bl}(\GS(V))\}_{\mu\in\co}\]
on $\dkf(\fr(V))$ and $H_{\That}^{\bl}(\GS(V))$ respectively, defined by
\[  F^{\mu}\dkf(\fr(V)):= \fr(V)_{\geqslant \mu}/\left(\fr(V)_{\geqslant \mu}\cap (\mathcal{I}_{-,\cchh} +\mathcal{I}_+ )\right)\]
where $\fr(V)_{\geqslant \mu}:= V_{\geqslant \mu}\ot\ugd$ (recall $\fr(V)\simeq V\ot\ugd$ by Remark \ref{BFenhanceHCeg1rmk}), and
\[ F^{\mu}H_{\That}^{\bl}(\GS(V)) : = \im\left( H_{\ol{S_{\mu}^-}, \That}^{\bl}(\GS(V)) \ra H_{\That}^{\bl}(\GS(V))\right) \] 
where the map in the RHS is induced by the inclusion $\ol{S_{\mu}^-}\hookrightarrow \ag$. By \cite[Lemma 1 \& Lemma 6]{BF}, there are canonical isomorphisms of $\utd\ot_{\CC[\hp]}\zgd$-modules
\[ grF^{\mu}\dkf(\fr(V))\simeq \mathcal{C}_{\mu}\ot V_{\mu}~\text{ and }~grF^{\mu}H_{\That}^{\bl}(\GS(V))\simeq \mathcal{C}_{\mu}\ot H_{S^-_{\mu}}^{2\rho(\mu)}(\gs(V)) \]
where $\mathcal{C}_{\mu}:=\utd$ on which $\utd$ and $\zgd$ act via the identity and the ring homomorphism $z\mapsto \twist_{-\mu}(\Theta_{HC}(z))$ respectively. By the geometric Satake equivalence, we thus have an isomorphism
\begin{equation}\label{BFgradedpiece}
grF^{\mu}\dkf(\fr(V))\simeq grF^{\mu}H_{\That}^{\bl}(\GS(V))
\end{equation}
of $\utd\ot_{\CC[\hp]}\zgd$-modules.

\begin{theorem} \label{BFenhancepreservefil} (\cite[Theorem 6]{BF}) There exists a unique isomorphism 
\[ \bfisod_{\fr(V)}:\dkf(\fr(V))\xrightarrow{\sim} H_{\That}^{\bl}(\GS(V))\]
of $\utd\ot_{\CC[\hp]}\zgd$-modules such that 
\begin{enumerate}
\item it is compatible with the above filtrations, i.e. $\bfisod_{\fr(V)}\left( F^{\mu}\dkf(\fr(V))\right)\subseteq F^{\mu}H_{\That}^{\bl}(\GS(V))$ for any $\mu\in\co$; and 

\item it induces isomorphism \eqref{BFgradedpiece} on the associated graded pieces.
\end{enumerate}  
Moreover, $\bfisod_{\fr(V)}$ is $W$-equivariant and hence induces the desired isomorphism $\bfisok_{\fr(V)}$.
\hfill $\square$
\end{theorem}

\begin{myproof}{Proposition}{\ref{BFenhancedivisor}}
Observe that the assignment $v\mapsto [f_{\lambda,v}]$ defines a morphism 
\[f_{\lambda}:\fr(S(\lambda))[2\rho(\lambda)]\ra\KK\]
in $\hct$. (Here, $S(\lambda)$ is graded such that every element of $S(\lambda)_{\mu}$ has degree $2\rho(\mu)$.) Let $g_{\lambda}:\GS(S(\lambda))[2\rho(\lambda)]\ra \daggg$ be the morphism induced by $f_{\lambda}$ via Theorem \ref{BFqttheorem}. Put $M(\lambda):=\fr(S(\lambda))[2\rho(\lambda)]\in\hcf$ and $\JJ_{\lambda}:= \GS(S(\lambda))\in \ic$. We have the following commutative diagram

\begin{equation}\nonumber
\begin{tikzpicture}
\tikzmath{\x1 = 3; \x2 = 1.5; \x3=2.5; \x4=12; \x5=0; }
\node (A) at (0,0) {$\dkf(M(\lambda))$} ;
\node (B) at (0,-\x3) {$\dkf(\KK)$} ;
\node (C) at (\x1,-\x2) {$H_{\That}^{\bl+2\rho(\lambda)}(\JJ_{\lambda})$} ;
\node (D) at (\x1,-\x2-\x3) {$H_{-\bl}^{\That}(\ag)$} ;
\node (E) at (\x4,-\x5) {$H_{\That}^{\bl+2\rho(\lambda)}(\iota_{\mu,\leqslant\lambda}^{!}\JJ_{\lambda})$} ;
\node (F) at (\x4,-\x3-\x5) {$H_{-\bl}^{\That}(\ol{S^-_{\mu}}\cap\agll)$} ;
\node (G) at (0.5*\x1+0.5*\x4,-0.5*\x2-0.5*\x5) {$F^{\mu}H_{\That}^{\bl+2\rho(\lambda)}(\JJ_{\lambda})$} ;
\node (H) at (-\x1+0.5*\x1+0.5*\x4,\x2-0.5*\x2-0.5*\x5)  {$F^{\mu}\dkf(M(\lambda))$} ;

\path[->, font=\tiny] (A) edge node[left]{$\dkf(f_{\lambda})$} (B);
\path[->, font=\tiny] (A) edgenode[right]{$~\bfisod_{M(\lambda)}$} (C);
\path[->, font=\tiny] (C) edge node[right]{$H_{\That}^{\bl}(g_{\lambda})$} (D);
\path[->, font=\tiny] (B) edge node[below left]{$\pbfd$} (D);

\path[->, font=\tiny] (E) edge node[right]{$H_{\That}^{\bl}(\iota_{\mu,\leqslant\lambda}^{!}(g_{\lambda}))$} (F);
\path[->, font=\tiny] (F) edge node[below]{$H_{-\bl}^{\That}(\iota_{\mu,\leqslant\lambda})$} (D);
\path[->, font=\tiny] (G) edge node[above]{$\iota_2$} (C);
\path[->, font=\tiny] (E) edge node[above]{$j$} (G);

\path[->, font=\tiny] (H) edge node[above]{$\iota_1$} (A);
\path[->, font=\tiny] (H) edge node[right]{$~\bfisod_{M(\lambda)}|_{F^{\mu}\dkf(M(\lambda))}$} (G);

\end{tikzpicture}
\end{equation}
where $\iota_1, \iota_2$ are the inclusion maps, $\iota_{\mu,\leqslant\lambda}:\ol{S^-_{\mu}}\cap\agll\hookrightarrow \ag$ is the inclusion map and $j$ is the canonical isomorphism.

By assumption, we have $[v\otimes 1]\in F^{\mu}\dkf(M(\lambda))$, and by definition, we have $[f_{\lambda,v}]=\dkf(f_{\lambda})\circ\iota_1([v\otimes 1])$. Hence, by the above diagram, we have 
\begin{equation}\label{BFenhancedivisoreq1}
\pbfd([f_{\lambda,v}])=H_{-\bl}^{\That}(\iota_{\mu,\leqslant\lambda})\circ H_{\That}^{\bl}(\iota_{\mu,\leqslant\lambda}^{!}(g_{\lambda}))(x)
\end{equation}
where $x:=j^{-1}\circ\bfisod_{M(\lambda)}|_{F^{\mu}\dkf(M(\lambda))}([v\otimes 1])$. Observe that $H_{\That}^{\bl}(\iota_{\mu,\leqslant\lambda}^{!}(g_{\lambda}))(x)$ lies in $H_{2\rho(\lambda-\mu)}^{\That}(\ol{S^-_{\mu}}\cap\agll)$ and $2\rho(\lambda-\mu)$ is the maximal real dimension of the irreducible components of $\ol{S^-_{\mu}}\cap\agll$, by \cite[Theorem 3.2]{MV}. It follows that $\pbfd([f_{\lambda,v}])$ is equal to a linear combination, with \textit{complex} coefficients, of the $\That$-equivariant fundamental classes of the irreducible components of $\ol{S^-_{\mu}}\cap\agll$ of maximal dimension, i.e. the MV cycles of type $\lambda$ and weight $\mu$. Since these coefficients are constant polynomials in the equivariant parameters, we can determine them by looking at the non-equivariant case.

Denote by $\FF_{eq=0}$ the functor killing all equivariant parameters. By condition (2) in Theorem \ref{BFenhancepreservefil}, the element
\[ x':=\FF_{eq=0}(j^{-1}\circ\bfisod_{M(\lambda)}|_{F^{\mu}\dkf(M(\lambda))})([v\otimes 1])\in H^{2\rho(\mu)}(\iota_{\mu,\leqslant\lambda}^{!}\JJ_{\lambda})\simeq H^{2\rho(\mu)}_{\ol{S^-_{\mu}}}(\JJ_{\lambda})\]
corresponds to $v\in S(\lambda)_{\mu}$ via the first isomorphism in \eqref{MVbasis}. Denote by $\iota_{\lambda}$ and $\iota_{\mu,\lambda}$ the inclusions $\agl\hookrightarrow \ag$ and $\ol{S^-_{\mu}}\cap\agl\hookrightarrow \ag$ respectively. Recall there is a canonical morphism $\JJ_{\lambda} \ra (\iota_{\lambda})_*\CC[2\rho(\lambda)]$ which we shall denote by $h_{\lambda}$. Observe that $\iota_{\lambda}^!(h_{\lambda})$ is an isomorphism and 
\[  \iota_{\lambda}^!(g_{\lambda})\circ \iota_{\lambda}^!(h_{\lambda})^{-1}[2\rho(\lambda)]: \CC[4\rho(\lambda)]\simeq \mathbb{D}_{\agl} \ra \mathbb{D}_{\agl}  \]
is equal to $c_{\lambda}\id_{\mathbb{D}_{\agl} }$ for some $c_{\lambda}\in\CC$. It is not difficult to see that the composition
\begin{align*}
H_{2\rho(\lambda-\mu)}(\ol{S^-_{\mu}}\cap\agll) & \simeq  H_{\ol{S^-_{\mu}}}^{2\rho(\mu-\lambda)}((\iota_{\lambda})_*\CC[4\rho(\lambda)])\xrightarrow{ a^{-1}} 
H_{\ol{S^-_{\mu}}}^{2\rho(\mu-\lambda)}(\JJ_{\lambda}[2\rho(\lambda)]) \\
& \xrightarrow{b  }
H^{2\rho(\mu-\lambda)}(\mathbb{D}_{\ol{S^-_{\mu}}\cap\agll})
\simeq H_{2\rho(\lambda-\mu)}(\ol{S^-_{\mu}}\cap\agll) 
\end{align*}
is equal to $c_{\lambda}$ times the identity, where
\[ a:=H_{\ol{S^-_{\mu}}}^{2\rho(\mu-\lambda)}(h_{\lambda}[2\rho(\lambda)])\quad\text{ and }\quad  b:=H^{2\rho(\mu-\lambda)}(\iota_{\mu,\leqslant\lambda}^{!}(g_{\lambda})).\]
(Notice that $a$ is the second isomorphism in \eqref{MVbasis}.) Therefore, $H^{\bl}(\iota_{\mu,\leqslant\lambda}^{!}(g_{\lambda}))$ sends $x'$ to $c_{\lambda}[Z]$.

It remains to show $c_{\lambda}=1$. Notice that $c_{\lambda}$ depends only on $\lambda$ but not $\mu$. By Lemma \ref{dunnowheretoplacelemmaa} which we will prove in the next subsection, $\FF_{eq=0}(\dkf(\KK))\simeq\OO(\ZZZ_0)$ where $\ZZZ_0:=\{b\in\Bd|~b\cdot e=e\}$. Since $\Gd$ is of adjoint type, we have $\ZZZ_0\subseteq\Ud$. Now take $\mu=\lambda$ and $v=v_e\in S(\lambda)_{\lambda}$, the vector corresponding to the unique MV cycle $\{t^{\lambda}\}$ of type $\lambda$ and weight $\lambda$. Notice that $f_{\lambda,v_e}|_{\Ud}\equiv 1$. Hence,
\[ [f_{\lambda,v_e}] = 1\in \FF_{eq=0}(\dkf(\KK)).\]
Since $\pbfd|_{\kf(\KK)}$ is a ring homomorphism, we have $\FF_{eq=0}(\pbfd)(1)=1$. But we also have $[t^{\lambda}]=1\in H_0(\ag)$. Therefore, we have $c_{\lambda}=1$, as desired.
\end{myproof}

%%%%%%%%%%%%%%%%%%%%%%%%%%%%%%%%%%%
%%%%%%%%%%%%%%%%%%%%%%%%%%%%%%%%%%% 
\subsection{Semi-classical limit}\label{BFlimit}
In this subsection, we discuss the semi-classical limit of the homomorphism $\pbfd$ defined in Definition \ref{BFpbfdef}. Let $\FF_{\hp=0}$ denote the functor $N\ra N/\hp N$.
\begin{lemma}\label{dunnowheretoplacelemmaa}
$\FF_{\hp=0}(\dkf(\KK))$ is isomorphic as $\sym^{\bl}(\ttd)$-modules to the coordinate ring $\OO(\ZZZ)$ of the scheme
\[\ZZZ:=\left\{ (b,\xi)\in \Bd\times (e+\ttd) \left|~ b\cdot\xi=\xi \right. \right\}\]
where $e$ is defined in \eqref{betaande} and the $\sym^{\bl}(\ttd)$-module structure on $\OO(\ZZZ)$ is induced by the morphism $\ZZZ\ra \spec\sym^{\bl}(\ttd)\simeq (\ttd)^*$ defined by $(b,\xi)\mapsto \b(\xi,-)|_{\ttd}$. ($\b$ is also defined in \eqref{betaande}.) 
\end{lemma}
\begin{proof}
We first determine $\FF_{\hp=0}(\KK)$. By Definition \ref{KKdef}, we have $\KK=\kf(\dm{\Gd})$ where $\kf$ is applied to $\hcdot^R$ defined in Example \ref{BFenhanceHCeg2}. By Lemma \ref{Alemma3} which says that $\FF_{\hp=0}$ commutes with $\kf$, $\FF_{\hp=0}(\KK)$ is isomorphic to $\kf(\FF_{\hp=0}(\dm{\Gd}))\simeq \kf(\OO(T^*\Gd))$. By the definition of $\hcdot^R$ and the assumption that the character $-\cchh$ is used for $\kf$ (see the paragraph following Definition \ref{BFpbfdef}), we see that the last module is equal to $(\OO(T^*\Gd)/\mathcal{I}_1)^{\Umd}$ where $\mathcal{I}_1$ is the ideal generated by $-x^R+\cchh(x)$ ($x\in\nnd_-$) and the $\Umd$-action is induced by left translations on $\Gd$. Therefore, $\FF_{\hp=0}(\KK)$ is isomorphic to the coordinate ring of $Y:=\Umd\setminus \widetilde{Y}$, the left $\Umd$-quotient of the closed subscheme $\widetilde{Y}$ of $T^*\Gd$ defined by the equations $x^R-\cchh(x)$ ($x\in \nnd_-$).

We now determine $\FF_{\hp=0}(\dkf(\KK))$. Clearly, $\FF_{\hp=0}$ commutes with $\dkf$, and so, by the previous paragraph, $\FF_{\hp=0}(\dkf(\KK))$ is isomorphic to the coordinate ring of certain closed subscheme $Y'$ of $Y$ which can be expressed as $\Umd\setminus \widetilde{Y}'$ for some closed subscheme $\widetilde{Y}'$ of $\widetilde{Y}$. It is straightforward to see that 
\[ \widetilde{Y}' =\left\{ (g,\xi)\in \Gd\times \ggd \left|~ \xi\in e+\ttd,~g\cdot\xi\in e+\bbd_- \right. \right\}.\]
(Here, we have identified $T^*\Gd$, first with $\Gd\times(\ggd)^*$ via left translations, and then with $\Gd\times\ggd$ via $\b$.) But by Kostant's slice theorem, we have, for any $\xi\in e+\ttd$, 
\[ \left\{ g\in\Gd\left|~g\cdot\xi\in e+\bbd_- \right. \right\} = \Umd\cdot \left\{ b\in\Bd\left|~b\cdot\xi=\xi \right. \right\}, \]
and hence
\[ \widetilde{Y}'=\Umd\cdot \left\{(b,\xi)\in\Bd\times (e+\ttd)\left|~b\cdot\xi=\xi \right. \right\}= \Umd\cdot \ZZZ. \]
Therefore, 
\[ Y'\simeq \Umd\setminus\widetilde{Y}'\simeq  \Umd\setminus\left(\Umd\cdot \ZZZ\right) \simeq \ZZZ.\]
\end{proof}

Recall \textit{Yun-Zhu's isomorphism} \cite{YZ}
\[ \Phi_{YZ}:\OO(\Bed)\xrightarrow{\sim} H_{-\bl}^T(\ag)\]
where 
\begin{equation}\label{Beddef}
\Bed:=\{(b,h)\in\Bd\times\spec H_T^{\bl}(\pt)|~b\cdot e^T(h)=e^T(h)\}
\end{equation}
is the centralizer group scheme of the $H_T^{\bl}(\pt)$-point $e^T$ of $\bbd$ defined by $e^T:=e+f$ with $e:= \sum_{i=1}^r|\al_i|^2 e_i^{\vee}$ (see \eqref{betaande}) and $ f:\ad_i\mapsto |\ad_i|^2\a_i$. (This map is also constructed by Bezrukavnikov, Finkelberg and Mirkovi\'c \cite{BFM}.) Clearly, the isomorphism
\[
\begin{array}{ccc}
\Bd\times\spec H_T^{\bl}(\pt) & \ra & \Bd\times(e+\ttd)\\ [.5em]
(b,h) &\mapsto & (b,e^T(h))
\end{array}
\]
induces an isomorphism $\Bed\xrightarrow{\sim}\ZZZ$ of schemes over $\spec H_T^{\bl}(\pt)\simeq (\ttd)^*$, where $\ZZZ$ is defined in Lemma \ref{dunnowheretoplacelemmaa}. (Notice that $\b(e+f(h),-)|_{\ttd}=h$ for any $h\in (\ttd)^*$.)

\begin{proposition}\label{BFenhanceequalYZ}
After identifying  $\FF_{\hp=0}(\dkf(\KK))$ with $\OO(\Bed)$ via Lemma \ref{dunnowheretoplacelemmaa} and the above isomorphism, we have $\FF_{\hp=0}(\pbfd)=\Phi_{YZ}$.
\end{proposition}
\begin{proof}
There is a $W$-action on $\FF_{\hp=0}(\dkf(\KK))$ which is induced by the isomorphism $\tt_{\KK}$ (see Lemma \ref{Alemma2}). There is also a $W$-action on $\OO(\Bed)$ because $\Bed\simeq\ZZZ$ is the base change of the \textit{universal centralizer}
\[ \mathcal{C}:= \Gd\setminus\!\!\!\setminus\{ (g,\xi)\in \Gd\times\ggd|~g\cdot\xi = \xi\} \simeq \Umd\setminus\{ (g,\xi)\in \Gd\times(e+\mathfrak{b}^{\vee}_-)|~g\cdot\xi=\xi\}\]
along the canonical morphism $\ttd\simeq e+\ttd \ra \Umd\setminus (e+\bbd_-) \simeq  \ttd/W$. It is not difficult to see that these two actions coincide under our identification (because $\FF_{\hp=0}(\kf(\KK))\simeq \OO(\mathcal{C})$ which can be proved the same way how Lemma \ref{dunnowheretoplacelemmaa} is proved). By definition, $\FF_{\hp=0}(\pbfd)$ is $W$-equivariant, and by the proof of \cite[Proposition 6.6]{YZ}, $\Phi_{YZ}$ is $W$-equivariant as well.

Let $\lambda\in\cop$. Denote by $v_e\in S(\lambda)_{\lambda}$ the vector corresponding to the unique MV cycle $\{t^{\lambda}\}$ of type $\lambda$ and weight $\lambda$ via the composite isomorphism \eqref{MVbasis}. By \cite[Remark 3.4]{YZ}, we have
\[ \Phi_{YZ}([f_{\lambda,v_e}]) = [t^{\lambda}]\in H_0^T(\ag).\]
Since $\Phi_{YZ}$ is a $W$-equivariant isomorphism and $[t^{\lambda}]$ ($\lambda\in\cop$) together with their $W$-translates generate the module $H_{\bl}^T(\ag)$ over $\fof(H_T^{\bl}(\pt))$, it follows that $[f_{\lambda,v_e}]$ ($\lambda\in\cop$) together with their $W$-translates generate the module $\OO(\Bed)$ over $\fof(H_T^{\bl}(\pt))$. Therefore, we are done if we can show 
\begin{equation}\label{BFenhanceequalYZeq1} \nonumber
\FF_{\hp=0}(\pbfd)([f_{\lambda,v_e}]) = [t^{\lambda}]
\end{equation}
for any $\lambda\in\cop$. But this is a special case of Proposition \ref{BFenhancedivisor}.
\end{proof}

%%%%%%%%%%%%%%%%%%%%%%%%%%%%%%%%%%%%
%%%%%%%%%%%%%%%%%%%%%%%%%%%%%%%%%%%%
%%%%%%%%%%%%%%%%%%%%%%%%%%%%%%%%%%%%
%%%%%%%%%%%%%%%%%%%%%%%%%%%%%%%%%%%%
%%%%%%%%%%%%%%%%%%%%%%%%%%%%%%%%%%%%
%%%%%%%%%%%%%%%%%%%%%%%%%%%%%%%%%%%%
%%%%%%%%%%%%%%%%%%%%%%%%%%%%%%%%%%%%
\section{Proof of main results}\label{final}
\subsection{Summary}\label{finalstrategy}
We construct $\Mir$ as stated in Theorem \ref{main} by completing the following diagram.
\begin{equation}\label{finalstrategydiag}\nonumber
\begin{tikzpicture}
\tikzmath{\x1 = 5; \x2 = 2.5; \x3=7;}
\node (A) at (0,0) {$\bries$} ;
\node (B) at (\x1,0) {$\dm{\Gd}\ot\OO(\ZL)/\VV$} ;
\node (C) at (\x1+\x3,0) {$ QH_{\That}^{\bl}(G/P)[\ecc] $} ;
\node (D) at (\x1+0.5*\x3,\x2) {$\dkf(\KK)\ot\OO(\ZL)/\WW$} ;
\node (E) at (0.5*\x1,1.03*\x2) {} ;

%\path[->] (A) edge node[left]{$\Phi_{\rhod}$} (E);
%\path[->] (E) edge node[right]{$\Phi_0$} (B);
\path[->] (A) edge node[above]{$\Phi_0$} (B);
\path[->] (D) edge node[left]{$\Phi_{1~} ~$} (B);
\path[->] (D) edge node[right]{$~\Phi_2$} (C);
\path[->] (B) edge node[above]{$\Phi_3$} (C);
\end{tikzpicture}
\end{equation}

Here, 
\begin{enumerate}
\item $\dkf(\KK)$ is the desymmetrized quantum Toda lattice defined in Definition \ref{BFQTLdef}.

\item $\VV$ and $\WW$ are some submodules which we will define in Section \ref{finaltwomod} and Section \ref{finaltwomodw} respectively.

\item $\Phi_0$ is a $\utd\ot\OO(\ZL)$-linear map which we will define in Section \ref{finalphi0}. We will show that it is bijective (Proposition \ref{finalphi0bijective}).

\item $\Phi_1$ is a $\utd\ot\OO(\ZL)$-linear map which we will define in Section \ref{finalphi1}. We will show that it is surjective (Proposition \ref{finalphi1phi1surj}) and becomes bijective after applying the functor $N\mapsto N/\hp N$ (Proposition \ref{finalphi1phi1limitisbij}).

\item $\Phi_2$ is a $\utd\ot\OO(\ZL)$-linear map which we will define in Section \ref{finalphi2}. We will show that it is surjective (Proposition \ref{finalphi2phi2surj}) and becomes bijective after applying the functor $N\mapsto N/\hp N$ (Proposition \ref{finalphi2phi2limitisbij}).

\item $\Phi_3$ is the unique $\utd\ot\OO(\ZL)$-linear map such that $\Phi_3\circ\Phi_1=\Phi_2$. We will show that it exists (Proposition \ref{finalphi3exist}) and is bijective (Proposition \ref{finalphi3bij}) in Section \ref{finalphi3}.
\end{enumerate}

We define
\[\Mir := \Phi_3\circ\Phi_0\] 
which is thus a bijective $\utd\ot\OO(\ZL)$-linear map. We will show that it is $\dm{\ZL}$-linear (Proposition \ref{finalfinaldmlinear}) in Section \ref{finalfinal}. In Section \ref{conclusionofproof}, we will verify the rest of the properties of $\Mir$ as well as other statements in Theorem \ref{main} and Theorem \ref{main2}.

\begin{remark}
Readers who have read Section \ref{outlinefriendly} may find the following dictionary useful:
\[ \widetilde{\Phi}_R^{\hp} = \Phi_0^{-1}\circ\Phi_1\circ q|_{\dkf(\KK)},\quad \Phi_{YZ}^{\hp}=\pbfd,\quad \Phi_{PLS}^{\hp}=\pgmpt\]
and 
\[ \Phi_2\circ q|_{\dkf(\KK)}=\pgmpt\circ\Phi_{BF}^{\delta}\]
where $q:\dkf(\KK)\otimes\OO(\ZL)\ra\dkf(\KK)\otimes\OO(\ZL)/\WW$ is the quotient map and $\pbfd$ (resp. $\pgmpt$) comes from Definition \ref{BFpbfdef} (resp. Definition \ref{GMPcomputedef}).
\end{remark}
%%%%%%%%%%%%%%%%%%%%%%%%%%%%%%%%
%%%%%%%%%%%%%%%%%%%%%%%%%%%%%%%%
\subsection{The submodule $\VV$}\label{finaltwomod}
Let $\cchh:\nd\ra\CC$ be the character from \eqref{cchhdef}.
\begin{definition} \label{finaltwomodVdef}
Define $\VV\subset \dm{\Gd}\ot\OO(\ZL)$ to be the sum of 
\begin{enumerate}
\item the left ideal generated by $x^L-\cchh(x)$ with $x\in\nnd_-$;

\item the left ideal generated by $x^R-\cchh(x)$ with $x\in\nnd_-$;

\item the right ideal generated by $x^L$ with $x\in\nnd$; and

\item the sub-$\utd$-module generated by $\vp\in\OO(\Gd)\ot\OO(\ZL)\simeq \OO(\Gd\times\ZL)$ satisfying $\vp|_{\Xd}\equiv 0$ where the $\utd$-module structure on $\dm{\Gd}\ot\OO(\ZL)$ is defined by left multiplication of left-invariant vector fields on $\Gd$.
\end{enumerate}
\end{definition}

\begin{lemma} $\VV$ is homogeneous with respect to the natural grading on $\dm{\Gd}\ot\OO(\ZL)$ so there is a grading on the quotient $\dm{\Gd}\ot\OO(\ZL)/\VV$.
\end{lemma}
\begin{proof}
This is straightforward and left to the reader.
\end{proof}

\begin{definition}\label{finaltwomodtwistdef}
Define a $\utd$-module structure $\twdot$ on $\dm{\Gd}\ot\OO(\ZL)$ by the equality
\[ x\twdot y:= (x^L-\hp\rhod(x))y,\quad x\in\ttd, ~y\in \dm{\Gd}\ot\OO(\ZL)\]
where $\rhod$ is the half-sum of the positive coroots. It is not hard to see that $\twdot$ induces a $\utd$-module structure on $\dm{\Gd}\ot\OO(\ZL)/\VV$. 
\end{definition} 

Recall the paragraph before Proposition \ref{BFenhancedivisor} where we defined, for any $\lambda\in\cop$, the vector $v_e\in S(\lambda)_{\lambda}$ and, for any $v\in S(\lambda)$, the regular function $f_{\lambda,v}\in\OO(\Gd)$.

\begin{definition}\label{finaltwomodqvdef} For any $\lambda\in\cop$, define
\[ \widetilde{\shift}^{\VV}_{\lambda} :  \dm{\Gd}\ot\OO(\ZL)\ra \dm{\Gd}\ot\OO(\ZL)\]
by
\[ \widetilde{\shift}^{\VV}_{\lambda}(y) := f_{\lambda,v_e}y  \qquad y\in \dm{\Gd}\ot\OO(\ZL). \]
It is $(-\lambda)$-twisted $\utd$-linear in the sense of Definition \ref{Bshifttwistedlinear}.
\end{definition}

\begin{lemma}\label{finaltwomodvextend} (Shift operators) $ \widetilde{\shift}^{\VV}_{\lambda} $ descends to a $(-\lambda)$-twisted $\utd$-linear map
\[\shift^{\VV}_{\lambda}:\dm{\Gd}\ot\OO(\ZL)/\VV\ra \dm{\Gd}\ot\OO(\ZL)/\VV.\]
\end{lemma}
\begin{proof}
We have to show $\widetilde{\shift}^{\VV}_{\lambda} (\VV)\subseteq \VV$. Clearly $\widetilde{\shift}^{\VV}_{\lambda}$ preserves the subspaces (1), (2) and (4) from Definition \ref{finaltwomodVdef}. Let $y$ be an element of the remaining subspace (3). We may assume $y=x^Ly'$ for some $x\in\mathfrak{n}^{\vee}$ and $y'\in \dm{\Gd}\ot\OO(\ZL)$. Since $\LL_{x^L}f_{\lambda,v_e}=0$, it follows that
\[ f_{\lambda,v_e}y=x^Lf_{\lambda,v_e}y' - \hp(\LL_{x^L}f_{\lambda,v_e})y' = x^Lf_{\lambda,v_e}y'\equiv 0~(\bmod{\VV}).\]
This gives $\widetilde{\shift}^{\VV}_{\lambda}(y)\in\VV$.
\end{proof}

\begin{lemma}\label{finalVshiftinvertible} $\shift^{\VV}_{\lambda}$ is invertible.
\end{lemma}
\begin{proof}
This follows from the fact that $\shift^B_{\lambda}$ is invertible (Lemma \ref{Bshiftlemmamulti}) and two results which will be proved in Section \ref{finalphi0}, namely, Lemma \ref{finalphi0shift} and Proposition \ref{finalphi0bijective}.
\end{proof}
%%%%%%%%%%%%%%%%%%%%%%%%%%%%
%%%%%%%%%%%%%%%%%%%%%%%%%%%%
\subsection{The submodule $\WW$}\label{finaltwomodw} 
Recall $\KK:=\kf(\dm{\Gd})$ defined in Definition \ref{KKdef}.
\begin{definition}\label{finaltwomodadef}
For any $\lambda\in\cop$ and $v\in S(\lambda)$, define $a_{\lambda,v}\in \kf(\KK)$ as follows. Fix a homogeneous $\utd^W$-basis $\{z_1,\ldots,z_{|W|}\}$ of $\utd$ with $z_1=1$. Notice that $z_i$ has strictly positive degree for any $i\geqslant 2$. Recall the $\utd$-linear map $\tt_{\fr(S(\lambda))}:\utd\ot_{\zgd}\kf(\fr(S(\lambda)))\ra \dkf(\fr(S(\lambda)))$. Since it is bijective (Lemma \ref{BFenhanceKFisom}), there exist unique $\widetilde{a}_{\lambda,v,1},\ldots, \widetilde{a}_{\lambda,v,|W|}\in \kf(\fr(S(\lambda)))$ such that 
\begin{equation}\label{finaltwomodadefeq1}
[1\ot v]=\sum_{i=1}^{|W|} \tt_{\fr(S(\lambda))}(z_i\ot \widetilde{a}_{\lambda,v,i}) .
\end{equation}
Let $f_{\lambda}:\fr(S(\lambda))[2\rho(\lambda)]\ra \KK$ be the morphism in $\hct$ defined by $1\ot v'\mapsto [f_{\lambda,v'}]$. ($f_{\lambda,v'}$ is defined in the paragraph before Proposition \ref{BFenhancedivisor}.) We define
\[ a_{\lambda,v}:= \kf(f_{\lambda})(\widetilde{a}_{\lambda,v,1}). \]
\end{definition}

\begin{lemma}\label{finaltwomodarmk}
Suppose $v\in S(\lambda)_{\mu}$ for some $\mu\in\Q$. Then $a_{\lambda,v}$ is represented by $f_{\lambda,v}$ plus an element which is equal to a $\ubd$-linear combination of $f_{\lambda,v'}$ with $v'\in S(\lambda)_{<\mu}$. Here, the $\ubd$-module structure on $\dm{\Gd}$ is given by the left multiplication of left-invariant vector fields.
\end{lemma}
\begin{proof}
By Lemma \ref{Alemma2} applied to $M:=\fr(S(\lambda)_{\leqslant\mu})$, every $\widetilde{a}_{\lambda,v,i}$ lies in the image of the canonical map $\kf(M)\ra\kf(\fr(S(\lambda)))$ and hence is represented by a $\ubd$-linear combination of $1\ot v''$ with $v''\in S(\lambda)_{\leqslant\mu}$. By a degree argument, we cannot have $v''\in S(\lambda)_{\mu}$ if $i\geqslant 2$. The result now follows from \eqref{finaltwomodadefeq1} and the canonical isomorphism of $\utd$-modules
\[ \dkf(M)/\dkf(\fr(S(\lambda)_{<\mu})) \simeq \twist_{-\rhod}(\utd)\ot S(\lambda)_{\mu} .\]
\end{proof}

Recall $\dkf(\KK)$ is naturally a right $\kf(\KK)$-module. Hence $\dkf(\KK)\ot\OO(\ZL)$ is a right $\kf(\KK)\ot\OO(\ZL)$-module. For any $\lambda\in\cop$ and $w\in W$, denote by $v_w\in S(\lambda)_{w(\lambda)}$ the vector corresponding to the unique MV cycle of type $\lambda$ and weight $w(\lambda)$ via the composite isomorphism \eqref{MVbasis}. (Notice that $v_w$ depends on $\lambda$. But since it is unlikely to cause any confusion, we drop $\lambda$ from the notation for simplicity.) 
\begin{definition} \label{finaltwomodWdef}
Define $\WW\subset \dkf(\KK)\ot\OO(\ZL)$ to be the vector subspace generated by elements of the form $x
\cdot a $ where $x\in\dkf(\KK)\ot\OO(\ZL)$ and $a$ is either 
\begin{enumerate}
\item $a_{\lambda,v_{\wp}}-q^{[w_0(\lambda)]}$ for some $\lambda\in\cop$; or

\item $a_{\lambda,v}$ for some $\lambda\in\cop$ and $v\in S(\lambda)_{<\wp(\lambda)}$.
\end{enumerate}
\end{definition}

\begin{remark}
The definition of $\WW$ is inspired by \cite{me2}. See Section 6.1 therein. 
\end{remark}

\begin{lemma} $\WW$ is homogeneous with respect to the natural grading on $ \dkf(\KK)\ot\OO(\ZL)$ so there is a grading on the quotient $ \dkf(\KK)\ot\OO(\ZL)/\WW$.
\end{lemma}
\begin{proof}
This is straightforward and left to the reader.
\end{proof}

\begin{definition}\label{finaltwomodqwdef} For any $\lambda\in\cop$, define
\[ \widetilde{\shift}^{\WW}_{\lambda} : \dkf(\KK)\ot\OO(\ZL) \ra \dkf(\KK)\ot\OO(\ZL)\]
by
\begin{equation}\label{finaltwomodqwdefeq1}
\widetilde{\shift}^{\WW}_{\lambda}( [x] ):= [f_{\lambda,v_e}x] \qquad [x]\in \dkf(\KK)\ot\OO(\ZL)
\end{equation}
where $x\in\dm{\Gd}\ot\OO(\ZL)$ is any representative of $[x]$. (Notice that $\dkf(\KK)$ is a sub-quotient of $\dm{\Gd}$.) It is $(-\lambda)$-twisted $\utd$-linear in the sense of Definition \ref{Bshifttwistedlinear}.
\end{definition}

\begin{remark}\label{finaltwomodqwrmk}
One has to show that the RHS of \eqref{finaltwomodqwdefeq1} is well-defined, i.e. $f_{\lambda,v_e}x$ represents an element of $\dkf(\KK)\ot\OO(\ZL)$ which is independent of the representative $x$. Since the proof is straightforward, we omit it. 
\end{remark}

\begin{lemma}\label{finaltwomodwextend} (Shift operators) $ \widetilde{\shift}^{\WW}_{\lambda} $ descends to a $(-\lambda)$-twisted $\utd$-linear map
\[\shift^{\WW}_{\lambda}:\dkf(\KK)\ot\OO(\ZL)/\WW\ra \dkf(\KK)\ot\OO(\ZL)/\WW.\]
\end{lemma}
\begin{proof}
Obvious.
\end{proof}

\begin{lemma}\label{finalWshiftinvertible} $\shift^{\WW}_{\lambda}$ is invertible.
\end{lemma}
\begin{proof}
We prove the lemma by showing that (1) $\widetilde{\shift}^{\WW}_{\lambda}$ is invertible and (2) its inverse preserves $\WW$. 

Since $\pbfd$ is an isomorphism of modules with respect to the isomorphism $\pbfk$ of rings (Theorem \ref{BFqttheorem}) and $\pbfd([f_{\lambda,v_e}])=[t^{\lambda}]$ (Proposition \ref{BFenhancedivisor}), $\widetilde{\shift}^{\WW}_{\lambda}|_{\dkf(\KK)}$ is identified with $[t^{\lambda}]\aggt -$ under this isomorphism. (A priori, $[t^{\lambda}]\aggt -$ is a map from $H_{-\bl}^{\Ghat}(\ag)$ to $H_{-\bl}^{\That}(\ag)$. But we can extend it to an endomorphism of $H_{-\bl}^{\That}(\ag)$ by extension of scalars because it is a homomorphism of modules with respect to the composition $H_{\Ghat}^{\bl}(\pt) \hookrightarrow H_{\That}^{\bl}(\pt) \xrightarrow{\twist_{-\lambda}} H_{\That}^{\bl}(\pt)$.) Clearly, $[t^{\lambda}]\aggt -$ has an inverse which is given by $[t^{-\lambda}]\aggt -$. This proves (1). 

Observe that the left multiplication by $[t^{-\lambda}]$ commutes with the right multiplication by any elements of $H_{-\bl}^{\Ghat}(\ag)$. Since $\WW$ is additively generated by $x\cdot a$ where $x\in\dkf(\KK)\ot\OO(\ZL)$ is arbitrary and $a$ belongs to a subset of $\kf(\KK)\ot\OO(\ZL)$ (see Definition \ref{finaltwomodWdef}), (2) follows from this observation and the fact that $\pbfd$ is an isomorphism of modules with respect to $\pbfk$.
\end{proof}
%%%%%%%%%%%%%%%%%%%%%%%%%%%%
%%%%%%%%%%%%%%%%%%%%%%%%%%%%
\subsection{The map $\Phi_0$}\label{finalphi0}
For any $\a\in R$, pick a generator $e_{\ad}$ of $\ggd_{\ad}$. Denote by $\tt_{\a,\b}\in\OO(\Gd)$ ($\b\in R$) and $\tt_{\a,\hh}:\Gd\ra\ttd$ the unique functions satisfying
\begin{equation}\label{finalphi0descendlemmaeq1}
e_{\ad}^R = \sum_{\b\in R}\tt_{\a,\b}e_{\bd}^L +\sum_i\langle h^i, \tt_{\a,\hh}\rangle h_i^L
\end{equation}
where $\{h_i\}$ and $\{h^i\}$ are dual bases of $\ttd$ and $(\ttd)^*$ respectively. Consider the splitting $\ggd=\bbd\oplus\nnd_-$ which induces a splitting of the tangent bundle $\mathcal{T}_{\Gd}$ of $\Gd$ by left translations. Let $\pr_{\bbd}^L$ and $\pr_{\nnd_-}^L$ denote the projections from $\mathcal{T}_{\Gd}$ onto the corresponding direct summands. Let $\a\in -R^+$. Define
\[\widetilde{\zeta}_{\ad}:= \pr_{\bbd}^L(e_{\ad}^R)\oplus 0\in\mathfrak{X}(\Gd\times\ZL/\ZL).\]

\begin{lemma}\label{finalphi0longlemma}
For any $\a\in -R^+$ and $\widetilde{\vp}\in\OO(\Gd\times\ZL)$, we have
\begin{align*}
& \hp\rhod(\tt_{\a,\hh})\widetilde{\vp} + \hp\LL_{\widetilde{\zeta}_{\ad}}\widetilde{\vp} + \left( \cchh(e_{\ad}) - \sum_{\b\in - R^+}\tt_{\a,\b}\cchh(e_{\bd})\right)\widetilde{\vp} - \sum_ih_i\twdot\langle h^i,\tt_{\a,\hh}\rangle\widetilde{\vp}\\
=~& \widetilde{\vp}\left( -(e_{\ad}^R -\cchh(e_{\ad}))+\sum_{\b\in -R^+}\tt_{\a,\b} (e_{\bd}^L-\cchh(e_{\bd}))\right) + \sum_{\b\in R^+}e_{\bd}^L\tt_{\a,\b} \widetilde{\vp} ~\in \dm{\Gd}\ot\OO(\ZL).
\end{align*}
(Recall $\twdot$ is defined in Definition \ref{finaltwomodtwistdef}.) 
\end{lemma}
\begin{proof}
First by the equality $\pr_{\bbd}^L(e_{\ad}^R)=e_{\ad}^R - \sum_{\b\in -R^+}\tt_{\a,\b}e_{\bd}^L$, we have
\begin{align}\label{finalphi0longlemmaeq1}
& -\widetilde{\vp}\widetilde{\zeta}_{\ad} + \left(\cchh(e_{\ad})-\sum_{\b\in -R^+}\tt_{\a,\b}\cchh(e_{\bd})\right)\widetilde{\vp} \nonumber \\
= ~& \widetilde{\vp}\left( -(e_{\ad}^R-\cchh(e_{\ad})) + \sum_{\b\in -R^+}\tt_{\a,\b}(e_{\bd}^L-\cchh(e_{\bd}))\right).
\end{align}

Next, using $\pr_{\bbd}^L(e_{\ad}^R) =\sum_i\langle h^i,\tt_{\a,\hh}\rangle h_i^L +\sum_{\b\in R^+}\tt_{\a,\b}e_{\bd}^L$, we have 
\begin{align}
\widetilde{\zeta}_{\ad}\widetilde{\vp} =~& 
\left( \sum_i\langle h^i,\tt_{\a,\hh}\rangle h_i^L +\sum_{\b\in R^+}\tt_{\a,\b}e_{\bd}^L\right)\widetilde{\vp} \nonumber \\
= ~& \left( \sum_i h_i^L \langle h^i,\tt_{\a,\hh}\rangle  -\hp \sum_i\langle h^i,\LL_{h_i^L}\tt_{\a,\hh}\rangle +\sum_{\b\in R^+} e_{\bd}^L\tt_{\a,\b} - \hp\sum_{\b\in R^+}\LL_{e_{\bd}^L}\tt_{\a,\b} \right)\widetilde{\vp}.  \nonumber
\end{align}
Observe that $\LL_{h_i^L}\tt_{\a,\hh} =0$ and $\LL_{e_{\bd}^L}\tt_{\a,\b}  = \bd(\tt_{\a,\hh})$. It follows that
\begin{equation}\label{finalphi0longlemmaeq2}
\widetilde{\zeta}_{\ad}\widetilde{\vp} =\sum_ih_i\twdot\langle h^i,\tt_{\a,\hh}\rangle\widetilde{\vp}-\hp\rhod(\tt_{\a,\hh})\widetilde{\vp}+\sum_{\b\in R^+}e_{\bd}^L\tt_{\a,\b}  \widetilde{\vp}.
\end{equation}

Our result now follows from \eqref{finalphi0longlemmaeq1}, \eqref{finalphi0longlemmaeq2} and the equality $\hp\LL_{\widetilde{\zeta}_{\ad}}\widetilde{\vp} = \widetilde{\zeta}_{\ad}\widetilde{\vp}- \widetilde{\vp}\widetilde{\zeta}_{\ad}$.
\end{proof}

Recall the Rietsch mirror $(\Xd,W,\pi,p,\vol)$ (Definition \ref{Rietschmirrordef}). Let $\a\in -R^+$. Observe that $\widetilde{\zeta}_{\ad}$ is tangent to $\Xd$. Define 
\[\zeta_{\ad}:=\widetilde{\zeta}_{\ad}|_{\Xd}\in\mathfrak{X}(\Xd/\ZL).\]

\begin{lemma}\label{finalphi0globalframelemma}
$\{\zeta_{\ad}\}_{\a\in w_0(R^+\setminus R^+_P)}$ is a global frame of $\Xd$ relative to $\ZL$.
\end{lemma}
\begin{proof}
This amounts to showing $\mathfrak{n}^{\vee}_-\cap g^{-1}\cdot V=0$ for any $g\in\Bd\cap\Umd(\wpd)^{-1}\ZL\Umd$ where $V:=\bigoplus_{\a\in w_0(R^+\setminus R^+_P)}\ggd_{\ad}\subset\ggd$. This follows from the facts that $\Umd$ preserves $V$, $\wpd$ maps $V$ into $\mathfrak{n}^{\vee}$ and $\Bmd$ preserves $\mathfrak{n}^{\vee}_-$.
\end{proof}

\begin{remark} Compare the frame $\{\zeta_{\ad}\}_{\a\in w_0(R^+\setminus R^+_P)}$ with the one introduced in \cite[Lemma 5.4]{Rietsch}.
\end{remark}

We now define $\Phi_0$. Since the fiberwise volume form $\vol\in \O^{top}(\Xd/\ZL)$ is nowhere vanishing, we have 
\[\O^{top}(\Xd/\ZL)= \OO(\Xd)\cdot \vol\quad\text{and}\quad \O^{top-1}(\Xd/\ZL)=\bigoplus_{\a\in w_0(R^+\setminus R^+_P)} \OO(\Xd)\cdot \iota_{\zeta_{\ad}}\vol.\]
It is not hard to see that
\[ \briesp\simeq M/\UU_{\vol}\]
where $M:=\utd\ot\OO(\Xd)$ and $\UU_{\vol}\subseteq M$ is the vector subspace generated by elements of the form 
\begin{align*}\label{finalphi0predefeq1}
& m_{\a}(z,\vp)\\
:=~ &  z\ot\left(\hp\left(\frac{\LL_{\zeta_{\ad}}\vol}{\vol}\right)\vp + \hp\LL_{\zeta_{\ad}}\vp + (\LL_{\zeta_{\ad}} W)\vp \right) - \sum_izh_i\ot(\iota_{\zeta_{\ad}}p^*\langle h^i,\mccccc_{\Td}\rangle)\vp
\end{align*}
where $\a\in w_0(R^+\setminus R_P^+)$, $z\in\utd$ and $\vp\in\OO(\Xd)$. Notice that $m_{\a}(z,\vp)$ is in fact well-defined for any $\a\in -R^+$ and also contained in $\UU_{\vol}$.

\begin{definition}\label{finalphi0def}
Define a $\utd\ot\OO(\ZL)$-linear map
\[ \widetilde{\Phi}_0 : M\ra \dm{\Gd}\ot\OO(\ZL)/\VV\]
by
\[ \widetilde{\Phi}_0(z\ot\vp):= [z\twdot\widetilde{\vp}]\qquad z\in\utd,~\vp\in\OO(\Xd)\]
where $\widetilde{\vp}\in\OO(\Gd\times\ZL)$ is any extension of $\vp$ and $\twdot$ is defined in Definition \ref{finaltwomodtwistdef}. (It is well-defined because $\VV$ contains the subspace (4) from Definition \ref{finaltwomodVdef}.)
\end{definition}

\begin{lemma}\label{finalphi0descendlemma} $\widetilde{\Phi}_0$ descends to a $\utd\ot\OO(\ZL)$-linear map
\[ \Phi_0:M/\UU_{\vol}\simeq  \briesp\ra \dm{\Gd}\ot\OO(\ZL)/\VV.\]
\end{lemma}
\begin{proof}
We have to show $\widetilde{\Phi}_0(m_{\a}(z,\vp))=0$ for any $\a\in w_0(R^+\setminus R_P^+)$, $z\in\utd$ and $\vp\in\OO(\Xd)$. By Lemma \ref{ApplemmaA}, Lemma \ref{ApplemmaB} and Lemma \ref{ApplemmaC}, 
\[  \rhod(\tt_{\a,\hh}),\quad  \cchh(e_{\ad})-\sum_{\b\in -R^+}\tt_{\a,\b}\cchh(e_{\bd}) \quad\text{and}\quad\langle h^i,\tt_{\a,\hh}\rangle \]
are extensions of 
\[ \frac{\LL_{\zeta_{\ad}}\vol}{\vol},\quad \LL_{\zeta_{\ad}}W\quad\text{and}\quad\iota_{\zeta_{\ad}}p^*\langle h^i,\mccccc_{\Td}\rangle\]
respectively. It follows that $\widetilde{\Phi}_0(m_{\a}(z,\vp))$ is represented by  
\[ z\twdot\left(  \hp\rhod(\tt_{\a,\hh})\widetilde{\vp} +\hp\LL_{\widetilde{\zeta}_{\ad}}\widetilde{\vp} + \left( \cchh(e_{\ad}) - \sum_{\b\in - R^+}\tt_{\a,\b}\cchh(e_{\bd})\right)\widetilde{\vp} - \sum_ih_i\twdot\langle h^i,\tt_{\a,\hh}\rangle\widetilde{\vp}\right) \]
where $\widetilde{\vp}\in\OO(\Gd\times\ZL)$ is an extension of $\vp$. The result now follows from Lemma \ref{finalphi0longlemma}.
\end{proof}

\begin{lemma}\label{finalphi0graded} $\Phi_0$ is graded.
\end{lemma}
\begin{proof}
This is because $[\vol]$ has degree 0 (Lemma \ref{Fiberwisevolumeformdegzero}).
\end{proof}

\begin{lemma} \label{finalphi0shift} For any $\lambda\in\cop$, we have
\[ \Phi_0\circ\shift^B_{\lambda} = \shift^{\VV}_{\lambda}\circ\Phi_0\]
where $\shift^B_{\lambda}$ and $\shift^{\VV}_{\lambda}$ come from Lemma \ref{Bshiftwell} and Lemma \ref{finaltwomodvextend} respectively.
\end{lemma}
\begin{proof}
It suffices to note that $f_{\lambda,v_e}$ (regarded as a regular function on $\Gd\times\ZL$) is an extension of $\lambda\circ p\in \OO(\Xd)$. 
\end{proof}

\begin{proposition}\label{finalphi0bijective}
$\Phi_0$ is bijective.
\end{proposition}
\begin{proof}
We construct the inverse of $\Phi_0$ as follows. 

Denote by $\VV'\subset\VV\subset \dm{\Gd}\ot\OO(\ZL)$ the sum of the ideals (1) and (3) from Definition \ref{finaltwomodVdef}. By PBW's theorem, we have
\[ \dm{\Gd}\simeq \und\ot_{\CC[\hp]}(\utd\ot\OO(\Gd))\ot_{\CC[\hp]}\unmd.\]
The characters 
\[ 0_{\mathfrak{n}^{\vee}}: \mathfrak{n}^{\vee}\ra \CC\quad\text{ and }\quad \cchh:\nd\ra\CC\]
induce $\CC[\hp]$-algebra homomorphisms
\[ U_{\hp}(0_{\mathfrak{n}^{\vee}}):\und\ra\CC[\hp]\quad\text{ and }\quad U_{\hp}(\cchh):\unmd\ra \CC[\hp]\]
respectively, and hence splittings
\[ \und\simeq\CC[\hp]\oplus \mathfrak{n}^{\vee}\cdot \und\quad \text{ and }\quad \unmd\simeq \CC[\hp]\oplus\ker U_{\hp}(\cchh).\]
This yields a splitting
\begin{equation}\label{finalphi0bijectiveeq1}
\dm{\Gd}\ot\OO(\ZL)\simeq \VV'\oplus \left( \utd\ot\OO(\Gd\times\ZL)\right).
\end{equation}

Now define 
\[ \widetilde{\Phi}_0' :\dm{\Gd}\ot\OO(\ZL)/\VV' \ra M/\UU_{\vol} \]
by 
\[  \widetilde{\Phi}_0'([z\twdot \widetilde{\vp}] ) := [z\ot \widetilde{\vp}|_{\Xd}]\qquad z\in\utd,~\widetilde{\vp}\in\OO(\Gd\times\ZL).\]
We show that $\widetilde{\Phi}_0'$ descends to 
\[ \Phi_0': \dm{\Gd}\ot\OO(\ZL)/\VV \ra M/\UU_{\vol} .\]
This will complete the proof because $\Phi_0'$ will clearly be the inverse of $\Phi_0$. It is clear that $\widetilde{\Phi}_0'$ sends the subspace (4) from Definition \ref{finaltwomodVdef} to 0. It remains to deal with the left ideal (2). By \eqref{finalphi0bijectiveeq1} and the fact that $e_{\beta^{\vee}_1}^R-\cchh(e_{\beta^{\vee}_1})$ commutes with $e_{\beta^{\vee}_2}^L-\cchh(e_{\beta^{\vee}_2})$ for any $\b_1,\b_2\in -R^+$, every element of the left ideal (2) is equal, modulo $\VV'$, to a sum of elements of the form $z\twdot \widetilde{\vp}(e_{\ad}^R-\cchh(e_{\ad}))$ where $z\in\utd$, $\widetilde{\vp}\in\OO(\Gd\times\ZL)$ and $\a\in -R^+$. By Lemma \ref{finalphi0longlemma}, Lemma \ref{ApplemmaA}, Lemma \ref{ApplemmaB} and Lemma \ref{ApplemmaC}, we have
\[ \widetilde{\Phi}_0'\left([z\twdot \widetilde{\vp}(e_{\ad}^R-\cchh(e_{\ad}))]\right) = [m_{\a}(z, \widetilde{\vp}|_{\Xd})]=0. \]
\end{proof}

\begin{remark}
The construction of $\Phi_0$ is inspired by Teleman's interpretation of the Rietsch mirror \cite[Section 6.6]{T1}.
\end{remark}

%%%%%%%%%%%%%%%%%%%%%%%%%%%%
%%%%%%%%%%%%%%%%%%%%%%%%%%%%
\subsection{The map $\Phi_1$}\label{finalphi1}
Denote by $\VV''\subset\VV\subset \dm{\Gd}\ot\OO(\ZL)$ the sum of the ideals (1), (2) and (3) from Definition \ref{finaltwomodVdef}. Observe that $\dkf(\KK)$ is a subquotient of $\dm{\Gd}$, i.e. $\dkf(\KK)\simeq N_1/N_2$ for some subspaces $N_2\subseteq N_1\subseteq \dm{\Gd}$. It is not difficult to see that $N_2\subseteq \VV''$.

\begin{definition}\label{finalphi1phi1def}
Define 
\[ \widetilde{\Phi}_1:\dkf(\KK)\ot\OO(\ZL)\ra \dm{\Gd}\ot\OO(\ZL)/\VV\]
to be the composition of the two canonical homomorphisms
\[ \dkf(\KK)\ot\OO(\ZL)\ra\dm{\Gd}\ot\OO(\ZL)/\VV''  \]
and
\[ \dm{\Gd}\ot\OO(\ZL)/\VV''\ra \dm{\Gd}\ot\OO(\ZL)/\VV .\]
\end{definition}

\begin{lemma}\label{finalphi1phi1ok}
$\widetilde{\Phi}_1$ descends to a $\utd\ot\OO(\ZL)$-linear map
\[ \Phi_1: \dkf(\KK)\ot\OO(\ZL)/\WW \ra \dm{\Gd}\ot\OO(\ZL)/\VV . \] 
\end{lemma} 
\begin{proof}
Recall from Definition \ref{finaltwomodWdef} that $\WW$ is generated by elements of the form $x
\cdot a $ where $x\in\dkf(\KK)\ot\OO(\ZL)$ and $a$ is either $a_{\lambda,v_{\wp}}-q^{[w_0(\lambda)]}$ for some $\lambda\in\cop$ or $a_{\lambda,v}$ for some $\lambda\in\cop$ and $v\in S(\lambda)_{<\wp(\lambda)}$. We have to show $\widetilde{\Phi}_1(x\cdot a)=0$. Let $\widetilde{x}\in \dm{\Gd}\ot\OO(\ZL)$ be a representative of $x$. It is not hard to see that $\dm{\Gd}\ot\OO(\ZL)/\VV''$ is naturally a right $\kf(\KK)\ot\OO(\ZL)$-module and the first canonical homomorphism from Definition \ref{finalphi1phi1def} is $\kf(\KK)\ot\OO(\ZL)$-linear. This allows us to assume $\widetilde{x}\in \utd\ot\OO(\Gd\times\ZL)$ where $\utd\subset\dm{\Gd}$ is generated by $z^L$ with $z\in\ttd$ (see \eqref{finalphi0bijectiveeq1}). By Lemma \ref{finaltwomodarmk}, $x\cdot a_{\lambda,v}$ (resp. $x\cdot (a_{\lambda,v_{\wp}}-q^{[w_0(\lambda)]})$) is represented by (resp. $\widetilde{x}(f_{\lambda,v_{\wp}}-q^{[w_0(\lambda)]})$ plus) a sum of elements of the form $\widetilde{x}zf_{\lambda,v'}$ where $z\in \ubd$ and $v'\in S(\lambda)_{<\wp(\lambda)}$. Observe that $f_{\lambda,v_{\wp}}|_{\Xd}=q^{[w_0(\lambda)]}|_{\Xd}$ and $f_{\lambda,v'}|_{\Xd}\equiv 0$ for any $v'\in S(\lambda)_{<\wp(\lambda)}$. It follows that these representatives belong to the sum of the subspaces (3) and (4) from Definition \ref{finaltwomodVdef}, and hence $\widetilde{\Phi}_1(x\cdot a)=0$ as desired.
\end{proof}

\begin{lemma}\label{finalphi1graded} $\Phi_1$ is graded. 
\end{lemma}
\begin{proof}
This is because $\Phi_1$ is induced by $\id_{\dm{\Gd}\ot\OO(\ZL)}$.
\end{proof}

\begin{lemma} \label{finalphi1shift}  For any $\lambda\in\cop$, we have
\[ \Phi_1\circ \shift^{\WW}_{\lambda} = \shift^{\VV}_{\lambda}\circ\Phi_1\]
where $\shift^{\VV}_{\lambda}$ and $\shift^{\WW}_{\lambda}$ come from Lemma \ref{finaltwomodvextend} and Lemma \ref{finaltwomodwextend} respectively.
\end{lemma} 
\begin{proof}
This follows immediately from the definitions of $\shift^{\VV}_{\lambda}$, $\shift^{\WW}_{\lambda}$ and $\Phi_1$.
\end{proof}

\begin{proposition}\label{finalphi1phi1surj}
$\Phi_1$ is surjective.
\end{proposition}
\begin{proof}
Consider the operators $\shift^{\VV}_{\lambda}$ and $\shift^{\WW}_{\lambda}$ ($\lambda\in\cop$) from Lemma \ref{finaltwomodvextend} and Lemma \ref{finaltwomodwextend} respectively. They are invertible (Lemma \ref{finalVshiftinvertible} and Lemma \ref{finalWshiftinvertible}) and satisfy $\Phi_1\circ \shift^{\WW}_{\lambda} = \shift^{\VV}_{\lambda}\circ\Phi_1$ (Lemma \ref{finalphi1shift}). Hence it suffices to show that for any $x\in\dm{\Gd}\ot\OO(\ZL)/\VV$, there exists $\lambda\in\cop$ and $y\in\dkf(\KK)\ot\OO(\ZL)/\WW$ such that $\shift^{\VV}_{\lambda}(x)=\Phi_1(y)$. Thanks to the ideals (1) and (3) from Definition \ref{finaltwomodVdef} and the $\utd$-linearity of $\Phi_1$, we may assume $x$ is represented by some $\vp\in\OO(\Gd\times\ZL)$.

Let $\lambda_0\in\cop$ be regular dominant. The vectors $v_e, v_{\wp}\in S(\lambda_0)$ give rise to two sections $s_1$, $s_2$ of the line bundle $\LL(\lambda_0):=\Gd\times^{\Bmd}\CC_{\lambda_0}$ on $\GB$ respectively. Denote by $D_1$ and $D_2$ the zero loci of $s_1$ and $s_2$ respectively. Define $U:=(\GB)\setminus (D_1\cup D_2)$. Observe that the morphism $\Gd\times\ZL\ra \GB\times\ZL$ defined by $(g,t)\mapsto (g^{-1}\Bmd,t)$ takes $\Xd$ isomorphically onto a closed subscheme of $U\times\ZL$. Since $\lambda_0$ is regular dominant, $\LL(\lambda_0)$ is ample, and hence every regular function on $U\times \ZL$ is an $\OO(\ZL)$-linear combination of $s/s_1^{n}s_2^{n}$ where $n\in\ZZ_{\geqslant 0}$ and $s\in H^0(\GB;\LL(2n\lambda_0))\simeq S(2n\lambda_0)$. It is not hard to see that the pull-back of $s/s_1^{n}s_2^{n}$ via the above morphism is of the form $f_{2n\lambda_0,v}/f_{\lambda_0,v_e}^{n}f_{\lambda_0,v_{\wp}}^{n}$ where $v\in S(2n\lambda_0)$. Therefore, we can find $n, N\in\ZZ_{\geqslant 0}$, $c_1,\ldots,c_N\in\OO(\ZL)$ and $v_1,\ldots,v_N\in S(2n\lambda_0)$ such that 
\begin{equation}\label{finalphi1phi1surjeq1}\nonumber
 f_{\lambda_0,v_e}^{n}f_{\lambda_0,v_{\wp}}^{n}\vp|_{\Xd} = \sum_{i=1}^N c_i f_{2n\lambda_0,v_i}|_{\Xd}. 
\end{equation}
Notice that $ f_{\lambda_0,v_e}^{n}f_{\lambda_0,v_{\wp}}^{n}\vp$ represents $\shift^{\VV}_{n\lambda_0}( q^{n[w_0(\lambda_0)]}x)$ and each $f_{2n\lambda_0,v_i}$ represents an element of $\dkf(\KK)$, and hence $ \sum_{i=1}^N c_i f_{2n\lambda_0,v_i}$ represents $\Phi_1(y)$ where $y:=\sum_{i=1}^N \left[ [f_{2n\lambda_0,v_i}]\ot c_i\right]\in\dkf(\KK)\ot\OO(\ZL)/\WW$. The proof is complete.
\end{proof}

\begin{proposition}\label{finalphi1phi1limitisbij}
$\FF_{\hp=0}(\Phi_1)$ is bijective
\end{proposition}
\begin{proof}
This follows from Lemma \ref{dunnowheretoplacelemma1} below. 
\end{proof}

\begin{lemma}\label{dunnowheretoplacelemma1}
There are canonical isomorphisms of $\sym^{\bl}(\ttd)\ot\OO(\ZL)$-modules
\[ \FF_{\hp=0}(\dkf(\KK)\ot\OO(\ZL)/\WW)\xrightarrow{\sim}\OO(\ZZZ_P)\xrightarrow{\sim}\FF_{\hp=0}(\dm{\Gd}\ot\OO(\ZL)/\VV)\] 
whose composition is equal to $\FF_{\hp=0}(\Phi_1)$, where 
\[\ZZZ_P:=\left\{ (b,\xi,t)\in \Bd\times (e+\ttd)\times \ZL \left|~ b\cdot\xi = \xi,~ b\in \Umd(\wpd)^{-1}t\Umd \right. \right\}.\] 
Here, $e$ is defined in \eqref{betaande} and the $\sym^{\bl}(\ttd)$-module structure on $\OO(\ZZZ_P)$ is defined as in Lemma \ref{dunnowheretoplacelemmaa}.
\end{lemma}
\begin{proof}
Since $\FF_{\hp=0}$ is right exact, $\FF_{\hp=0}(\dm{\Gd}\ot\OO(\ZL)/\VV )$ is isomorphic to the quotient of $\FF_{\hp=0}(\dm{\Gd}\ot\OO(\ZL))\simeq \OO(T^*\Gd\times\ZL)$ by a subspace which is the sum of the subspaces obtained by applying $\FF_{\hp=0}$ to (1) to (4) from Definition \ref{finaltwomodVdef}. It is not difficult to see that this subspace is in fact an ideal so that $\FF_{\hp=0}(\dm{\Gd}\ot\OO(\ZL)/\VV )$ is a ring, and 
\begin{align}\label{dunnowheretoplacelemma1eq1}
&\spec \FF_{\hp=0}(\dm{\Gd}\ot\OO(\ZL)/\VV ) \nonumber \\
\simeq ~& \left\{ (g,\xi,t)\in \Gd\times \ggd\times \ZL \left|~ \xi\in e+\ttd,~ g\cdot\xi\in e+\bbd_-,~ (g,t)\in \Xd \right. \right\}
\end{align}
as schemes. Here, as in the proof of Lemma \ref{dunnowheretoplacelemmaa}, we have identified $T^*\Gd$, first with $\Gd\times(\ggd)^*$ via left translations, and then with $\Gd\times\ggd$ via $\b$. By Kostant's slice theorem, we have
\[ b\cdot \xi\in e+\bbd_-\Longleftrightarrow b\cdot\xi=\xi\]
for any $b\in\Bd$ and $\xi\in e+\ttd$. Therefore, the RHS of \eqref{dunnowheretoplacelemma1eq1} is isomorphic to $\ZZZ_P$.

Again, by the right exactness of $\FF_{\hp=0}$, $\FF_{\hp=0}(\dkf(\KK)\ot\OO(\ZL)/\WW)$ is isomorphic to the quotient of $\FF_{\hp=0}(\dkf(\KK)\ot\OO(\ZL))$ by the ideal generated by $a_{\lambda,v_{\wp}}-q^{[w_0(\lambda)]}$ ($\lambda\in\cop$) and $a_{\lambda,v}$ ($\lambda\in\cop$, $v\in S(\lambda)_{<\wp(\lambda)}$). (More precisely, the elements of $\FF_{\hp=0}(\dkf(\KK)\ot\OO(\ZL))$ they represent.) By Lemma \ref{finaltwomodarmk} and an induction argument, this ideal is equal to the ideal generated by $f_{\lambda,v_{\wp}}-q^{[w_0(\lambda)]}$ ($\lambda\in\cop$) and $f_{\lambda,v}$ ($\lambda\in\cop$, $v\in S(\lambda)_{<\wp(\lambda)}$). By Lemma \ref{dunnowheretoplacelemmaa},  
\begin{equation}\label{dunnowheretoplacelemma1eq2}
\spec \FF_{\hp=0}(\dkf(\KK))\simeq \ZZZ=\left\{ (b,\xi)\in \Bd\times (e+\ttd) \left|~ b\cdot\xi=\xi \right. \right\}.
\end{equation}
Notice that $\ZL$ is a subtorus of $\Td$ and $q^{[w_Pw_0(\lambda)]}=q^{[w_0(\lambda)]}\in\OO(\ZL)$ for any $\lambda\in\co$. It follows that, by \eqref{dunnowheretoplacelemma1eq2} and Lemma \ref{dunnowheretoplacelemmab},
\begin{align*}
& \FF_{\hp=0}(\dkf(\KK)\ot\OO(\ZL)/\WW)\\
 \simeq~& \left\{ (b,\xi,t)\in \Bd\times (e+\ttd)\times \ZL \left|~ b\cdot\xi = \xi,~ b\in \Umd(\wpd)^{-1}t\Umd \right. \right\} = \ZZZ_P.
\end{align*}

This gives us the desired isomorphisms. The last assertion that their composition is equal to $\FF_{\hp=0}(\Phi_1)$ follows from the observation that all the maps involved are induced by the identity map of $\OO(T^*\Gd\times\ZL)$.
\end{proof}

%%%%%%%%%%%%%%%%%%%%%%%%%%%%
%%%%%%%%%%%%%%%%%%%%%%%%%%%%
\subsection{The map $\Phi_2$}\label{finalphi2}
Recall the $\utd$-linear maps
\[ \pgmpt: H^{\That}_{-\bl}(\ag) \ra QH_{\That}^{\bl}(G/P)[\ecc]\]
and 
\[ \pbfd: \dkf(\KK)\ra H^{\That}_{-\bl}(\ag) \]
defined in Definition \ref{GMPcomputedef} and Definition \ref{BFpbfdef} respectively.

\begin{definition}\label{finalphi2phi2predef}$~$
\begin{enumerate}
\item Define 
\[ \psgmpt:= (\pgmpt)_{\CC[\eccc] }: H^{\That}_{-\bl}(\ag)[\eccc] \ra QH_{\That}^{\bl}(G/P)[\ecc]\]
to be the map obtained from $\pgmpt$ by extension of scalars and
\[ \psbfd:= \pbfd\ot\mirror:\dkf(\KK)\ot\OO(\ZL)\ra H^{\That}_{-\bl}(\ag)[\eccc] \]
where $\mirror: \OO(\ZL)\xrightarrow{\sim}\CC[\eccc]$ is defined in \eqref{mirrormapdef}. 

\item Define $\psgmpg$ and $\psbfk$ similarly. 
\end{enumerate}
\end{definition}

\begin{definition}\label{finalphi2phi2def}
Define a $\utd\ot\OO(\ZL)$-linear map
\[ \widetilde{\Phi}_2:= \psgmpt\circ\psbfd:\dkf(\KK)\ot\OO(\ZL) \ra QH_{\That}^{\bl}(G/P)[\ecc].\]
\end{definition}
\begin{lemma}\label{finalphi2phi2ok}
$\widetilde{\Phi}_2$ descends to a $\utd\ot\OO(\ZL)$-linear map
\[ \Phi_2:\dkf(\KK)\ot\OO(\ZL)/\WW \ra QH_{\That}^{\bl}(G/P)[\ecc].\]
\end{lemma} 
\begin{proof}
Let $x\in\dkf(\KK)\ot\OO(\ZL)$ and $a\in\kf(\KK)\ot\OO(\ZL)$. Since $\pbfd$ is a homomorphism of modules with respect to the ring isomorphism $\pbfk$ (Theorem \ref{BFqttheorem}) and $\gmpg$ is a module action (Proposition \ref{GMPactionmodule}), we have 
\[  \widetilde{\Phi}_2(x\cdot a) = \psgmpt\left( \psbfd(x)\aggt \psbfk(a)\right) = \psbfd(x)\gmpt\left( \psgmpg(\psbfk(a))\right).\]
Therefore, it suffices to show 
\begin{equation}\label{finalphi2phi2okeq1}
\psgmpg(\psbfk(a)) =0
\end{equation}
if $a=a_{\lambda,v_{\wp}}-q^{[w_0(\lambda)]}$ with $\lambda\in\cop$ or $a=a_{\lambda,v}$ with 
$\lambda\in\cop$ and $v\in S(\lambda)_{<\wp(\lambda)}$. (See Definition \ref{finaltwomodWdef}.)

By Proposition \ref{BFenhancedivisor}, $\pbfd([f_{\lambda,v_{\wp}}])$ is equal to $[Z]$ where $Z$ is the unique MV cycle of type $\lambda$ and weight $\wp(\lambda)$. Since $[Z]$ has degree $2\rho(\lambda-\wp(\lambda))$, it follows that, by Proposition \ref{GMPcompute1}, $\pgmpt([Z])=c q^{[w_0(\lambda)]}$ for some $c\in\CC$. To determine $c$, it suffices to look at $\FF_{\hp=0}(\pgmpt)$ which is Peterson-Lam-Shimozono's homomorphism $\Phi_{PLS}$ (see Proposition \ref{GMPPLS}). Hence, by \cite[Lemma 6.3]{me2}, we have $c=1$. It follows that 
\begin{equation}\label{finalphi2phi2okeqa}
\psgmpt(\psbfd([f_{\lambda,v_{\wp}}]-q^{[w_0(\lambda)]}))=0.
\end{equation}

On the other hand, for any $v\in S(\lambda)_{<\wp(\lambda)}$, we have $\pbfd([f_{\lambda,v}])\in H^{\That}_{>2\rho(\lambda-\wp(\lambda))}(\agll)$, and hence, by Proposition \ref{GMPcompute0}, we have
\begin{equation}\label{finalphi2phi2okeqb}
\psgmpt(\psbfd([f_{\lambda,v}]))=0.
\end{equation}

Now, \eqref{finalphi2phi2okeq1} follows from \eqref{finalphi2phi2okeqa}, \eqref{finalphi2phi2okeqb} and the fact that $\psbfd$ and $\psgmpt$ are induced by $\psbfk$ and $\psgmpg$ respectively. (Recall from Definition \ref{finaltwomodadef} that $a_{\lambda,v}$ is characterized by the equality $[f_{\lambda,v}]=a_{\lambda,v} + \sum_{i=2}^{|W|} z_i\cdot\kf(f_{\lambda})(\widetilde{a}_{\lambda,v,i})$.)
\end{proof}

\begin{lemma}\label{finalphi2graded} $\Phi_2$ is graded.
\end{lemma}
\begin{proof}
This is because $\pgmpt$ and $\pbfd$ are graded (Lemma \ref{GMPactiongraded} and Theorem \ref{BFqttheorem}(2)). 
\end{proof}

\begin{lemma}\label{finalphi2shift} For any $\lambda\in\cop$, we have
\[ \Phi_2\circ \shift^{\WW}_{\lambda} = \shift^A_{\lambda}\circ\Phi_2\]
where $\shift^A_{\lambda}$ and $\shift^{\WW}_{\lambda}$ come from Lemma \ref{Ashiftwell} and Lemma \ref{finaltwomodwextend} respectively.
\end{lemma}
\begin{proof}
It suffices to show $\widetilde{\Phi}_2\circ  \widetilde{\shift}^{\WW}_{\lambda} = \shift^A_{\lambda}\circ  \widetilde{\Phi}_2.$ From the second paragraph of the proof of Lemma \ref{finalWshiftinvertible}, we have
\[ \psbfd\circ \widetilde{\shift}^{\WW}_{\lambda}= \left( ([t^{\lambda}] \aggt - ) \ot\id_{\CC[\eccc]}\right) \circ\psbfd. \]
Hence it suffices to show 
\begin{equation}\label{finalphi2shifteq1}
\pgmpt\circ\left([t^{\lambda}] \aggt -  \right) = \shift^A_{\lambda}\circ \pgmpt. 
\end{equation}
Since both sides of \eqref{finalphi2shifteq1} are $(-\lambda)$-twisted $H_{\That}^{\bl}(\pt)$-linear, we may assume the given input element belongs to $H_{-\bl}^{\Ghat}(\ag)$. In this case, we apply Proposition \ref{GMPactionmodule} to get the result.
\end{proof}

\begin{proposition}\label{finalphi2phi2surj}
$\Phi_2$ is surjective.
\end{proposition}
\begin{proof}
This follows from the facts that $\psgmpt$ is surjective (Corollary \ref{GMPPLScor}) and $\psbfd$ is surjective (Theorem \ref{BFqttheorem}).
\end{proof}

\begin{proposition}\label{finalphi2phi2limitisbij}
$\FF_{\hp=0}(\Phi_2)$ is bijective.
\end{proposition}
\begin{proof}
Recall the following
\begin{enumerate}
\item $\ZZZ:=\{(b,\xi)\in\Bd\times (e+\ttd)|~b\cdot\xi=\xi\}$ (see Lemma \ref{dunnowheretoplacelemmaa});

\item $\ZZZ_P:=\left\{ (b,\xi,t)\in \Bd\times (e+\ttd)\times \ZL \left|~ b\cdot\xi = \xi,~ b\in \Umd(\wpd)^{-1}t\Umd \right. \right\}$ (see Lemma \ref{dunnowheretoplacelemma1});

\item $\Bed:=\{(b,h)\in\Bd\times\spec H_T^{\bl}(\pt)|~b\cdot e^T(h)=e^T(h)\}$ (see \eqref{Beddef});

\item $\Phi_{YZ}:\OO(\Bed)\xrightarrow{\sim}H_{-\bl}^T(\ag)$, Yun-Zhu's isomorphism \cite{YZ} (see Proposition \ref{BFenhanceequalYZ}); and

\item $\Phi_{PLS}:H_{-\bl}^T(\ag)\ra QH_T^{\bl}(G/P)[\ecc]$, Peterson-Lam-Shimozono's homomorphism \cite{me1, LS, Peter} (see Proposition \ref{GMPPLS}).
\end{enumerate}
For simplicity, put 
\[ N_1:= H_{-\bl}^T(\ag)[\eccc]\quad\text{ and }\quad N_2:= QH_T^{\bl}(G/P)[\ecc].\]
We have the following commutative diagram
\begin{equation}\label{finalphi2phi2limitisbijdiag}
\begin{tikzpicture}
\tikzmath{\x1 = 5; \x2 = 2;}
\node (A) at (-0.1*\x1,0) {$\OO(\ZZZ_P)$} ;
\node (B) at (\x1,0) {$\OO(\ZZZ\times\ZL)$} ;
\node (C) at (1.7*\x1,0) {$N_1$} ;
\node (D) at (2.4*\x1,0) {$N_2$} ;
\node (E) at (-0.1*\x1,-\x2) {$\OO(\Bed\times_{\Gd}\Umd(\wpd)^{-1}\Bmd)$} ;
\node (F) at (\x1,-\x2) {$\OO(\Bed\times\ZL)$} ;
\node (G) at (1.7*\x1,-\x2) {$N_1$} ;
\node (H) at (2.4*\x1,-\x2) {$N_2$} ;

\draw[->>, font=\tiny] (B) edge node[above]{$\eta_1$} (A);
\path[->,  font=\tiny] (B) edge node[above]{$\FF_{\hp=0}(\psbfd)$} (C);
\path[->,  font=\tiny] (C) edge node[above]{$\FF_{\hp=0}(\psgmpt)$} (D);
\draw[->>, font=\tiny] (F) edge node[above]{$\eta_4$} (E);
\path[->,  font=\tiny] (F) edge node[above]{$\Phi_{YZ}\ot\mirror$} (G);
\path[->,  font=\tiny] (G) edge node[above]{$(\Phi_{PLS})_{\CC[\eccc]}$} (H);

\path[->,  font=\tiny] (A) edge node[left]{$\eta_2$} (E);
\path[->,  font=\tiny] (B) edge node[left]{$\eta_3$} (F);

\path[->] (A) edge node[right]{$\simeq$} (E);
\path[->] (B) edge node[right]{$\simeq$} (F);
\draw[double equal sign distance] (C) -- (G);
\draw[double equal sign distance] (D) -- (H);
\end{tikzpicture}.
\end{equation}
Here, 
\begin{enumerate}[(i)]
\item the leftmost commutative square is induced by the commutative diagram
\begin{equation}\nonumber
\begin{tikzpicture}
\tikzmath{\x1 = 7; \x2 = 2;}
\node (A) at (0,0) {$\ZZZ_P$} ;
\node (B) at (\x1,0) {$\ZZZ\times\ZL$} ;
\node (C) at (0,-\x2) {$\Bed\times_{\Gd}\Umd(\wpd)^{-1}\Bmd$} ;
\node (D) at (\x1,-\x2) {$\Bed\times\ZL$} ;

\draw[right hook->, font=\tiny]  (A) edge node[above]{canonical} (B);
\path[->, font=\tiny] (C) edge node[left]{induced by right arrow} (A);
\path[->, font=\tiny] (C) edge node[right]{$\simeq$} (A);
\path[->, font=\tiny] (D) edge node[right]{$(b,h,t)\mapsto (b,e^T(h),t)$} (B);
\path[->, font=\tiny] (D) edge node[left]{$\simeq$} (B);
\draw[right hook->, font=\tiny] 
(C) edge node[above]{$(b,h)\mapsto (b,h,t(b))$} (D);
\end{tikzpicture}
\end{equation}
where $t(b)\in\Td$ is the unique element such that
\begin{equation}\label{finalphi2phi2limitisbijeqq1}
b\in \Umd (\wpd)^{-1}t(b)\Umd
\end{equation}
(by the proof of \cite[Lemma 6.10]{me2}, $t(b)$ in fact lies in $\ZL$);

\item the middle commutative square is given by Lemma \ref{dunnowheretoplacelemmaa} and Proposition \ref{BFenhanceequalYZ}; and

\item the rightmost commutative square is given by Proposition \ref{GMPPLS}.
\end{enumerate}

By Lemma \ref{dunnowheretoplacelemma1}, we have
\begin{equation}\label{finalphi2phi2limitisbijeq1}
\FF_{\hp=0}(\Phi_2)\circ\eta_1 = \FF_{\hp=0}(\psgmpt)\circ\FF_{\hp=0}(\psbfd).
\end{equation}
We wish to show 
\begin{equation}\label{finalphi2phi2limitisbijeq2}
 \Phi_{loc}\circ \eta_4= (\Phi_{PLS})_{\CC[\eccc]} \circ (\Phi_{YZ}\ot\mirror)
\end{equation}
where $\Phi_{loc}$ is the algebra isomorphism from \cite[Theorem B]{me2}. This will complete the proof of Proposition \ref{finalphi2phi2limitisbij} because we will have $\FF_{\hp=0}(\Phi_2)=\Phi_{loc}\circ\eta_2$ by \eqref{finalphi2phi2limitisbijdiag} and \eqref{finalphi2phi2limitisbijeq1}. By the definition of $\Phi_{loc}$ (see \cite[Definition 6.7]{me2}), we have $\Phi_{loc}\circ\eta_4|_{\OO(\Bed)}=\Phi_{PLS}\circ\Phi_{YZ}$. Hence it remains to show that the restrictions to $\OO(\ZL)$ of both sides of \eqref{finalphi2phi2limitisbijeq2} are equal. Let $\lambda\in\Q$ be anti-dominant. By definition, $\eta_4$ sends $\lambda|_{\ZL}\in \OO(\ZL)$ to the regular function $(b,h)\mapsto \lambda(t(b))$. Consider the regular function $g_{P,w_0(\lambda)}\in\OO(\Bed)$ defined by $(b,h)\mapsto \langle v^*_{w_0(\lambda)}, b\cdot v_{\wp}\rangle$. By \eqref{finalphi2phi2limitisbijeqq1} and \cite[Proposition B.7]{me2}, $\eta_4(g_{P,w_0(\lambda)})$ is the regular function $(b,h)\mapsto \wp(w_0(\lambda))(t(b))=w_P(\lambda)(t(b))$. Since $t(b)\in\ZL$, we have $w_P(\lambda)(t(b))=\lambda(t(b))$, and hence $\eta_4(g_{P,w_0(\lambda)})=\eta_4(\lambda|_{\ZL})$. By \cite[Lemma 6.3]{me2}, $\Phi_{loc}\circ\eta_4(g_{P,w_0(\lambda)})=q^{[\lambda]}$. Therefore,
\[  \Phi_{loc}\circ\eta_4(\lambda|_{\ZL}) = \Phi_{loc}\circ\eta_4(g_{P,w_0(\lambda)})=q^{[\lambda]} = (\Phi_{PLS})_{\CC[\eccc]}\circ (\Phi_{YZ}\ot\mirror)( \lambda|_{\ZL}). \]
Since $\lambda|_{\ZL}$ ($\lambda\in\Q$ anti-dominant) and their inverses generate the $\CC$-algebra $\OO(\ZL)$, the result follows.
\end{proof}

\begin{remark}\label{addrmk1}
From the proof of Proposition \ref{finalphi2phi2limitisbij} (more precisely, the equality $\FF_{\hp=0}(\Phi_2)=\Phi_{loc}\circ\eta_2$), we see that $\FF_{\hp=0}(\Phi_2)$ is a ring map where the ring structure on $\FF_{\hp=0}(\dkf(\KK)\ot\OO(\ZL)/\WW)$ is induced by $\OO(\ZZZ_P)$ via Lemma \ref{dunnowheretoplacelemma1}. We will use this fact when we prove that $\FF_{\hp=0}(\Mir)$ is a ring map in Section \ref{conclusionofproof}.
\end{remark}

%%%%%%%%%%%%%%%%%%%%%%%%%%%%%%%%%%
%%%%%%%%%%%%%%%%%%%%%%%%%%%%%%%%%%
\subsection{The map $\Phi_3$}\label{finalphi3} 
\begin{proposition}\label{finalphi3exist}
There exists a unique $\utd\ot\OO(\ZL)$-linear map $\Phi_3$ such that $\Phi_3\circ\Phi_1=\Phi_2$.
\end{proposition}
\begin{proof}
Since $\Phi_1$ is surjective (Proposition \ref{finalphi1phi1surj}), the existence of $\Phi_3$ is equivalent to $\ker\Phi_1\subseteq\ker\Phi_2$. Let $y\in\ker\Phi_1$. Since $\FF_{\hp=0}(\Phi_1)$ is bijective (Proposition \ref{finalphi1phi1limitisbij}), we have $y=\hp y_1$ for some $y_1\in\dkf(\KK)\ot\OO(\ZL)/\WW$. Then $0=\Phi_1(y)=\hp\Phi_1(y_1)$. By Proposition \ref{hTF} and Proposition \ref{finalphi0bijective}, $\dm{\Gd}\ot\OO(\ZL)/\VV$ is $\hp$-torsion-free. Hence we have $y_1\in\ker\Phi_1$. Continuing, we obtain $y_1,y_2,y_3,\ldots\in\ker\Phi_1$ satisfying $y=\hp^k y_k$ for any $k$. It follows that $\Phi_2(y)=\hp^k\Phi_2(y_k)$ for any $k$, and hence $\Phi_2(y)=0$ because $QH_{\That}^{\bl}(G/P)[\ecc]$ is $\CC[\hp]$-free. This establishes the existence of $\Phi_3$. The uniqueness follows from the surjectivity of $\Phi_1$. 
\end{proof}

\begin{lemma}\label{finalphi3graded}
$\Phi_3$ is graded.
\end{lemma}
\begin{proof}
This is because $\Phi_1$ and $\Phi_2$ are graded (Lemma \ref{finalphi1graded} and Lemma \ref{finalphi2graded}), $\Phi_1$ is surjective (Proposition \ref{finalphi1phi1surj}) and $\Phi_3\circ \Phi_1=\Phi_2$ (Proposition \ref{finalphi3exist}).
\end{proof}

\begin{lemma} \label{finalphi3shift}
For any $\lambda\in\cop$, we have
\[ \Phi_3\circ\shift^{\VV}_{\lambda} = \shift^A_{\lambda}\circ \Phi_3\]
where $\shift^A_{\lambda}$ and $\shift^{\VV}_{\lambda}$ come from Lemma \ref{Ashiftwell} and Lemma \ref{finaltwomodvextend} respectively.
\end{lemma}
\begin{proof}
This follows from the analogous equalities for $\Phi_1$ and $\Phi_2$ (Lemma \ref{finalphi1shift} and Lemma \ref{finalphi2shift}), the surjectivity of $\Phi_1$ (Proposition \ref{finalphi1phi1surj}) and the equality $\Phi_3\circ \Phi_1=\Phi_2$ (Proposition \ref{finalphi3exist}).
\end{proof}

\begin{proposition}\label{finalphi3bij}
$\Phi_3$ is bijective.
\end{proposition}
\begin{proof}
The surjectivity of $\Phi_3$ follows from that of $\Phi_2$ (Proposition \ref{finalphi2phi2surj}) and the equality $\Phi_3\circ\Phi_1=\Phi_2$ (Proposition \ref{finalphi3exist}). It remains to show that $\Phi_3$ is injective. 

Since $\Phi_3$ is surjective and $QH_{\That}^{\bl}(G/P)[\ecc]$ is a free $H_{\That}^{\bl}(\pt)[\eccc]$-module, we have a (non-canonical) splitting
\[ \dm{\Gd}\ot\OO(\ZL)/\VV\simeq\ker\Phi_3\oplus QH_{\That}^{\bl}(G/P)[\ecc].\]
By Proposition \ref{coherent} and Proposition \ref{finalphi0bijective}, $\dm{\Gd}\ot\OO(\ZL)/\VV$ is a coherent $\OO_{\AA^1_{\hp}\times\hh\times\ZL}$-module, and hence $\ker\Phi_3$ is also coherent. It follows that $\ker\Phi_3=0$, i.e. $\Phi_3$ is injective, if we can show that the fiber dimension of $\dm{\Gd}\ot\OO(\ZL)/\VV$ at every point $(\hp,h,t)\in \AA^1_{\hp}\times\hh\times\ZL$ is at most $|W/W_P|$, the rank of $QH_{\That}^{\bl}(G/P)[\ecc]$. 

Since $\Phi_0$ is bijective (Proposition \ref{finalphi0bijective}), it suffices to look at $\briesp$. Denote the fiber of $\bries$ at $(\hp,h,t)$ by $G_0(\hp,h,t)$. By Proposition \ref{Rietschfiberlemma}, we have $\dim  G_0(\hp,h,t)=\dim G_0(c\hp,ch,t)$ for any $(\hp,h,t)\in \AA^1_{\hp}\times\hh\times\ZL$ and $c\in\CC$ with $c,\hp\ne 0$. This equality together with the coherence of $\bries\simeq \dm{\Gd}\ot\OO(\ZL)/\VV$ implies $\dim G_0(\hp,h,t)\leqslant\dim G_0(0,0,t)$. But since $\FF_{\hp=0}(\Phi_1)$ and $\FF_{\hp=0}(\Phi_2)$ are bijective (Proposition \ref{finalphi1phi1limitisbij} and Proposition \ref{finalphi2phi2limitisbij}), we have 
\begin{equation}\label{finalphi3bijeq1}
\dim G_0(0,0,t)=\dim G_0(0,h,t)= |W/W_P|,
\end{equation}
and hence we have $\dim G_0(\hp,h,t)\leqslant |W/W_P|$ for any $(\hp,h,t)$ with $\hp\ne 0$. It remains to deal with the case $\hp=0$. But this also follows from \eqref{finalphi3bijeq1}.
\end{proof}

%%%%%%%%%%%%%%%%%%%%%%%%%%%%%%%%%%
%%%%%%%%%%%%%%%%%%%%%%%%%%%%%%%%%%
\subsection{The map $\Mir$}\label{finalfinal}
\begin{definition}\label{finalfinalphidef}
Define
\[\Mir:=\Phi_3\circ\Phi_0: \bries\ra QH_{\That}^{\bl}(G/P)[\ecc]\]
where $\Phi_0$ and $\Phi_3$ come from Lemma \ref{finalphi0descendlemma} and Proposition \ref{finalphi3exist} respectively.
\end{definition}

\begin{proposition}\label{finalfinaldmlinear}
$\Mir$ is $\dm{\ZL}$-linear.
\end{proposition}

We will prove Proposition \ref{finalfinaldmlinear} after some preparation. Take $\lambda_0\in\cop$ such that $\a_i(w_0(\lambda_0))$ is non-zero precisely when $i\in\iip$. Let $i\in\iip$. There is a unique MV cycle of type $\lambda_0$ and weight $w_0(\lambda_0)+\ad_i$. In fact, by \cite[Lemma A.1]{me2}, this MV cycle is equal to $\ol{\BB\cdot t^{s_{\a_i}(w_0(\lambda_0))}}$. Let $v_i\in S(\lambda_0)_{w_0(\lambda_0)+\ad_i}$ correspond to this MV cycle via the composite isomorphism \eqref{MVbasis}. Define $\ffff_i:=f_{\lambda_0,v_i}$. Define also $\ffff_0:=f_{\lambda_0,v_{\wp}}=f_{\lambda_0,v_{w_0}}$. (For the definition of $f_{\lambda,v}$, see the paragraph before Proposition \ref{BFenhancedivisor}. For the definition of $v_w$, see the paragraph before Definition \ref{finaltwomodWdef}.)

Define
\[ \widetilde{A}_i:\dm{\Gd}\ot\OO(\ZL)\ra \dm{\Gd}\ot\OO(\ZL)\]
by
\[ \widetilde{A}_i(x q^{[\mu]}):=x(\ffff_i-(\w_i^L-\hp\rhod(\w_i))\ffff_0) q^{[\mu-w_0(\lambda_0)-\ad_i]}+\hp x\partial_{q_i}q^{[\mu]}\]
for any $x\in\dm{\Gd}$ and $q^{[\mu]}\in\OO(\ZL)$, where $\partial_{q_i}q^{[\mu]}:=\w_i(\mu) q^{[\mu-\ad_i]}$.

\begin{lemma}\label{finalfinalAipreserveV}
$\widetilde{A}_i$ preserves $\VV$.
\end{lemma}
\begin{proof}
It is clear that $\widetilde{A}_i$ preserves the ideal (3) from Definition \ref{finaltwomodVdef}. By Lemma \ref{finalfinaladdlemma}, both $\ffff_i-\w_i^L \ffff_0$ and $\ffff_0$ represent elements of $\kf(\KK)$, and so does $\ffff_i-(\w_i^L-\hp\rhod(\w_i))\ffff_0$. It follows that $\widetilde{A}_i$ maps the ideals (1) and (2) into $\VV$. It remains to handle the subspace (4). Let $\hey\in\utd$ and $\widetilde{\vp}\in\OO(\Gd\times\ZL)$ such that $\widetilde{\vp}|_{\Xd}\equiv 0$. We have 
\begin{align*}
 \widetilde{A}_i(\hey\widetilde{\vp}) = ~&\hey\widetilde{\vp}(\ffff_i-(\w_i^L-\hp\rhod(\w_i))\ffff_0)q^{-[w_0(\lambda_0)+\ad_i]}+ \hp\hey\partial_{q_i}\widetilde{\vp}\\
  =~&\hey\widetilde{\vp} \ffff_i q^{-[w_0(\lambda_0)+\ad_i]} - \w_i^L\hey \widetilde{\vp} \ffff_0 q^{-[w_0(\lambda_0)+\ad_i]} + \hp\rhod(\w_i) \hey\widetilde{\vp} \ffff_0  q^{-[w_0(\lambda_0)+\ad_i]}\\
 & +\hp\hey(\LL_{\w_i^L}\widetilde{\vp})\ffff_0 q^{-[w_0(\lambda_0)+\ad_i]} + \hp\hey\partial_{q_i}\widetilde{\vp}.
\end{align*}
Clearly the first three terms in the last expression lie in $\VV$. Moreover, we have
\begin{equation}\label{finalfinalAipreserveVeq1}
\ffff_0q^{-[w_0(\lambda_0)+\ad_i]}|_{\Xd}\equiv q^{-[\ad_i]}|_{\Xd}. 
\end{equation}
(We apply the equality $(\wpd)^{-1}\cdot v_{\wp}=v_e$ which is \cite[Proposition B.7]{me2}.) Therefore, 
\[  \widetilde{A}_i(\hey\widetilde{\vp})\equiv \hp\hey\LL_{q^{-[\ad_i]} \w_i^L+\partial_{q_i}}\widetilde{\vp}~(\bmod{\VV}).\]
The result now follows from Lemma \ref{ApplemmaD} which says that the vector field $\widetilde{\zzeta}_i:=q^{-[\ad_i]} \w_i^L+\partial_{q_i}\in\mathfrak{X}(\Gd\times\ZL)$ is tangent to $\Xd$.
\end{proof}

\begin{definition}\label{finalfinalAidef}
Define
\[ A_i: \dm{\Gd}\ot\OO(\ZL)/\VV\ra \dm{\Gd}\ot\OO(\ZL)/\VV \]
to be the map induced by $\widetilde{A}_i$ via Lemma \ref{finalfinalAipreserveV}.
\end{definition}

Recall $\dm{\ZL}$ is the $\CC[\hp]$-algebra generated by $q_i^{\pm 1}$ and $\xi_i$ ($i\in\iip$) subject to the relations $q_iq_j=q_jq_i$, $\xi_iq_j-q_j\xi_i=\hp\delta_{ij}$ and $\xi_i\xi_j=\xi_j\xi_i$. Since $\Phi_0$ is bijective (Proposition \ref{finalphi0bijective}), there is a unique $\dm{\ZL}$-module structure on $\dm{\Gd}\ot\OO(\ZL)/\VV$ such that $\Phi_0$ is $\dm{\ZL}$-linear. 

\begin{lemma}\label{finalfinalAicoincidexi}
For any $i\in\iip$, the operator $\xi_i\cdot - \in\eendo(\dm{\Gd}\ot\OO(\ZL)/\VV)$ coincides with the operator $A_i$ from Definition \ref{finalfinalAidef}.
\end{lemma}
\begin{proof}
Define $\widetilde{\zzeta}_i:= q^{-[\ad_i]}\w_i^L+\partial_{q_i}\in\mathfrak{X}(\Gd\times\ZL)$. By Lemma \ref{ApplemmaD}, $\widetilde{\zzeta}_i$ is tangent to $\Xd$ and $\zzeta_i:=\widetilde{\zzeta}_i|_{\Xd}$ is a lift of $\partial_{q_i}\in\mathfrak{X}(\ZL)$ with respect to $\pi$. Every element of $\dm{\Gd}\ot\OO(\ZL)/\VV$ is represented by $\hey\widetilde{\vp}$ for some $\hey\in\utd$ and $\widetilde{\vp}\in \OO(\Gd\times\ZL)$. Put $\vp:=\widetilde{\vp}|_{\Xd}$. By the definitions of the $\dm{\ZL}$-module structures on $ \briesp$ and $\dm{\Gd}\ot\OO(\ZL)/\VV$, we have 
\begin{align*}
  \xi_i\cdot [\hey\widetilde{\vp}] =~&  \widetilde{\Phi}_0\left(  \twist_{-\rhod}^{-1}(\hey)\ot\left( \hp\left(\frac{\LL_{\zzeta_i}\vol}{\vol}\right)\vp + \hp\LL_{\zzeta_i}\vp+(\LL_{\zzeta_i}W)\vp\right) {\color{white} \sum_j}\right.\\
&\qquad \left. - \sum_j\twist_{-\rhod}^{-1}(\hey)h_j\ot(\iota_{\zzeta_i}p^*
\langle h^j,\mccccc_{\Td}\rangle )\vp\right) .
\end{align*}
(See Definition \ref{finalphi0def} and Definition \ref{Fiberwisevolumeformdef} for the definitions of $\widetilde{\Phi}_0$ and $\vol$ respectively.) By Lemma \ref{ApplemmaF}, $\LL_{\zzeta_i}\vol=0$. By Lemma \ref{finalfinalextend}, $q^{-[w_0(\lambda_0)+\ad_i]}\ffff_i$ is an extension of $\LL_{\zzeta_i}W$. By Lemma \ref{ApplemmaE}, $\iota_{\zzeta_i}p^*\langle h^j,\mccccc_{\Td}\rangle  = q^{-[\ad_i]}\langle h^j,\w_i\rangle$. Therefore,
\begin{align*}
&~ \xi_i\cdot [\hey\widetilde{\vp}] \\
 =  ~ & \left[\hey\left(\hp\LL_{\widetilde{\zzeta}_i}\widetilde{\vp} +q^{-[w_0(\lambda_0)+\ad_i]}\ffff_i\widetilde{\vp}\right) - \hey(\w_i^L-\hp\rhod(\w_i))q^{-[\ad_i]}\widetilde{\vp} \right]\\
 =  ~ &\left[ \hp\hey\LL_{q^{-[\ad_i]}\w_i^L+\partial_{q_i}}\widetilde{\vp} + \hey\widetilde{\vp}(\ffff_i - (\w_i^L-\hp\rhod(\w_i))\ffff_0)q^{-[w_0(\lambda_0)+\ad_i]} \right. \\ 
& - \left.\hey(\w_i^L\widetilde{\vp}-\widetilde{\vp}\w_i^L)q^{-[\ad_i]} + \hey\widetilde{\vp}(\w_i^L-\hp\rhod(\w_i))\left( \ffff_0q^{-[w_0(\lambda_0)+\ad_i]}-q^{-[\ad_i]} \right)  \right] .
\end{align*}
By \eqref{finalfinalAipreserveVeq1}, the term $\hey\widetilde{\vp}(\w_i^L-\hp\rhod(\w_i))\left( \ffff_0q^{-[w_0(\lambda_0)+\ad_i]}-q^{-[\ad_i]} \right)$ belongs to the subspace (4) from Definition \ref{finaltwomodVdef}. Moreover, we have $\hey(\w_i^L\widetilde{\vp}-\widetilde{\vp}\w_i^L)q^{-[\ad_i]}=\hp\hey\LL_{q^{-[\ad_i]} \w_i^L}\widetilde{\vp}$. Therefore,
\[\xi_i\cdot [\hey\widetilde{\vp}] =\left[  \hey\widetilde{\vp}(\ffff_i - (\w_i^L-\hp\rhod(\w_i))\ffff_0)q^{-[w_0(\lambda_0)+\ad_i]} +\hp\hey\partial_{q_i}\widetilde{\vp}   \right]=\left[ \widetilde{A}_i(\hey\widetilde{\vp})\right] = A_i([\hey \widetilde{\vp} ]) . \]
\end{proof}

\begin{myproof}{Proposition}{\ref{finalfinaldmlinear}}
By definition, $\Phi_0$ is $\dm{\ZL}$-linear. Hence it suffices to prove that $\Phi_3$ is $\dm{\ZL}$-linear. Let $x\in \dm{\Gd}\ot\OO(\ZL)/\VV$. Since $\Phi_1$ is surjective (Proposition \ref{finalphi1phi1surj}), there exists $y\in\dkf(\KK)\ot\OO(\ZL)/\WW$ such that $\Phi_1(y)=x$. Let $\widetilde{y}\in\dkf(\KK)\ot\OO(\ZL)$ be a representative of $y$. Let $i\in\iip$. Denote by $[\ffff_i - (\w_i^L-\hp\rhod(\w_i))\ffff_0]\in\kf(\KK)$ the element represented by $\ffff_i - (\w_i^L-\hp\rhod(\w_i))\ffff_0$. (See Lemma \ref{finalfinaladdlemma}.) We have
\[ A_i(x) = \widetilde{\Phi}_1\left( \widetilde{y} q^{-[w_0(\lambda_0)+\ad_i]} \cdot [\ffff_i - (\w_i^L-\hp\rhod(\w_i))\ffff_0] +\hp\partial_{q_i}\widetilde{y} \right).\]
By Lemma \ref{finalfinalAicoincidexi}, the equality $\Phi_3\circ\Phi_1=\Phi_2$ (Proposition \ref{finalphi3exist}) and the definition of $\Phi_2$ (Lemma \ref{finalphi2phi2ok}), we have
\begin{equation}\label{finalfinaldmlineareqa}
\Phi_3(\xi_i\cdot x)= \Phi_3(A_i(x))= \psgmpt\circ \psbfd\left( \widetilde{y} q^{-[w_0(\lambda_0)+\ad_i]} \cdot [\ffff_i - (\w_i^L-\hp\rhod(\w_i))\ffff_0] +\hp\partial_{q_i}\widetilde{y}\right).
\end{equation}
Since $\pbfd$ is a homomorphism of modules with respect to the ring isomorphism $\pbfk$ (Theorem \ref{BFqttheorem}) and $\gmpg$ is a module action (Proposition \ref{GMPactionmodule}), we have
\begin{align}\label{finalfinaldmlineareqb}
&  \psgmpt\circ\psbfd\left( \widetilde{y} q^{-[w_0(\lambda_0)+\ad_i]} \cdot [\ffff_i - (\w_i^L-\hp\rhod(\w_i))\ffff_0]\right) \nonumber \\
=~&  q^{-[w_0(\lambda_0)+\ad_i]}\psgmpt\left(\psbfd(\widetilde{y})\aggt \pbfk\left(\left[\ffff_i - (\w_i^L-\hp\rhod(\w_i))\ffff_0\right]\right)\right) \nonumber \\
=~&  q^{-[w_0(\lambda_0)+\ad_i]}\psbfd(\widetilde{y})\gmpt\left( \pgmpg\circ\pbfk\left(\left[\ffff_i - (\w_i^L-\hp\rhod(\w_i))\ffff_0\right] \right)\right).
\end{align}

By Proposition \ref{BFenhancedivisor} and \cite[Lemma A.1]{me2}, we have 
\[ \pbfd([\ffff_i])=\left[\ol{\BB\cdot t^{s_{\a_i}(w_0(\lambda_0))}}\right]\quad\text{ and }\quad\pbfd([\ffff_0])=\left[\ol{\BB\cdot t^{w_0(\lambda_0)}}\right] .\]
(By our assumption on $\lambda_0$, we have $\wp(\lambda_0)=w_0(\lambda_0)$.) Since $\pbfd$ is an extension of $\pbfk$ and is $\utd$-linear where the $\utd$-module structure on $\dkf(\KK)$ is given by the left multiplication of left-invariant vector fields twisted by $\twist_{-\rhod}$ (see Example \ref{BFenhanceHCeg2} and Definition \ref{BFenhanceKFdef}), it follows that 
\[ \pbfk\left(\left[\ffff_i - (\w_i^L-\hp\rhod(\w_i))\ffff_0\right]\right)=  \left[\ol{\BB\cdot t^{s_{\a_i}(w_0(\lambda_0))}}\right]-\w_i \left[\ol{\BB\cdot t^{w_0(\lambda_0)}}\right].\]
Now by Proposition \ref{GMPPLS},
\begin{equation}\label{finalfinaldmlineareqc}
\pgmpg\left(\left[\ol{\BB\cdot t^{s_{\a_i}(w_0(\lambda_0))}}\right]-\w_i \left[\ol{\BB\cdot t^{w_0(\lambda_0)}}\right]\right) = q^{[w_0(\lambda_0)]}(\s_{s_{\a_i}}-\w_i) =q^{[w_0(\lambda_0)]}c_1^{\Ghat}(L_{\w_i}). 
\end{equation}
Combining \eqref{finalfinaldmlineareqa}, \eqref{finalfinaldmlineareqb} and \eqref{finalfinaldmlineareqc} gives
\begin{equation}\label{finalfinaldmlineareq6}
\Phi_3(\xi_i\cdot x)=q^{-[w_0(\lambda_0)+\ad_i]}\psbfd(\widetilde{y})\gmpt\left(q^{[w_0(\lambda_0)]}c_1^{\Ghat}(L_{\w_i})\right) + \hp\partial_{q_i}\left( \psbfd(\widetilde{y})\gmpt 1\right).   
\end{equation}
Write $\widetilde{y}=\sum_j \widetilde{y}_j\ot q^{[\mu_j]}$ with $ \widetilde{y}_j\in\dkf(\KK)$. Put $\widetilde{Y}_j:=\pbfd(\widetilde{y}_j)$. The RHS of \eqref{finalfinaldmlineareq6} is equal to 
\begin{align*}
& \sum_j\widetilde{Y}_j\gmpt\left( \left(q^{-[\ad_i]}c_1^{\Ghat}(L_{\w_i})+\hp\partial_{q_i}\right) q^{[\mu_j]}\right)\\
=~& \sum_j \widetilde{Y}_j\gmpt \left(\qc{G}_{\partial_{q_i}}q^{[\mu_j]}\right) = \sum_j \qc{T}_{\partial_{q_i}}\left( \widetilde{Y}_j\gmpt q^{[\mu_j]}\right) = \qc{T}_{\partial_{q_i}}\left(\Phi_3(x)\right)
\end{align*}
where the second equality follows from Proposition \ref{GMPactioncommute}. The proof is complete.
\end{myproof}

%%%%%%%%%%%%%%%%%%%%%%%%%%%%%%%%%%
%%%%%%%%%%%%%%%%%%%%%%%%%%%%%%%%%%
\subsection{Conclusion of proof}\label{conclusionofproof}

\begin{myproof}{Theorem}{\ref{main}}
Define $\Mir$ to be the map from Definition \ref{finalfinalphidef}. 

\begin{enumerate}
\item \underline{\textit{$\Mir$ is bijective.}} This is because $\Phi_0$ and $\Phi_3$ are bijective (Proposition \ref{finalphi0bijective} and Proposition \ref{finalphi3bij}).
\vspace{0cm}

\item \underline{\textit{$\Mir$ is $\dm{\ZL}$-linear.}} This is Proposition \ref{finalfinaldmlinear}.
\vspace{0cm}

\item \underline{\textit{$\Mir([\vol])=1$.}} Consider the elements of $\dm{\Gd}\ot\OO(\ZL)/\VV$ and $\dkf(\KK)\ot\OO(\ZL)/\WW$ represented by $1$. By abuse of notation, we denote both of them by $[1]$. By the definition of $\Phi_0$, we have $\Phi_0([\vol])=[1]$. It is clear that $\Phi_1([1])=[1]$. By Proposition \ref{BFenhancedivisor} (applied to $f_{0,v_e}=1$) and Proposition \ref{GMPPLS}, we have $\Phi_2([1])=1$. Since $\Phi_3\circ\Phi_1=\Phi_2$ (Proposition \ref{finalphi3exist}), we have 
\[ \Mir([\vol])=\Phi_3\circ\Phi_0([\vol]) = \Phi_3([1]) = \Phi_3\circ\Phi_1([1]) =\Phi_2([1]) =1.\]
\vspace{-.6cm}

\item \underline{\textit{$\FF_{\hp=0}(\Mir)$ is a ring isomorphism.}} Identify $\FF_{\hp=0}(\bries)$ with $\jac(\Xd,W,\pi,p)$ via Lemma \ref{RietschJacobilemma} where the fiberwise volume form is taken to be $\vol$ from Definition \ref{Fiberwisevolumeformdef}. By Lemma \ref{dunnowheretoplacelemma1}, both $\FF_{\hp=0}(\dm{\Gd}\ot\OO(\ZL)/\VV)$ and $\FF_{\hp=0}(\dkf(\KK)\ot\OO(\ZL)/\WW)$ are naturally rings. It is clear that $\FF_{\hp=0}(\Phi_0)$ and $\FF_{\hp=0}(\Phi_1)$ are ring maps. By Remark \ref{addrmk1}, $\FF_{\hp=0}(\Phi_2)$ is also a ring map. Since $\Phi_3\circ\Phi_1=\Phi_2$ (Proposition \ref{finalphi3exist}) and $\FF_{\hp=0}(\Phi_1)$ is bijective (Proposition \ref{finalphi1phi1limitisbij}), $\FF_{\hp=0}(\Phi_3)$ and hence $\FF_{\hp=0}(\Mir)$ is a ring map.
\vspace{0cm}

\item \underline{\textit{$\Mir\circ\shift^B_{\lambda}=\shift^A_{\lambda}\circ\Mir$ for any $\lambda\in\Q$.}} By $\shift^{A/B}_{\lambda_1}\circ\shift^{A/B}_{\lambda_2}=\shift^{A/B}_{\lambda_1+\lambda_2}$ (Lemma \ref{Ashiftlemmamulti} and Lemma \ref{Bshiftlemmamulti}), we may assume $\lambda\in\cop$. The result follows from the analogous results for $\Phi_0$ and $\Phi_3$ (Lemma \ref{finalphi0shift} and Lemma \ref{finalphi3shift}).
\vspace{0cm}

\item \underline{\textit{$\Mir$ is graded.}} This is because $\Phi_0$ and $\Phi_3$ are graded (Lemma \ref{finalphi0graded} and Lemma \ref{finalphi3graded}).
\vspace{0cm}
\end{enumerate}

Finally, we prove that any $\utd\ot\OO(\ZL)$-linear map satisfying (3) and (5) must be equal to $\Mir$. Let $\Phi'$ be such a map. For any $\lambda\in\Q$, we have
\[ \Phi'(\shift^B_{\lambda}([\vol])) = \shift^A_{\lambda}(\Phi'([\vol])) = \shift^A_{\lambda}(1) = \shift^A_{\lambda}(\Mir([\vol])) = \Mir(\shift^B_{\lambda}([\vol])).\]
The result will be proved if we can show that $\bries$ is additively generated by $\shift^B_{\lambda}([\vol])$ ($\lambda\in\Q$) over $\fof(\utd)\ot\OO(\ZL)$. Since $\Mir$ satisfies (1), (3) and (5), it suffices to show that $QH_{\That}^{\bl}(G/P)[\ecc]$ is additively generated by $\shift^A_{\lambda}(1)$ ($\lambda\in\Q$) over $\fof(H_{\That}^{\bl}(\pt))\ot\CC[\eccc]$. This follows from Corollary \ref{GMPPLScor} and the fact that $H_{\bl}^{\That}(\ag)$ is additively generated by $[t^{\lambda}]$ ($\lambda\in\Q$) over $\fof (H_{\That}^{\bl}(\pt))$.
\end{myproof}

\bigskip
\begin{myproof}{Theorem}{\ref{main2}}
Notice that 
\[ \widetilde{\Phi}_R^{\hp}= \Phi_0^{-1}\circ\widetilde{\Phi}_1|_{\dkf(\KK)},\quad \Phi_{YZ}^{\hp} = \pbfd\quad \text{ and }\quad \Phi_{PLS}^{\hp}=\pgmpt.\]
The commutativity of the diagram follows from $\widetilde{\Phi}_2|_{\dkf(\KK)}=\pgmpt\circ\pbfd$ (Definition \ref{finalphi2phi2def}), $\Phi_3\circ\Phi_1=\Phi_2$ (Proposition \ref{finalphi3exist}) and $\Mir=\Phi_3\circ\Phi_0$ (Definition \ref{finalfinalphidef}). The uniqueness follows from the surjectivity of $\Phi_1$ (Proposition \ref{finalphi1phi1surj}).
\end{myproof}
%%%%%%%%%%%%%%%%%%%%%%%%%%%%%%%%%%%
%%%%%%%%%%%%%%%%%%%%%%%%%%%%%%%%%%%
%%%%%%%%%%%%%%%%%%%%%%%%%%%%%%%%%%%
\appendix  
\section{Proofs from preceding sections} \label{B}
We prove several lemmas which are used in the preceding sections.

\begin{lemma}\label{ApplemmaA}(Used in the proof of Lemma \ref{finalphi0descendlemma}) For any $\a\in -R^+$, $\rhod(\tt_{\a,\hh})$ is an extension of $\frac{\LL_{\zeta_{\ad}}\vol}{\vol}$.
\end{lemma}
\begin{proof} Define $\UUU:=\Umd(\wpd)^{-1}\Pmd/\Pmd$, the open Schubert cell in $\Gd/\Pmd$. Notice that $\UP\subseteq \UUU$ where $\UP$ is defined in Section \ref{Rietschmirror} (before Lemma \ref{Rietschfiblemma}). Let $f\in\OO(\UUU)$ be a defining function of the reduced closed subscheme $\UUU\setminus\UP$. By the assumption on $\w_{\UP}$ (see the paragraph before Definition \ref{Fiberwisevolumeformdef}), $f|_{\UP}\w_{\UP}$ extends to a volume form $\w_{\UUU}$ on $\UUU$. Define $\ol{\w}_{\UUU}:=(\pr_{\UP}\circ\nu)^*\w_{\UUU}$ and $\ol{f}:=f\circ\pr_{\UP}\circ\nu$ so that $\vol=\ol{f}^{-1}\ol{\w}_{\UUU}$. (Recall $\nu$ comes from Lemma \ref{Rietschfiblemma}.) Lemma \ref{ApplemmaA} will be proved if we can show 
\begin{align}
\LL_{\zeta_{\ad}}\ol{\w}_{\UUU}&= 0 \label{ApplemmaAeq1} \\
\LL_{\zeta_{\ad}}\ol{f} &= -\ol{f}\rhod(\tt_{\a,\hh})|_{\Xd} . \label{ApplemmaAeq2}
\end{align}

Let us prove \eqref{ApplemmaAeq1} first. Let $V_{\ad}\in\mathfrak{X}(\Gd/\Pmd)$ be the vector field generated by the $\ga$-action $s\mapsto \psi_{\ad}^s:= \exp(s e_{\ad})\cdot -$. Observe that $e_{\ad}^R$ is a lift of $V_{\ad}$ with respect to the projection $\pi_{\Gd/\Pmd}:\Gd\ra \Gd/\Pmd$. Since the left translates of $\mathfrak{n}^{\vee}_-$ are tangent to the fibers of $\pi_{\Gd/\Pmd}$, it follows that $\pr_{\mathfrak{b}^{\vee}}^L(e_{\ad}^R)$ is also a lift of $V_{\ad}$ with respect to $\pi_{\Gd/\Pmd}$, and hence $\nu_*\zeta_{\ad}=V_{\ad}|_{\UP}\oplus 0$. Thus, \eqref{ApplemmaAeq1} is equivalent to $\LL_{V_{\ad}}\w_{\UUU}=0$. Observe that for each $s$, $(\psi_{\ad}^s)^*\w_{\UUU}$ is a volume form on $\UUU$. Since $\UUU$ is an affine space, there is a morphism $\ga\ni s\mapsto c_s\in\gm$ such that $(\psi_{\ad}^s)^*\w_{\UUU}=c_s\w_{\UUU}$. Then $c_s$ is necessarily constant (in fact $\equiv 1$). This gives 
\[ \LL_{V_{\ad}}\w_{\UUU} = \left. \frac{d}{ds}(\psi_{\ad}^s)^*\w_{\UUU}\right|_{s=0} =0 \]
as desired.

It remains to prove \eqref{ApplemmaAeq2}. We first determine $f$ as follows. Identify $\GB$ with $\Gd_{sc}/\Bd_{sc,-}$ where $\Gd_{sc}$ is the universal covering of $\Gd$ and $\Bd_{sc,-}$ is the Borel subgroup of $\Gd_{sc}$ lying over $\Bmd$. Let $\{\wdd_1,\ldots,\wdd_r\}$ be the dual basis of $\{\a_1,\ldots,\a_r\}$. For any $1\leqslant i\leqslant r$, define $\LL(\wdd_i):=\Gd_{sc}\times^{\Bd_{sc,-}}\CC_{\wdd_i}$ and $S(\wdd_i):=H^0(\GB; \LL(\wdd_i))$. Then $\LL(\wdd_i)$ is a line bundle on $\GB$ and $S(\wdd_i)$ is a representation of $\Gd_{sc}$ (in fact the $i$-th fundamental representation). Take a non-zero vector $v_e\in S(\wdd_i)_{\wdd_i}$ (i.e. highest weight vector). Let $v_{\wdd_i}^*\in S(\wdd_i)^*$ be the unique vector satisfying $v_{\wdd_i}^*|_{S(\wdd_i)_{\ne\wdd_i}}=0$ and $\langle v_{\wdd_i}^*, v_e\rangle =1$. Define $\vp_i\in\OO(\Gd_{sc})$ by $\vp_i(g):=\langle v_{\wdd_i}^*, g^{-1}\cdot v_e\rangle$. Then $\vp_i$ descends to a section of $\LL(\wdd_i)$ whose scheme-theoretic zero locus is equal to $\ol{\Ud\dot{s}_{\a_i}\Bmd/\Bmd}$. Define $\UPd:=\exp\left(\bigoplus_{\b\in w_0(R^+\setminus R^+_P)}\ggd_{\beta^{\vee}}\right)\subseteq \Umd$. Notice that the morphism $\UPd\ra\UUU$ defined by $u\mapsto u(\wpd)^{-1}\Pmd$ is an isomorphism. By the definition of $\UP$ and the fact that every $\Bd_-$-orbit in $\GB$ intersects every $\Bd$-orbit transversely, we have, after possibly rescaling $f$,
\begin{equation}\label{ApplemmaAeq3}
f(u(\wpd)^{-1}\Pmd) = (\vp_1\cdots\vp_r)(u(\wpd)^{-1})
\end{equation}
for any $u\in \UPd$. (The points $u$, $\dot{w}_P$ and $\dot{w}_0$ have natural lifts in $\Gd_{sc}$. By abuse of notation, we denote these lifts by the same symbols.)

Next, we express $\ol{f}$ in terms of some known functions on $\Xd$ based on what we have obtained from the previous paragraph. Let $x=(g,t)\in\Xd$. There are $u_0\in\Ud$, $t_0\in\Td$ and $u_1,u_2\in\Umd$ such that 
\[ g= u_0t_0 =u_1(\wpd)^{-1}tu_2.\]
These points are unique if we require $u_1\in\UPd$. Put $u_2':=tu_2t^{-1}\in\Umd$. We have 
\begin{align*}
\vp_i(u_1(\wpd)^{-1}) &= \langle v_{\wdd_i}^*, (u_1(\wpd)^{-1})^{-1}\cdot v_e\rangle \\
& =  \langle v_{\wdd_i}^*,(u_2')^{-1} (u_1(\wpd)^{-1})^{-1} u_0\cdot v_e\rangle  \\
& =\langle v_{\wdd_i}^*, tt_0^{-1}\cdot v_e\rangle \\
& =\wdd_i(tt_0^{-1}) \\
& = (\wdd_i\circ(\pi/p))(x).
\end{align*}
(Recall we are abusing notation. The point $tt_0^{-1}$ is actually a lift of the corresponding point in $\Td$. It depends on the natural lifts of $u_0$, $u_1$, $u_2'$, $\dot{w}_P$ and $\dot{w}_0$.) Therefore, by \eqref{ApplemmaAeq3}, 
\[ \ol{f}(x) = \left(\prod_{i=1}^r\vp_i\right)(u_1(\wpd)^{-1}) = \prod_{i=1}^r(\wdd_i\circ (\pi/p))(x) = (\rhod\circ (\pi/p))(x).\]

Finally, we show $\LL_{\zeta_{\ad}}(\rhod\circ (\pi/p)) = -(\rhod\circ (\pi/p)) \rhod(\tt_{\a,\hh})|_{\Xd}$. This will give \eqref{ApplemmaAeq2}. It suffices to show $\iota_{\zeta_{\ad}}(\pi/p)^*\mccccc_{\Td}= -\tt_{\a,\hh}|_{\Xd}$. Clearly, we have $\iota_{\zeta_{\ad}}\pi^*\mccccc_{\Td}=0$. Thus we are done if we can show $\iota_{\zeta_{\ad}}p^*\mccccc_{\Td} = \tt_{\a,\hh}|_{\Xd}$. This follows from Lemma \ref{ApplemmaC} below.
\end{proof}

\begin{lemma}\label{ApplemmaB}(Lemma \ref{finalphi0descendlemma}) For any $\a\in -R^+$, $\cchh(e_{\ad})-\sum_{\b\in -R^+}\tt_{\a,\b}\cchh(e_{\bd})$ is an extension of $\LL_{\zeta_{\ad}}W$.
\end{lemma}
\begin{proof}
By definition, $\zeta_{\ad}=(\pr_{\bbd}^L(e_{\ad}^R)\oplus 0)|_{\Xd}$, and by \eqref{finalphi0descendlemmaeq1} we have
\[\pr_{\bbd}^L(e_{\ad}^R)=e_{\ad}^R - \sum_{\b\in -R^+}\tt_{\a,\b}e_{\bd}^L .\]
Observe that both $e_{\ad}^R$ and $e_{\bd}^L$ are tangent to $\Umd(\wpd)^{-1}\ZL\Umd$, and $W$ is the restriction of the regular function $W'\in \OO(\Umd(\wpd)^{-1}\ZL\Umd\times\ZL)$ defined by
\[ W'(x):= e^{\cchh}(u_1)+e^{\cchh}(u_2)\quad \text{for any }~x=(u_1(\wpd)^{-1}t_1u_2,t).\]
It is not difficult to see that 
\[ \LL_{(e_{\ad}^R\oplus 0)}W'= \cchh(e_{\ad})\quad\text{ and }\quad \LL_{(e_{\bd}^L\oplus 0)}W'= \cchh(e_{\bd}).\]
The rest is clear.
\end{proof}

\begin{lemma}\label{ApplemmaC}(Lemma \ref{finalphi0descendlemma}, Lemma \ref{ApplemmaA}) For any $\a\in -R^+$, $\tt_{\a,\hh}$ is an extension of $\iota_{\zeta_{\ad}}p^*\mccccc_{\Td}$.
\end{lemma}
\begin{proof}
By definition, $\zeta_{\ad}=(\pr_{\bbd}^L(e_{\ad}^R)\oplus 0)|_{\Xd}$, and by \eqref{finalphi0descendlemmaeq1} we have
\[\pr_{\bbd}^L(e_{\ad}^R) =\sum_i\langle h^i,\tt_{\a,\hh}\rangle h_i^L +\sum_{\b\in R^+}\tt_{\a,\b}e_{\bd}^L .\]
Observe that both $h_i^L$ and $e_{\bd}^L$ are tangent to $\Bd$, and $p$ is the restriction of the regular function $p'\in \OO(\Bd\times\ZL)$ defined by
\[ p'(x):= t_0\quad\text{for any }~x=(u_0t_0,t) .\]
It is not hard to see that 
\[ \LL_{(h_i^L\oplus 0)}(p')^*\mccccc_{\Td}=h_i\quad\text{ and }\quad \LL_{(e_{\bd}^L\oplus 0)}(p')^*\mccccc_{\Td}=0.\]
The rest is clear.
\end{proof}

\begin{lemma} \label{dunnowheretoplacelemmab} (Lemma \ref{dunnowheretoplacelemma1}) The closed subscheme of $\Gd\times\Td$ defined by $f_{\lambda,v_{\wp}}-q^{\wp(\lambda)}$ ($\lambda\in\cop$) and $f_{\lambda,v}$ ($\lambda\in\cop$, $v\in S(\lambda)_{<\wp(\lambda)}$) is equal to 
\[ \mathfrak{X}_P:=\left\{ (g,t)\in\Gd\times\Td\left|~g\in\Umd(\wpd)^{-1}t\Umd \right.\right\}.\]
\end{lemma}  
\begin{proof}
It is straightforward to see that these equations vanish on $\mathfrak{X}_P$. Let $\vp=\sum_{\lambda\in\Q}\vp_{\lambda}q^{\lambda}\in\OO(\Gd\times\Td)$ with $\vp_{\lambda}\in\OO(\Gd)$. Suppose $\vp|_{\mathfrak{X}_P}=0$. Let $S\subseteq\Q$ be the set of $\lambda$ for which $\vp_{\lambda}\ne 0$. Then $S$ is finite so there exists $\mu_0\in\Q$ such that $\lambda+\mu_0\in \wp\cop$ for any $\lambda\in S$. It follows that 
\[\vp= q^{-\mu_0}\sum_{\lambda\in S}\vp_{\lambda}(q^{\lambda+\mu_0}-f_{(\wp)^{-1}(\lambda+\mu_0),v_{\wp}}) +q^{-\mu_0}\sum_{\lambda\in S} \vp_{\lambda} f_{(\wp)^{-1}(\lambda+\mu_0),v_{\wp}}.\]
Put $\phi:=\sum_{\lambda\in S} \vp_{\lambda} f_{(\wp)^{-1}(\lambda+\mu_0),v_{\wp}}\in \OO(\Gd)$. Since $\vp|_{\mathfrak{X}_P}=0$, we have $\phi|_{\Umd (\wpd)^{-1}\Td\Umd}=0$. By looking at the line bundles on $\GB$ and using an argument from the proof of \cite[Lemma 6.2]{me2}, we see that there exists $\lambda_0\in\cop$ such that $\phi f_{\lambda_0,v_{\wp}}$ belongs to the ideal of $\OO(\Gd)$ generated by $f_{\lambda,v}$ with $\lambda\in\cop$ and $v\in S(\lambda)_{<\wp(\lambda)}$. We are done because $\phi=q^{-\wp(\lambda_0)}\phi f_{\lambda_0,v_{\wp}} - q^{-\wp(\lambda_0)}\phi (f_{\lambda_0,v_{\wp}}-q^{\wp(\lambda_0)})$.
\end{proof} 

\begin{lemma}\label{finalfinaladdlemma} (Proposition \ref{finalfinaldmlinear}, Lemma \ref{finalfinalAipreserveV}) $\ffff_i-\w_i^L\ffff_0$ and $\ffff_0$ represent some elements of $\kf(\KK)$. 
\end{lemma}
\begin{proof}
It is clear that both represent some elements of $\KK$ and the latter even represents an element of $\kf(\KK)$. It remains to show that $[\ffff_i-\w_i^L\ffff_0]\in\KK$ is $\Umd$-invariant, or equivalently $\nd$-invariant, modulo $\mathcal{I}^{\KK}_{-,\cchh}$ (see Definition \ref{BFenhanceKFdef}). Recall from Example \ref{BFenhanceHCeg2} we are using the right translation. Let $x\in\nd$. We have
\[ x\cdot \ffff_i = \ffff_{\lambda_0,x\cdot v_i}\quad\text{ and }\quad x\cdot (\w_i^L\ffff_0) = [x,\w_i]^L\ffff_0=\ffff_0([x,\w_i]^L-\cchh([x,\w_i])) + \cchh([x,\w_i])\ffff_0.\]
Hence $x\cdot [\ffff_i-\w_i^L\ffff_0]$ is equal to $[\ffff_{\lambda_0,x\cdot v_i}-\cchh([x,\w_i])\ffff_0]$ modulo $\mathcal{I}_{-,\cchh}^{\KK}$. 

By \cite[Theorem 5.4]{Acta}, \cite[Lemma A.1]{me2} and a straightforward computation, we have $e_i^{\vee}\cdot v_{w_0} = -\a_i(w_0(\lambda_0))v_i$. (Recall we have made a specific choice of $e_i^{\vee}$. See Remark \ref{Rietschmirrorrmk}.) It follows that 
\[x\cdot v_i = -\a_i(w_0(\lambda_0))^{-1}x\cdot (e_i^{\vee}\cdot v_{w_0}) =-\a_i(w_0(\lambda_0))^{-1} [x,e_i^{\vee}]\cdot v_{w_0}.\]
Notice that $x\in |\ad_i|^2\b(e_i^{\vee},x)f_i^{\vee}+\bigoplus_{\a\in R^+\setminus\{\a_i\}}\ggd_{-\ad}$. (See the paragraph before Remark \ref{Rietschmirrorrmk} for the definitions of $|\ad_i|^2$ and $\b(-,-)$.) It follows that 
\begin{align*}
x\cdot v_i &= |\ad_i|^2\b(e_i^{\vee},x)v_{w_0}\\
&= \left(\sum_{j=1}^r |\ad_j|^2\b(e_j^{\vee}, [x,\w_i] ) \right)  v_{w_0}\\
& = \b(e, [x,\w_i]) v_{w_0}\\
& = \cchh([x,\w_i]) v_{w_0}.
\end{align*}
Therefore, $f_{\lambda_0,x\cdot v_i}-\cchh([x,\w_i])f_0=0$. The result follows.
\end{proof} 

\begin{lemma}\label{ApplemmaD}(Lemma \ref{finalfinalAipreserveV}, Lemma \ref{finalfinalAicoincidexi}) The vector field 
\[\widetilde{\zzeta}_i:=q^{-[\ad_i]} \w_i^L+\partial_{q_i}\in\mathfrak{X}(\Gd\times\ZL) \]
is tangent to $\Xd$. The restriction
\[ \zzeta_i := \widetilde{\zzeta}_i|_{\Xd}\in \mathfrak{X}(\Xd)\]
is a lift of $\partial_{q_i}\in\mathfrak{X}(\ZL)$ with respect to $\pi$.
\end{lemma} 
\begin{proof}
The first assertion follows from the observation that the flow of $q^{[\ad_i]}\widetilde{\zzeta}_i$ is given by $(g,t)\mapsto (ge^{s\w_i},te^{s\w_i})$. The second assertion is clear from definition.
\end{proof}

\begin{lemma}\label{ApplemmaF}(Lemma \ref{finalfinalAicoincidexi}) $\LL_{\zzeta_i}\vol=0$.
\end{lemma}
\begin{proof}
By Definition \ref{Fiberwisevolumeformdef}, $\vol$ is the pull-back of a volume form on $\UP$ via $\pr_{\UP}\circ\nu$. The flow of $q^{[\ad_i]}\zzeta_i$ is given by $(g,t)\mapsto (ge^{s\w_i},te^{s\w_i})$. It is clear that $\pr_{\UP}\circ\nu$ is invariant under this flow, and hence $\LL_{q^{[\ad_i]} \zzeta_i}\vol=0$. Since $q^{[\ad_i]}$ is a regular function on the base scheme $\ZL$, we have $\LL_{\zzeta_i}\vol=0$. 
\end{proof}

\begin{lemma}\label{finalfinalextend} (Lemma \ref{finalfinalAicoincidexi}) $q^{-[w_0(\lambda_0)+\ad_i]}\ffff_i$ is an extension of $\LL_{\zzeta_i}W$.
\end{lemma}
\begin{proof}
Let $(g,t)\in\Xd$ be a point where $g=u_1(\wpd)^{-1}tu_2$ with $u_1,u_2\in\Umd$.

The flow of $q^{[\ad_i]}\zzeta_i$ is given by $(g,t)\mapsto (ge^{s\w_i},te^{s\w_i})$. By the paragraph before Remark \ref{Rietschmirrorrmk} and a straightforward computation, we see that $\LL_{\zzeta_i}W$ sends $(g,t)$ to $\ad_i(t)^{-1}|\ad_i|^2\b(e_i^{\vee},y_i(u_2))$ where $y_i:=\pr_{\ggd_{-\ad_i}}\circ\exp_{\Umd}^{-1}$ is the composition of the inverse of the exponential map $\exp_{\Umd}:\nd\ra\Umd$ and the projection $\pr_{\ggd_{-\ad_i}}:\nd\ra\ggd_{-\ad_i}$. 

On the other hand, we have
\begin{align*}
q^{-[w_0(\lambda_0)+\ad_i]} \ffff_i(g,t) &= \lambda_0(w_0(t))^{-1} \ad_i(t)^{-1} \langle v_{\lambda_0}^*, u_1(\wpd)^{-1}tu_2\cdot v_i\rangle\\
&= \ad_i(t)^{-1} \langle v_{\lambda_0}^*, (\wpd)^{-1}\cdot (y_i(u_2)\cdot v_i)\rangle.
\end{align*}
By the proof of Lemma \ref{finalfinaladdlemma}, we have
\[ y_i(u_2)\cdot v_i = |\ad_i|^2\b(e_i^{\vee},y_i(u_2))v_{w_0}.\]
By \cite[Proposition B.7]{me2}, $(\wpd)^{-1}\cdot v_{w_0}=(\wpd)^{-1}\cdot v_{\wp}=v_e$. Therefore,
\[ \ad_i(t)^{-1} \langle v_{\lambda_0}^*, (\wpd)^{-1}\cdot (y_i(u_2)\cdot v_i)\rangle  =  \ad_i(t)^{-1} |\ad_i|^2 \b(e_i^{\vee},y_i(u_2)).\]
The proof is complete. 
\end{proof}

\begin{lemma}\label{ApplemmaE}(Lemma \ref{finalfinalAicoincidexi}) $\iota_{\zzeta_i}p^*\langle h^j,\mccccc_{\Td}\rangle  = q^{-[\ad_i]}\langle h^j,\w_i\rangle$.
\end{lemma}
\begin{proof}
The flow of $q^{[\ad_i]}\zzeta_i$ is given by $(g,t)\mapsto (ge^{s\w_i},te^{s\w_i})$. Using this flow, it is not hard to see that 
\[\iota_{q^{[\ad_i]} \zzeta_i}p^*\langle h^j,\mccccc_{\Td}\rangle  = \langle h^j,\w_i\rangle \]
The result follows.
\end{proof}

%%%%%%%%%%%%%%%%%%%%%%%%%%%%%%%%%%%
%%%%%%%%%%%%%%%%%%%%%%%%%%%%%%%%%%%
%%%%%%%%%%%%%%%%%%%%%%%%%%%%%%%%%%%
%%%%%%%%%%%%%%%%%%%%%%%%%%%%%%%%%%%
%%%%%%%%%%%%%%%%%%%%%%%%%%%%%%%%%%%
%%%%%%%%%%%%%%%%%%%%%%%%%%%%%%%%%%%
\section{Technical results on Brieskorn lattice}\label{D}
We prove three technical results on $\bries$. They are used in Section \ref{finalphi3} and rely on some results which are proved prior to that subsection. Denote by $G_0(\hp,h,t)$ the fiber of $\bries$ at a given point 
\[(\hp,h,t)\in \spec\left(\utd\otimes\OO(\ZL)\right) \simeq \AA^1_{\hp}\times\hh\times\ZL.\]

\begin{proposition} \label{Rietschfiberlemma}
For any $(\hp_0,h_0,t_0)\in\AA^1_{\hp}\times\hh\times\ZL$ and $c\in\CC^{\times}$ such that $\hp_0\ne 0$, we have 
\[ \dim G_0(\hp_0,h_0,t_0) = \dim G_0(c\hp_0,ch_0,t_0) < +\infty.  \]
\end{proposition}
\begin{proof}
We proceed by modifying the proof for the case $h_0=0$ which is well-known. See e.g. \cite[Section 1.1]{Sabbah} where the superpotential is assumed to be projective. 

Let $(\hp,h,t)\in \AA^1_{\hp}\times\hh\times\ZL$. Define $\Xdt:=\pi^{-1}(t)$, $W_t:=W|_{\Xdt}$, $p_t:=p|_{\Xdt}$ and $\w_h:=\langle h,\mctd\rangle$. (Recall $\mctd$ is the Maurer-Cartan form of $\Td$.) We have 
\[  G_0(\hp,h,t) \simeq \coker\left(\O^{top-1}(\Xdt)\xrightarrow{~\hp d+dW_t-p_t^*\w_h~}\O^{top}(\Xdt)\right).\]
Let $\mathcal{E}$ denote the integrable connection $(\OO_{\Td},d-\w_h)$ with degree shifted suitably. Then 
\begin{equation}\label{Rietschfibereq1}
\mathcal{H}^0(\textstyle{\int}_{W_t}\circ p_t^{\dag})\mathcal{E} \simeq \coker\left(\O^{top-1}(\Xdt)[\partial_s]\xrightarrow{~ d-\partial_sdW_t-p_t^*\w_h~}\O^{top}(\Xdt)[\partial_s]\right)
\end{equation}
where $\int_{W_t}$ and $p_t^{\dag}$ are the direct and inverse image functors of $D$-modules respectively. See e.g. \cite[Section 1.5]{Japan} for more details.

Consider the Fourier transform functor
\[ \FF:D_{\AA^1_s}\text{--}\module \ra D_{\AA^1_{\tau}}\text{--}\module\]
defined by sending each $D_{\AA^1_s}$-module to the module with the same underlying vector space on which $\tau$ and $\partial_{\tau}$ act the same way as $\partial_s$ and $-s$ do respectively. Since $\mathcal{E}$ is regular holonomic, so is $\mathcal{H}^0(\int_{W_t}\circ p_t^{\dag})\mathcal{E} $ (see e.g. \cite[Theorem 6.1.5]{Japan}), and hence, by a standard result on differential equations in dimension one (see e.g. \cite[Chapter V Proposition 2.2]{Sabbahbook}), the localization $(\FF\circ \mathcal{H}^0(\int_{W_t}\circ p_t^{\dag})\mathcal{E} )\otimes_{\CC[\tau]}\CC[\tau,\tau^{-1}]$ is a free $\CC[\tau,\tau^{-1}]$-module of finite rank. It follows that, by \eqref{Rietschfibereq1}, the $\CC[\tau,\tau^{-1}]$-module
\[ N(h,t) := \coker\left(\O^{top-1}(\Xdt)[\tau,\tau^{-1}]\xrightarrow{~ d-\tau dW_t-p_t^*\w_h~}\O^{top}(\Xdt)[\tau,\tau^{-1}]\right)\]
is free of finite rank. Put $(h,t):=(\hp_0^{-1}h_0,t_0)$. Then we have
\[ \dim G_0(\hp_0,h_0,t_0) = \dim N(h,t)_{\tau=-\hp_0^{-1}}=\dim N(h,t)_{\tau=-c^{-1}\hp_0^{-1}}=\dim G_0(c\hp_0,ch_0,t_0)<+\infty. \]
\end{proof}

\begin{proposition}\label{hTF}
$\bries$ is $\hp$-torsion-free.
\end{proposition}
\begin{proof}
This is proved by Lam and Templier \cite[Proposition 16.13]{LT}. For the convenience of the reader, we reproduce their proof in terms of our notations.

Observe that $\bries$ is the top cohomology of the cochain complex $C^{\bl}:= (\utd\ot\O^{\bl}(\Xd/\ZL),\partial)$ where $\partial$ is defined similarly as in Definition \ref{RietschBrieskorndef}. Since each $C^i$ is $\hp$-torsion-free, the sequence 
\[ 0\ra C^{\bl} \xrightarrow{\hp} C^{\bl} \ra C^{\bl}/\hp C^{\bl}\ra 0\]
is exact, and hence we have an exact sequence
\[ H^{d-1}(C^{\bl}/\hp C^{\bl}) \ra \bries \xrightarrow{\hp} \bries \ra  H^{d}(C^{\bl}/\hp C^{\bl})  \ra 0\]
where $d$ is the top degree. The proof is complete if we can show $ H^{d-1}(C^{\bl}/\hp C^{\bl})=0$.

Take a global frame $\{\zeta_j\}$ of $\Xd$ relative to $\pi$. For each $j$, define
\[ r_j:= 1\ot\LL_{\zeta_j}W-\sum_i h_i\ot \iota_{\zeta_j}p^*\langle h^i,\mccccc_{\Td}\rangle \in R:= \sym^{\bl}(\ttd)\ot\OO(\Xd).\]
Observe that $C^{\bl}/\hp C^{\bl}$ is the Koszul complex associated to the sequence $\{r_j\}$. By the theory of Koszul complexes, the vanishing of $H^{d-1}(C^{\bl}/\hp C^{\bl})$ follows if we can show that $\{r_j\}$ is an $R_{\mathfrak{p}}$-regular sequence for any prime ideal $\mathfrak{p}$ of $R$ which contains $\langle r_j\rangle$. By a dimension argument, the latter condition holds if we can show that $\spec R/\langle r_j\rangle$ is quasi-finite over $\spec(\sym^{\bl}(\ttd)\ot\OO(\ZL))$. We have
\[R/\langle r_j\rangle \simeq H^d(C^{\bl}/\hp C^{\bl})\simeq \FF_{\hp=0}(\bries).\]
By the bijectivity of $\FF_{\hp=0}(\Phi_i)$ for $i=0,1,2$ (Proposition \ref{finalphi0bijective}, Proposition \ref{finalphi1phi1limitisbij} and Proposition \ref{finalphi2phi2limitisbij}), we have
\begin{equation}\label{hTFeq1}
\FF_{\hp=0}(\bries)\simeq QH_T^{\bl}(G/P)[\ecc]
\end{equation}
which is a finitely generated $\sym^{\bl}(\ttd)\ot\OO(\ZL)$-module. The quasi-finiteness follows.
\end{proof}

\begin{remark}\label{hTFrmk}
In the above proof, we deduce the quasi-finiteness from Proposition \ref{finalphi2phi2limitisbij} which depends on the \textit{Peterson variety presentation} for $QH_T^{\bl}(G/P)[\ecc]$ \cite[Theorem B]{me2}. It should be pointed out that what we actually need is its corollary: the quasi-finiteness of another scheme $\Bed\times_{\Gd}\Umd(\wpd)^{-1}\Bd_-$ over the same base scheme, which is proved in the course of the proof of the Peterson variety presentation (see the proof of \cite[Lemma 5.7]{me2}). We believe that this less non-trivial result is already known to experts. See e.g. the proof of \cite[Proposition 6.2]{LamRietsch}.
\end{remark}

\begin{proposition}\label{coherent}
$\bries$ is a finitely generated $\utd\ot\OO(\ZL)$-module.
\end{proposition}
\begin{proof}
Since $\Phi_0$ is bijective (Proposition \ref{finalphi0bijective}) and $\Phi_1$ is surjective (Proposition \ref{finalphi1phi1surj}), it suffices to show that $\dkf(\KK)\ot\OO(\ZL)/\WW$ is a finitely generated $\utd\ot\OO(\ZL)$-module. 

Consider first the case $P=B$. By the bijectivity of $\pbfd$ (Theorem \ref{BFqttheorem}) and Proposition \ref{GMPPLS}, the $\utd$-linear map $\pgmpt\circ\pbfd$ is injective and identifies $\dkf(\KK)$ with the submodule 
\[\mathcal{N}:=\bigoplus_{w\in W}\bigoplus_{\lambda\in\Lambda_w}H_{\That}^{\bl}(\pt)\cdot q^{\lambda}\s_w\subset QH_{\That}^{\bl}(G/B)[q_i^{-1}|~i\in I] \]
where $\Lambda_w:=\{\lambda\in \Q|~\wl\in W_{af}^{-}\}$. By Lemma \ref{coherentlemma} below, for any $w\in W$, there exist $\lambda_{w,1},\ldots, \lambda_{w,k_w}\in\Q$ such that $\Lambda_w=\bigcup_{i=1}^{k_w}(\lambda_{w,i}-\cop)$. Hence $\mathcal{N}$ is a finitely generated $H_{\That}^{\bl}(\pt)[-\cop]$-module. Since $\pbfd$ is a homomorphism of modules with respect to the ring isomorphism $\pbfk$ (Theorem \ref{BFqttheorem}), $\gmpg$ is a module action (Proposition \ref{GMPactionmodule}) and $\pgmpg\circ\pbfk([f_{\lambda,v_{w_0}}])=\pgmpt([\agll])=q^{w_0(\lambda)}$ for any $\lambda\in\cop$ (Proposition \ref{GMPPLS} and Proposition \ref{BFenhancedivisor}), it follows that $\dkf(\KK)$ is a finitely generated $\utd[\cop]$-module where the $\CC[\cop]$-module structure is given by the (right) multiplication by $[f_{\lambda,v_{w_0}}]\in\kf(\KK)$. Let $\{y_1,\ldots,y_N\}\subset \dkf(\KK)$ be a set of generators.

Now let us return to the case where $P$ is arbitrary. We will show that $[y_i\ot 1]$, $i=1,\ldots,N$, generate the $\utd\ot\OO(\ZL)$-module $\dkf(\KK)\ot\OO(\ZL)/\WW$. Every element of this module is a $\utd\ot\OO(\ZL)$-linear combination of $[y_i\cdot [f_{\lambda,v_{w_0}}]\ot 1]$ where $i=1,\ldots,N$ and $\lambda\in\cop$. Notice that $[f_{\lambda,v_{w_0}}]=a_{\lambda,v_{w_0}}$ by Lemma \ref{finaltwomodarmk}. If $w_0(\lambda)=\wp(\lambda)$, then $a_{\lambda, v_{w_0}}=a_{\lambda,v_{\wp}}$, and hence 
\[[y_i\cdot [f_{\lambda,v_{w_0}}]\ot 1]= [(y_i\ot 1)\cdot (a_{\lambda,v_{\wp}} - q^{[w_0(\lambda)]})] + q^{[w_0(\lambda)]}[y_i\ot 1] = q^{[w_0(\lambda)]}[y_i\ot 1].\]
Otherwise, we have $w_0(\lambda)<\wp(\lambda)$, and hence $[y_i\cdot [f_{\lambda,v_{w_0}}]\ot 1]= [(y_i\ot 1)\cdot a_{\lambda,v_{w_0}}]=0$. The proof is complete. 
\end{proof}

\begin{lemma}\label{coherentlemma} For any $w\in W$, there exist $\lambda_{w,1},\ldots, \lambda_{w,k_w}\in\Q$ such that 
\[\Lambda_w:=\{\lambda\in\Q|~\wl\in W_{af}^-\}=\bigcup_{i=1}^{k_w}(\lambda_{w,i}-\cop).\]
\end{lemma}
\begin{proof}
First observe that $\Lambda_w\subseteq -\cop$ and $(-\cop)\setminus \Lambda_w$ is equal to a union of the intersections of $-\cop$ and the facets of its convex hull. Embed $\Q$ into $\ZZ^r$ such that $-\cop = \ZZ_{\geqslant 0}^r\cap \Q$. Put $n:=|\ZZ^r/\Q|$. Define $S:= [0,n]^r\cap \Lambda_w$. It is not difficult to see that $\Lambda_w=\bigcup_{x\in S}(x-\cop)$.
\end{proof}

%%%%%%%%%%%%%%%%%%%%%%%%%%%%%%%%%%%
%%%%%%%%%%%%%%%%%%%%%%%%%%%%%%%%%%%
%%%%%%%%%%%%%%%%%%%%%%%%%%%%%%%%%%%
%%%%%%%%%%%%%%%%%%%%%%%%%%%%%%%%%%%
%%%%%%%%%%%%%%%%%%%%%%%%%%%%%%%%%%%
%%%%%%%%%%%%%%%%%%%%%%%%%%%%%%%%%%%
\section{A Remark on Kostant functor} \label{A}
In this appendix, we prove an extension of Lemma \ref{BFenhanceKFisom} which is used in Definition \ref{BFpbfdef} and the proof of Lemma \ref{finaltwomodarmk} as well as a closely related lemma which is used in the proof of Lemma \ref{dunnowheretoplacelemmaa}. The key ideas can be found in \cite{BF} and a paper of Kostant \cite{KW} cited therein. 

Let $\CCC$ be the category whose objects are graded left $\ugd\otimes_{\CC[\hp]}\unmd^{op}$-modules $M$ equipped with a linear algebraic $\Umd$-action satisfying 
\begin{enumerate}
\item $\s(\ggd_{\ad})\cdot M_i\subseteq M_{i+2\rho(\ad)}$ for any $\a\in -R^+$ and $i\in\ZZ$, where $\s:\nndm\ra\eendo_{\CC}(M) $ is the linearization of the $\Umd$-action and $M_i$ is the $i$-th graded piece of $M$;

\item $u\cdot ((x\otimes y)\hcdot m)=((u\cdot x)\otimes (u\cdot y))\hcdot (u\cdot m)$ for any $m\in M$, $x\in\ugd$, $y\in\unmd$ and $u\in\Umd$; and
\item $(x\otimes 1-1\otimes x)\hcdot m= \hp\s(x) m$ for any $m\in M$ and $x\in\nndm$,
\end{enumerate}
and whose morphisms are graded $\Umd$-equivariant $\ugd\otimes_{\CC[\hp]}\unmd^{op}$-linear maps of degree zero. Clearly, there is a natural functor $\hct\ra\CCC$, and the functors $\kf$, $\dkf$ (Definition \ref{BFenhanceKFdef}) and the natural transformation $M\mapsto \tt_M$ \eqref{BFenhancettmdef} extend to $\CCC$. Let $\FF_{\hp=0}$ be the functor $N\mapsto N/\hp N$. We have a canonical natural transformation
\[ \CCC\ni M\mapsto \a_M:\FF_{\hp=0}\circ\kf(M)\ra \kf\circ\FF_{\hp=0}(M).\]

\begin{lemma}\label{Alemma2}
$\tt_M$ is bijective for any $M\in\CCC$.
\end{lemma}

\begin{lemma}\label{Alemma3}
$\a_M$ is bijective for any $M\in\CCC$.
\end{lemma}

We will prove these lemmas simultaneously. First notice that they hold for any $M\in\CCC$ satisfying $\hp M=0$. (The bijectivity of $\tt_M$ follows from Kostant's slice theorem and descent, and the bijectivity of $\a_M$ is obvious.) For any finite dimensional graded $\Umd$-module $V=\bigoplus_{i\in\ZZ}V_i$ satisfying $\ggd_{\ad}\cdot V_i\subseteq V_{i+2\rho(\ad)}$ for any $\a\in -R^+$ and $i\in\ZZ$, define $\fr(V):=\ugd\ot V$. As in Example \ref{BFenhanceHCeg1}, we can define an analogous structure on $\fr(V)$ so that it is an object of $\CCC$. Define $\CCCfr$ to be the full subcategory of $\CCC$ consisting of objects of this form. Define also
\begin{align*}
\CCCp & := \{M\in\CCC|~M[m]\text{ is positively graded for some }m\in\ZZ\}\\
\CCChtf & := \{M\in\CCC|~M\text{ is }\hp\text{-torsion-free}\}.
\end{align*}
Consider the following statements.
\begin{enumerate}
\item $M\in\CCCfr$ $\Longrightarrow$ $\tt_M$ is bijective.

\item $M\in\CCC$ $\Longrightarrow$ $\tt_M$ is bijective.

\item $M\in\CCCfr$ $\Longrightarrow$ $\a_M$ is bijective.

\item $M\in\CCC$ $\Longrightarrow$ $\a_M$ is bijective.

\item $M, M'\in\CCC$, $M'/\mathcal{I}_{-,\cchh}\in \CCCp\cap\CCChtf$, $M\ra M'\ra 0$ exact $\Longrightarrow$ $\kf(M)\ra\kf(M')\ra 0$ exact.

\item $M, M'\in\CCC$, $\hp M'=0$, $M\ra M'\ra 0$ exact $\Longrightarrow$ $\kf(M)\ra\kf(M')\ra 0$ exact.

\item $M, M'\in\CCC$, $M\ra M'\ra 0$ exact $\Longrightarrow$ $\kf(M)\ra\kf(M')\ra 0$ exact.
\end{enumerate}
(Notice that if $N\in\CCC$, then $N/\mathcal{I}_{-,\cchh}\in\CCC$.) We are going to verify (1) and the implications
\[(1) \Rightarrow (3) \Rightarrow (5)\wedge (6) \Rightarrow (7)\quad\text{ and }\quad (1)\wedge (7) \Rightarrow (2) \Rightarrow (4).\]
This will complete the proof of Lemma \ref{Alemma2} and Lemma \ref{Alemma3} because they correspond to (2) and (4) respectively.

\bigskip
\underline{\textit{(1) is true.}} Let $M:= \fr(V)\in\CCCfr$. Since $\utd$ is a free $\zgd$-module of rank $|W|$, we can write $\utd\otimes_{\zgd}\kf(M)$ as $(\kf(M))^{\oplus |W|}$. We have the commutative diagram
\begin{equation}\nonumber
\begin{tikzpicture}
\tikzmath{\x1 = 5; \x2 = 2;}
\node (A) at (0,0) {$(\kf(M))^{\oplus |W|}$} ;
\node (B) at (\x1,0) {$ (\FF_{\hp=0}\circ\kf(M))^{\oplus |W|}$} ;
\node (C) at (2*\x1,0) {$(\kf\circ\FF_{\hp=0}(M))^{\oplus |W|}$} ;
\node (E) at (0,-\x2) {$\dkf(M)$} ;
\node (F) at (\x1,-\x2) {$\FF_{\hp=0}\circ\dkf(M)$} ;
\node (G) at (2*\x1,-\x2) {$\dkf\circ\FF_{\hp=0}(M)$} ;

\path[->>] (A) edge (B);
\path[->, font=\tiny] (A) edge node[left]{$\tt_M$} (E);
\path[->>] (E) edge (F);
\path[->, font=\tiny] (B) edge node[left]{$\FF_{\hp=0}(\tt_M)$} (F);
\path[->, font=\tiny] (C) edge[draw = white] node[left]{$\tt_{\FF_{\hp=0}(M)}$}(G);
\path[->, font=\tiny] (C) edge node[right]{$\simeq$}(G);

\path[->, font=\tiny] (B) edge node[above]{$\a_M^{\oplus |W|}$} (C);
\path[->, font=\tiny] (F) edge node[above]{$\simeq$}  (G);
\end{tikzpicture}.
\end{equation}   
Observe that $M/\mathcal{I}_{-,\cchh}$ is $\hp$-torsion-free. It follows that $\a_M$ is injective and so is $\a_M^{\oplus |W|}$. Thus $\tt_M$ is injective by a Nakayama-type argument. 

As for the surjectivity, consider the canonical filtration on $V$ by degrees which induces a filtration on $M$ by sub-objects, and then a filtration on $M/\mathcal{I}_{-,\cchh}$ by sub-$\Umd$-modules whose associated graded pieces are direct sums of finitely many copies of $Y_{\hp,\cchh}:=\fr(\CC)/\mathcal{I}_{-,\cchh}$. 
Using the long exact sequences for group cohomology, five lemma and induction, it suffices to show that (i) $\tt_{\fr(\CC)}$ is surjective (we have proved that $\tt_{\fr(\CC)}$ is injective) and (ii) $H^1(\Umd;Y_{\hp,\cchh})=0$.

(i) is obvious because $\dkf(\fr(\CC))\simeq\utd$. To prove (ii), it suffices to verify 
\[H^1(\Umd;\utd\ot_{\zgd} Y_{\hp,\cchh})=0\]
because $\utd$ is free over $\zgd$ (here $\Umd$ acts trivially on $\utd$). We have the action morphism $Y_{\hp,\cchh}\ra Y_{\hp,\cchh}\ot\OO(\Umd)$. Composing it with the quotient map $Y_{\hp,\cchh}\ra Y_{\hp,\cchh}/\mathcal{I}_{+}\simeq\utd$ yields a map $Y_{\hp,\cchh}\ra \utd\ot\OO(\Umd)$. Observe that the last map is $\zgd$-linear if we twist the target $\utd$-module by $\twist_{-\rhod}$, and hence by extension of scalars we obtain a $\utd$-linear map
\[\Theta: \utd\ot_{\zgd}Y_{\hp,\cchh}\ra \twist_{-\rhod}(\utd)\ot\OO(\Umd).\] 
It is not difficult to see that $\Theta$ is graded and $\Umd$-linear where the $\Umd$-action on $\OO(\Umd)$ is induced by right translation. Moreover, $\FF_{\hp=0}(\Theta)$ is induced by the canonical section $\ttd\ra \ttd\times_{\ttd/W}(e+\mathfrak{b}^{\vee}_-)$. Since $e+\mathfrak{b}^{\vee}_-$ is a $\Umd$-torsor over $\ttd/W$ by Kostant's slice theorem, $\FF_{\hp=0}(\Theta)$ is bijective. By a Nakayama-type argument, $\Theta$ is bijective as well. Now the vanishing of $H^1(\Umd;\utd\ot_{\zgd} Y_{\hp,\cchh})$ follows from the vanishing of $H^1(\Umd;\OO(\Umd))$. A proof of the latter result can be found in the proof of \cite[Lemma 4.2]{KW}.

\bigskip
\underline{\textit{(1) $\Rightarrow$ (3) and (2) $\Rightarrow$ (4).}} Let $M\in\CCC$. Suppose $\tt_M$ is bijective. Since $\utd$ is a free $\zgd$-module of rank $|W|$, we can write $\utd\ot_{\zgd}\kf(M)$ as $(\kf(M))^{\oplus |W|}$. We have the commutative diagram
\begin{equation}\nonumber
\begin{tikzpicture}
\tikzmath{\x1 = 7; \x2 = 2;}
\node (A) at (0,0) {$(\FF_{\hp=0}\circ\kf(M))^{\oplus |W|}$} ;
\node (B) at (\x1,0) {$(\kf\circ \FF_{\hp=0}(M))^{\oplus |W|}$} ;

\node (E) at (0,-\x2) {$\FF_{\hp=0}\circ\dkf(M)$} ;
\node (F) at (\x1,-\x2) {$\dkf\circ \FF_{\hp=0}(M)$} ;

\path[->, font=\tiny] (A) edge node[above]{$\a_M^{\oplus |W|}$} (B);
\path[->, font=\tiny] (A) edge node[left]{$\FF_{\hp=0}(\tt_M)$} (E);
\path[->, font=\tiny] (E) edge node[above]{$\simeq$} (F);

\path[->, font=\tiny] (B) edge[draw = white] node[left]{$\simeq$} (F);
\path[->, font=\tiny] (B) edge node[right]{$\tt_{\FF_{\hp=0}(M)}$} (F);
\end{tikzpicture}.
\end{equation} 
The bottom and right arrows are always bijective. By assumption, the left arrow is bijective. It follows that $\alpha_M^{\oplus |W|}$, and hence $\alpha_M$, is bijective.

\bigskip
\underline{\textit{(3) $\Rightarrow$ (5)$\wedge$(6).}} Suppose the conditions in (5) are satisfied. We may assume $M$ is a (possibly infinite) direct sum of objects of $\CCCfr$ because every object of $\CCC$ is a quotient of such an object. We have the commutative diagram
\begin{equation}\nonumber
\begin{tikzpicture}
\tikzmath{\x1 = 5; \x2 = 1.5;}
\node (A) at (0,0) {$\kf(M)$} ;
\node (B) at (0.8*\x1,0) {$\FF_{\hp=0}\circ\kf(M)$} ;
\node (C) at (1.8*\x1,0) {$\kf\circ\FF_{\hp=0}(M)$} ;

\node (E) at (0,-\x2) {$\kf(M')$} ;
\node (F) at (0.8*\x1,-\x2) {$\FF_{\hp=0}\circ\kf(M')$} ;
\node (G) at (1.8*\x1,-\x2) {$\kf\circ\FF_{\hp=0}(M')$} ;

\path[->>] (A) edge (B);
\path[->] (A) edge (E);
\path[->>] (E) edge (F);
\path[->] (B) edge (F);
\path[->>] (C) edge (G);

\path[->, font=\tiny] (B) edge node[above]{$\a_M$} (C);
\path[->, font=\tiny] (F) edge node[above]{$\a_{M'}$}  (G);
\end{tikzpicture}.
\end{equation}   
By Kostant's slice theorem and descent, the rightmost arrow is surjective. Since $M'/\mathcal{I}_{-,\cchh}\in\CCChtf$, $\a_{M'}$ is injective. By (3), $\a_M$ is surjective. Hence we can apply a Nakayama-type argument to get the surjectivity of $\kf(M)\ra\kf(M')$. (The condition $M'/\mathcal{I}_{-,\cchh}\in\CCCp$ guarantees that the inductive process will terminate.)

Now suppose the conditions in (6) are satisfied. Again, we assume $M$ is a direct sum of objects of $\CCCfr$  so that $\a_M$ is surjective by (3). Consider the above diagram. By $\hp M'=0$, the two arrows in the bottom row are bijective. This yields the surjectivity of $\kf(M)\ra\kf(M')$ immediately.

\bigskip
\underline{\textit{(5)$\wedge$(6) $\Rightarrow$ (7).}} Let $f:M\ra M'$ denote the given surjective map. We first show that we can assume $\mathcal{I}_{-,\cchh}^{M'}=0$ and $M'\in\CCCp$. Since $\kf(M')\simeq\kf(M'/\mathcal{I}_{-,\cchh}^{M'})$, we can replace $M'$ by $M'/\mathcal{I}_{-,\cchh}^{M'}$ so that $M'$ now satisfies $\mathcal{I}_{-,\cchh}^{M'}=0$. Let $y\in\kf(M')=(M')^{\Umd}$. There exists $g:\fr(V)\ra M'$ such that $M'':=\im(g)$ contains $y$. Since $\mathcal{I}_{-,\cchh}^{M'}=0$, $g$ factors through $\fr(V)/\mathcal{I}_{-,\cchh}^{\fr(V)}\in\CCCp$, and hence $M''\in\CCCp$. Then we can replace $f$ by $f|_{f^{-1}(M'')}:f^{-1}(M'')\rightarrow M''$.

From now on, we do assume $\mathcal{I}_{-,\cchh}^{M'}=0$ and $M'\in\CCCp$. For each $i\geqslant 0$, define $M'_i:=\ker(\hp^i:M'\ra M')$. Define also $M'_{\infty}:=\bigcup_{i\geqslant 0}^{\infty} M'_i$. Let $z\in \kf(M')=(M')^{\Umd}$. By applying (5) to the composition $M\xrightarrow{f} M'\ra M'/M'_{\infty}$, we obtain $x_0\in \kf(M)$ such that $z-\kf(f)(x_0)\in (M_k')^{\Umd}$ for some $k$. By applying (6) to the composition $f^{-1}(M_k')\xrightarrow{f|_{f^{-1}(M_k')}} M_k'\ra M_k'/M_{k-1}'$, we obtain $x_1\in \kf(M)$ such that $z-\kf(f)(x_0+x_1)\in (M_{k-1}')^{\Umd}$. Continuing, we get $x_2,\ldots, x_k$ such that $z-\kf(f)(x_0+\cdots+x_k)\in M_0'=\{0\}$, i.e. $z=\kf(f)(x_0+\cdots+x_k)$.

\bigskip
\underline{\textit{(1)$\wedge$(7) $\Rightarrow$ (2).}} Let $M\in\CCC$. There exists an exact sequence $L_1\ra L_2\ra M\ra 0$ where $L_1$ and $L_2$ are (possibly infinite) direct sums of objects of $\CCCfr$. We have the commutative diagram
\begin{equation}\nonumber
\begin{tikzpicture}
\tikzmath{\x1 = 5; \x2 = 1.8;}
\node (A) at (0,0) {$\utd\ot_{\zgd}\kf(L_1)$} ;
\node (B) at (\x1,0) {$\utd\ot_{\zgd}\kf(L_2)$} ;
\node (C) at (2*\x1,0) {$\utd\ot_{\zgd}\kf(M)$} ;
\node (D) at (2.6*\x1,0) {$0$} ;

\node (E) at (0,-\x2) {$\dkf(L_1)$} ;
\node (F) at (\x1,-\x2) {$\dkf(L_2)$} ;
\node (G) at (2*\x1,-\x2) {$\dkf(M)$} ;
\node (H) at (2.6*\x1,-\x2) {$0$} ;

\path[->] (A) edge (B);
\path[->] (B) edge (C);
\path[->] (C) edge (D);
\path[->] (E) edge (F);
\path[->] (F) edge (G);
\path[->] (G) edge (H);

\path[->, font=\tiny] (A) edge node[left]{$\tt_{L_1}$} (E);
\path[->, font=\tiny] (B) edge node[left]{$\tt_{L_2}$} (F);
\path[->, font=\tiny] (C) edge node[left]{$\tt_{M}$} (G);
\end{tikzpicture}.
\end{equation}
By (1), $\tt_{L_1}$ and $\tt_{L_2}$ are bijective. By (7), the sequence in the top row is exact at $\utd\ot_{\zgd}\kf(M)$. Since $\dkf$ is right exact, the sequence in the bottom row is exact. It follows that $\tt_M$ is bijective by the five lemma.
%%%%%%%%%%%%%%%%%%%%%%%%%%%%%%%%%%%
%%%%%%%%%%%%%%%%%%%%%%%%%%%%%%%%%%%
%%%%%%%%%%%%%%%%%%%%%%%%%%%%%%%%%%%
%%%%%%%%%%%%%%%%%%%%%%%%%%%%%%%%%%%
%%%%%%%%%%%%%%%%%%%%%%%%%%%%%%%%%%%
%%%%%%%%%%%%%%%%%%%%%%%%%%%%%%%%%%%
\section{Rietsch mirror: Lam-Templier's vs ours} \label{C}
We show that our version of the Rietsch mirror (Section \ref{Rietschmirror}) is equivalent to the one defined by Lam and Templier \cite[Section 6.4]{LT}. As explained in \textit{loc. cit.}, the latter is equivalent to Rietsch's original one \cite[Section 4]{Rietsch}.

First, we interchange the roles of $G$ and $\Gd$. In other words, $G$ (resp. $\Gd$) is now denoted by $\Gd$ (resp. $G$). We also relabel the related objects (e.g. subgroups, roots, etc) correspondingly. In particular, we now have
\[ \Xd \simeq U_-(\wpd)^{-1}Z(L)U_-\cap B.\]

In \cite[Section 6.4]{LT}, Lam and Templier defined the Rietsch mirror to be
\[ X_{LT}:= UZ(L)\dot{(w_Pw_0)}_{LT}U\cap B_-\]
where $\dot{(w_Pw_0)}_{LT}\in N(T)$ (which is $\dot{w}_P$ in \textit{loc. cit.}) is a representative of $\wp$, defined using
\[ \dot{s}_{\a_i,LT}:= x_i(-1)y_i(1)x_i(-1)\]
(see \cite[Section 2.2]{LT}) instead of $\dot{s}_{\a_i}= y_i(-1)x_i(1)y_i(-1)$ which is what we are using. The analogues of $W$, $\pi$ and $p$ are defined in \cite[Section 6.4]{LT}; the analogue of $\vol$ is defined in \cite[Section 6.6]{LT}; and the analogue of the $\gm$-action is defined in \cite[Section 6.21]{LT}.

Let $g\mapsto g^T$ be the unique involutive antiautomorphism of $G$ satisfying
\[ t^T=t,~t\in T\quad\text{ and }\quad x_i(a)^T=y_i(a),~a\in \AA^1.\]
Define $g^{-T}:=(g^{-1})^T=(g^T)^{-1}$.

\begin{lemma}\label{LTvsours}
The map $g\mapsto g^T$ induces an isomorphism
\[ \Xd\simeq X_{LT}.\]
It preserves the additional structures including $W$, $\pi$, $p$, $\vol$ and the $\gm$-action. 
\end{lemma}
\begin{proof}
To prove the first part, it suffices to show $\dot{(w_Pw_0)}_{LT}=(\wpd)^{-T}$. Observe that $\dot{s}_{\a_i,LT}=(\dot{s}_{\a_i})^T$ for any $1\leqslant i\leqslant r$. It follows that $\dot{(w_Pw_0)}_{LT}=\dot{(w_0w_P)}^T$. Since $\ell(w_0w_P)+\ell(w_P)=\ell(w_0)$, we have $\dot{(w_0w_P)}\dot{w}_P=\dot{w}_0$, and hence $\dot{(w_0w_P)}=\dot{w}_0\dot{w}_P^{-1}$. But $\dot{w}_0=\dot{w}_0^{-1}$ because $G$ (which was $\Gd$ originally) is of adjoint type. It follows that $\dot{(w_0w_P)} = \dot{w}_0^{-1}\dot{w}_P^{-1}=(\wpd)^{-1}$, and hence $\dot{(w_Pw_0)}_{LT}=\dot{(w_0w_P)}^T= (\wpd)^{-T}$ as desired.

It remains to show that the above isomorphism preserves the additional structures. Let us handle the fiberwise volume forms only. The others are straightforward and left to the reader. Recall $\vol$ is defined to be the pullback of a volume form on $\UP$ via the isomorphism $\nu:\Xd\xrightarrow{\sim} \UP\times Z(L)$. See Definition \ref{Fiberwisevolumeformdef}. The analogue is the isomorphism $\nu_{LT}:X_{LT} \xrightarrow{\sim}  \overset{\circ}{G/P}\times Z(L)$ defined by 
\[ \nu_{LT}(u_1t\dot{(w_Pw_0)}_{LT}u_2):= (u_2^{-1}\dot{w}_0P,t)\quad \text{ assuming }~u_1\in U\cap\dot{w}_PU\dot{w}_P^{-1},\]
where $\overset{\circ}{G/P}$ is the isomorphic image of the open Richardson variety $\mathcal{R}^{w_0}_{w_P}:=B_-\dot{w}_PB/B\cap B\dot{w}_0B/B\subseteq G/B$ under the projection $G/B\ra G/P$. See \cite[Section 6.6]{LT}. Consider the diagram
\begin{equation}\nonumber
\begin{tikzpicture}
\tikzmath{\x1 = 7; \x2 = 2;}
\node (A) at (0,0) {$\Xd$} ;
\node (B) at (\x1,0) {$X_{LT}$} ;
\node (C) at (0,-\x2) {$\UP\times Z(L)$} ;
\node (D) at (\x1,-\x2) {$\overset{\circ}{G/P}\times Z(L)$} ;
\node (E) at (0,-2*\x2) {$G/P_-\times Z(L)$};
\node (F) at (\x1,-2*\x2) {$G/P\times Z(L)$};

\path[->, font=\small] (A) edge node[above]{$g\mapsto g^T$} (B);
\path[->, font=\small] (A) edge node[left]{$\nu$} (C);
\path[->, font=\small] (B) edge node[right]{$\nu_{LT}$} (D);
\path[->, dashed, font=\small]
(C) edge node[above]{$\simeq$} (D);
\path[right hook->] (C) edge node[left]{} (E);
\path[right hook->] (D) edge node[left]{} (F);
\path[->, font=\small] (E) edge node[above]{$(gP_-,t)\mapsto (g^{-T}P,t)$} (F);
\end{tikzpicture}.
\end{equation}
It is not difficult to show that the bottom arrow is an isomorphism. By \cite[Lemma 6.7]{LT}, this diagram (without the middle arrow) is commutative. This induces the middle arrow which is also an isomorphism such that the whole diagram is commutative.  Since the volume forms from both definitions are uniquely determined (up to scalar) by the corresponding boundary divisors, we are done.
\end{proof}

%%%%%%%%%%%%%%%%%%%%%%%%%%%%%%%%%%%
%%%%%%%%%%%%%%%%%%%%%%%%%%%%%%%%%%%
%%%%%%%%%%%%%%%%%%%%%%%%%%%%%%%%%%%
%%%%%%%%%%%%%%%%%%%%%%%%%%%%%%%%%%%
%%%%%%%%%%%%%%%%%%%%%%%%%%%%%%%%%%%
%%%%%%%%%%%%%%%%%%%%%%%%%%%%%%%%%%%


\begin{thebibliography}{99} 
\bibitem{BCFKS} V. Batyrev, I. Ciocan-Fontanine, B. Kim and D. van Straten, \textit{Mirror symmetry and toric degenerations of partial flag manifolds}, Acta Math. \textbf{184} (2000), no. 1, 1-39.

\bibitem{Satakeluminy} P. Baumann and S. Riche, \textit{Notes on the geometric Satake equivalence}, in: Relative aspects in representation theory, Langlands functoriality and automorphic forms, Lecture notes in Mathematics \textbf{2221}, Springer, Cham, 2018, 1-134.

\bibitem{Acta} P. Baumann, J. Kamnitzer and A. Knuston, \textit{The Mirkovi\'c-Vilonen basis and Duistermaat-Heckman measures. With an appendix by A. Dranowski, J. Kamnitzer and C. Morton-Ferguson}, Acta Math. \textbf{227} (2021), no. 1, 1-101.

\bibitem{BL} A. Beauville and Y. Laszlo, \textit{Un lemme de descente}, C. R. Acad. Sci. Paris S\'er. I Math. \textbf{320} (1995), no. 3, 335-340.

\bibitem{BK} A. Berenstein and D. Kazhdan, \textit{Geometric and unipotent crystals. II. From unipotent bicrystals to crystal bases}, in: Quantum groups, Contemp. Math. \textbf{433}, American Mathematical Society, Providence, RI, 2007, 13-88.

\bibitem{BF} R. Bezrukavnikov and M. Finkelberg, \textit{Equivariant Satake category and Kostant-Whittaker reduction}, Mosc. Math. J. \textbf{8} (2008), no. 1, 39-72.

\bibitem{BFM} R. Bezrukavnikov, M. Finkelberg and I. Mirkovi\'c, \textit{Equivariant homology and $K$-theory of affine Grassmannians and Toda's lattices}, Compos. Math. \textbf{141} (2005), no. 3, 746-768. 

\bibitem{Shift1} A. Braverman, D. Maulik and A. Okounkov, \textit{Quantum cohomology of the Springer resolution}, Adv. Math. \textbf{227} (2011), no. 1, 421-458.

\bibitem{me1} C.H. Chow, \textit{Peterson-Lam-Shimozono's theorem is an affine analogue of quantum Chevalley formula}, Forum Math. Sigma \textbf{13} (2025), e62.

\bibitem{me2} C.H. Chow, \textit{On D. Peterson's presentation of quantum cohomology of $G/P$}, Preprint (2022), available at \href{https://arxiv.org/abs/2210.17382}{https://arxiv.org/abs/2210.17382}.

\bibitem{toricstack} T. Coates, A. Corti, H. Iritani and H.-H. Tseng, \textit{Hodge-theoretic mirror symmetry for toric stacks}, J. Differential Geom. \textbf{114} (2020), no. 1, 41-115.

\bibitem{Mirror} D. Cox and S. Katz, \textit{Mirror symmetry and algebraic geometry}, in: Mathematical surveys and monographs \textbf{68}, American Mathematical Society, Providence, RI, 1999.

\bibitem{FG} E. Frenkel and B. Gross, \textit{A rigid irregular connection on the projective line}, Ann. of Math. (2) \textbf{170} (2009), no. 3, 1469-1512.

\bibitem{Ginzburg} V. Ginzburg, \textit{Perverse sheaves on a loop group and Langlands' duality}, Preprint (1995), available at \href{https://arxiv.org/abs/alg-geom/9511007}{https://arxiv.org/abs/alg-geom/9511007}.

\bibitem{Givental} A. Givental, \textit{Stationary phase integrals, quantum Toda lattices, flag manifolds and the mirror conjecture}, in: Topics in singularity theory: V. I. Arnold's 60th anniversary collection, American Mathematical Society Translations: Series 2, \textbf{180}, American Mathematical Society, Providence, RI, 1997, 103-115.

\bibitem{Giventalmirrortheorem} A. Givental, \textit{A mirror theorem for toric complete intersections}, in: Topological field theory, primitive forms and related topics (Kyoto, 1996), Progr. Math., \textbf{160}, Birkh\"auser Boston, Boston, MA, 1998, 141-175.

\bibitem{GMP} E. Gonz\'alez, C.Y. Mak and D. Pomerleano, \textit{Affine nil-Hecke algebras and quantum cohomology}, Adv. Math. \textbf{415} (2023).

\bibitem{HNY} J. Heinloth, B.-C. Ng\^o and Z. Yun, \textit{Kloosterman sheaves for reductive groups}, Ann. of Math. (2) \textbf{177} (2013), no. 1, 241-310.

\bibitem{Japan} R. Hotta, K. Takeuchi and T. Tanisaki, $D$-modules, perverse sheaves, and representation theory, Progress in Mathematics, \textbf{236}, Birkh\"auser, Boston, MA, 2008.

\bibitem{Hu} X. Hu, \textit{Mirror symmetry for quadric hypersurfaces}, Preprint (2022), available at \href{https://arxiv.org/abs/2204.07858}{https://arxiv.org/abs/2204.07858}.

\bibitem{Iritanibig} H. Iritani, \textit{A mirror construction for the big equivariant quantum cohomology of toric manifolds}, Math. Ann. \textbf{368} (2017), no. 1-2, 279-316.

\bibitem{Shift2} H. Iritani, \textit{Shift operators and toric mirror theorem}, Geom. Topol. \textbf{21} (2017), no. 1, 315-343.

\bibitem{JK} D. Joe and B. Kim, \textit{Equivariant mirrors and the Virasoro conjecture for flag manifolds}, Int. Math. Res. Not. IMRN \textbf{15} (2003), 859-882.

\bibitem{KKP} L. Katzarkov, M. Kontsevich and T. Pantev, \textit{Bogomolov-Tian-Todorov theorems for Landau-Ginzburg models},  J. Differ. Geom. \textbf{105} (2017), no. 1, 55-117.

\bibitem{Kim} B. Kim, \textit{Quantum cohomology of flag manifolds $G/B$ and quantum Toda lattices}, Ann. of Math. (2) \textbf{149} (1999), no. 1, 129-148.

\bibitem{KLS} A. Knutson, T. Lam and D.E. Speyer, \textit{Projections of Richardson varieties}, J. Reine Angew. Math. \textbf{687} (2014), 133-157.

\bibitem{HMS} M. Kontsevich, \textit{Homological algebra of mirror symmetry}, in: Proceedings of the International Congress of Mathematicians (Z\"urich, 1994), \textbf{1-2}, Basel, Birkh\"auser, 1995, 120-139.

\bibitem{KW} B. Kostant, \textit{On Whittaker vectors and representation theory}, Invent. Math. \textbf{48} (1978), 101-184.

\bibitem{KostantToda} B. Kostant, \textit{Quantization and representation theory}, in: Representation theory of Lie groups, Cambridge University, Cambridge, 1979, 287-316.

\bibitem{LGC} T. Lam, \textit{Whittaker functions, geometric crystals, and quantum Schubert calculus}, in: Schubert calculus - Osaka 2012, Adv. Stud. Pure Math., \textbf{71}, Math. Soc. Japan, 2016, 211-250.

\bibitem{LamRietsch} T. Lam and K. Rietsch, \textit{Total positivity, Schubert positivity, and geometric Satake}, J. Algebra \textbf{460} (2016), 284-319.

\bibitem{LS} T. Lam and M. Shimozono, \textit{Quantum cohomology of $G/P$ and homology of affine Grassmannian},  Acta Math. \textbf{204} (2010), no. 1, 49-90.

\bibitem{LT} T. Lam and N. Templier, \textit{The mirror conjecture for minuscule flag varieties}, Duke Math. J. \textbf{173} (2024), no. 1, 75-175.

\bibitem{RietschPlucker} C. Li, K. Rietsch, M. Yang and C. Zhang, \textit{A Pl\"ucker coordinate mirror for partial flag varieties and quantum Schubert calculus}, Preprint (2024), available at \href{https://arxiv.org/abs/2401.15640}{https://arxiv.org/abs/2401.15640}.

\bibitem{Shift3} T. Liebenschutz-Jones, \textit{Shift operators and connections on equivariant symplectic cohomology}, Preprint (2021), available at \href{https://arxiv.org/abs/2104.01891}{https://arxiv.org/abs/2104.01891}.

\bibitem{Manin} Y. Manin, Frobenius manifolds, quantum cohomology and moduli spaces, Amer. Math. Soc. Colloq. Publ. \textbf{47}, American Mathematical Society, Providence, RI, 2002.

\bibitem{MR} R. Marsh and K. Rietsch, \textit{The B-model connection and mirror symmetry for Grassmannians}, Adv. Math. \textbf{300} (2020).

\bibitem{Shift4} D. Maulik and A. Okounkov, \textit{Quantum groups and quantum cohomology}, Ast\'erisque \textbf{408}, Soci\'et\'e Math\'ematique de France, Paris, 2019.

\bibitem{MV} I. Mirkovi\'c and K. Vilonen, \textit{Geometric Langlands duality and representations of algebraic groups over commutative rings}, Ann. of Math. (2) \textbf{166} (2007), no. 1, 95-143.

\bibitem{Shift5} A. Okounkov and R. Pandharipande, \textit{The quantum differential equation of the Hilbert scheme of points in the plane}, Transform. Groups \textbf{15} (2010), no. 4, 965-982.

\bibitem{PR} C. Pech and K. Rietsch, \textit{A comparison of Landau-Ginzburg models for odd dimensional quadrics}, Bull. Inst. Math. Acad. Sin. (N.S.) \textbf{13} (2018), no. 3, 249-291.

\bibitem{PRW} C. Pech, K. Rietsch and L. Williams, \textit{On Landau-Ginzburg models for quadrics and flat sections of Dubrovin connections}, Adv. Math. \textbf{300} (2016), 275-319.

\bibitem{Peter} D. Peterson, \textit{Quantum cohomology of $G/P$}, Lecture course at MIT, Spring 1997. Notes typeset by Arun Ram and Gil Azaria. Available at \href{http://math.soimeme.org/~arunram/Resources/QuantumCohomologyOfGPL1-5.html}{Lectures 1-5}; \href{http://math.soimeme.org/~arunram/Resources/QuantumCohomologyOfGPL6-10.html}{6-10}; \href{http://math.soimeme.org/~arunram/Resources/QuantumCohomologyOfGPL11-15.html}{11-15}; \href{http://math.soimeme.org/~arunram/Resources/QuantumCohomologyOfGPL16-18.html}{16-18}.

\bibitem{RietschJAMS} K. Rietsch, \textit{Totally positive Toeplitz matrices and quantum cohomology of partial flag manifolds}, J. Amer. Math. Soc. \textbf{16} (2003), no. 2, 363-392.

\bibitem{Rietsch} K. Rietsch, \textit{A mirror symmetric construction of $qH_T^*(G/P)_{(q)}$}, Adv. Math. \textbf{217} (2008), no. 6, 2401-2442.

\bibitem{Sabbah} C. Sabbah, \textit{On a twisted de Rham complex}, Tohoku Math. J. (2) \textbf{51} (1999), no. 1, 125-140.

\bibitem{Sabbahbook} C. Sabbah, Isomonodromic deformations and Frobenius manifolds: an introduction, Universitext, Springer, London, 2008.

\bibitem{Savelyev1} Y. Savelyev, \textit{Quantum characteristic classes and the Hofer metric}, Geom. Topol. \textbf{12} (2008), no. 4, 2277-2326.

\bibitem{Savelyev2} Y. Savelyev, \textit{Virtual Morse theory on $\O Ham(M,\w)$},  J. Differ. Geom. \textbf{84} (2010), no. 2, 409-425.

\bibitem{Savelyev3} Y. Savelyev, \textit{Bott periodicity and stable quantum classes}, Selecta Math. (N.S.) \textbf{19} (2013), no. 2, 439-460.

\bibitem{Seidel} P. Seidel, \textit{$\pi_1$ of symplectic automorphism groups and invertibles in quantum homology rings}, Geom. Funct. Anal. \textbf{7} (1997), no. 6, 1046-1095.

\bibitem{T1} C. Teleman, \textit{Gauge theory and mirror symmetry}, in: Proceedings of the International Congress of Mathematicians (Seoul, 2014), \textbf{2}, Kyung Moon Sa, Seoul, 2014, 1309-1332.

\bibitem{YZ} Z. Yun and X. Zhu, \textit{Integral homology of loop groups via Langlands dual groups}, Represent. Theory \textbf{15} (2011), 347-369.

\bibitem{Zhu} X. Zhu, \textit{Frenkel-Gross' irregular connection and Heinloth-Ng\^o-Yun's are the same}, Selecta Math. (N.S.) \textbf{23} (2017), no. 1, 245-274.
\end{thebibliography}
\end{document}